\documentclass{article}
\usepackage[utf8]{inputenc}

\usepackage{hyperref}

\usepackage[document]{ragged2e}

\usepackage{amsmath}
\usepackage{amssymb}
\usepackage{amsthm}
\usepackage{mathrsfs}
\usepackage{mathtools}
\usepackage{comment}
\usepackage{enumerate}
\usepackage{xcolor}
\usepackage{slashed}

\usepackage[english]{babel}

\usepackage{cite}
\usepackage{url}
\usepackage{hyperref}

\usepackage{geometry}
\numberwithin{equation}{section}

\usepackage{amscd}
\usepackage{tikz-cd}
\usepackage[title]{appendix}

\makeatletter
\newtheorem*{rep@theorem}{\rep@title}
\newcommand{\newreptheorem}[2]{%
\newenvironment{rep#1}[1]{%
 \def\rep@title{#2 \ref{##1}}%
 \begin{rep@theorem}}%
 {\end{rep@theorem}}}
\makeatother

\newtheorem{theorem}{Theorem}[section]
\newreptheorem{theorem}{Theorem}

\newtheorem{proposition}[theorem]{Proposition}
\newtheorem{lemma}[theorem]{Lemma}
\newtheorem{corollary}[theorem]{Corollary}

\newtheorem{remark}[theorem]{Remark}
\newtheorem{example}[theorem]{Example}
\theoremstyle{definition}
\newtheorem{definition}[theorem]{Definition}

\usepackage{subfiles} 

\usepackage[nottoc]{tocbibind}

\title{Holomorphic Witten instanton complexes 
on stratified pseudomanifolds with K\"ahler wedge metrics.}
\author{Gayana Jayasinghe}
\date{}

\begin{document}

\maketitle

\begin{abstract}
\justifying
We construct Witten instanton complexes for K\"ahler Hamiltonian Morse functions on stratified pseudomanifolds with wedge K\"ahler metrics satisfying a local conformally totally geodesic condition. We use this to extend Witten's holomorphic Morse inequalities for the $L^2$ cohomology of Dolbeault complexes, deriving versions for Poincar\'e Hodge polynomials, the spin Dirac and signature complexes for which we prove rigidity results, in particular establishing the rigidity of $L^2$ de Rham cohomology for these circle actions. We study formulas for Rarita Schwinger operators, generalize formulas studied by Witten and Gibbons-Hawking for the equivariant signature and extend formulas used to compute NUT charges of gravitational instantons. We discuss conjectural inequalities extending known Lefschetz-Riemann-Roch formulas for other cohomology theories including those of Baum-Fulton-Quart.
This article contains the first extension of Witten's holomorphic Morse inequalities and instanton complexes to singular spaces.

\end{abstract}

\tableofcontents

\justify

\section{Introduction}

Atiyah and Bott generalized the classic fixed point theorem of Lefschetz to general elliptic complexes in \cite{AtiyahBott1}, the first of a large class of localization theorems. In a series of papers including \cite{witten1981dynamical,witten1982supersymmetry,witten1982constraints,witten1983fermion,witten1984holomorphic} Witten investigated supersymmetry breaking, giving an original and fascinating take on Morse theory which he formulated extensions of for other elliptic complexes including twisted Dolbeault complexes in the presence of K\"ahler Hamiltonian Morse functions. We refer the reader to Witten's essay ``Lessons from Raoul Bott" in \cite{Yaubookfoundersindex}, and Bott's classic article \cite{bott1988morse} for history and some exposition.
The perspectives, results and techniques that came out of the work above mentioned have now permeated mathematics and physics and it is important to have a systematic understanding of how these ideas extend to singular spaces.

In \cite{jayasinghe2023l2} we extended the Atiyah-Bott Lefschetz fixed point theorem for $L^2$ cohomology of twisted Dolbeault complexes satisfying the \textit{Witt condition} to stratified pseuodmanifolds with wedge metrics and wedge complex structures. We also extended the Witten instanton complex for the de Rham operator for such spaces, proving $L^2$ de Rham Morse inequalities.
In this article we construct holomorphic Witten instanton complexes in the presence of K\"ahler Hamiltonian Morse functions, extending Witten's equivariant holomorphic Morse inequalities \cite{witten1984holomorphic} (as opposed to Demailly's asymptotic holomorphic Morse inequalities \cite{demailly1991holomorphic} which are also sometimes referred to simply as the holomorphic Morse inequalities) for twisted Dolbeault complexes satisfying the Witt condition on stratified pseudomanifolds with K\"ahler wedge metrics. We do this for isometric K\"ahler Hamiltonian circle actions that preserve the strata, and have isolated fixed points. There is a functorial equivalence between stratified spaces $\widehat{X}$ and manifolds with corners with iterated fibration structures $X$ which resolve the singularities, and we work in the latter category for the most part. 
We give an outline of the main results in subsection \ref{subsection_overview_introduction}, referring to the text for more details. We give a brief summary of the construction of the instanton complexes in subsection \ref{subsection_construction_summary_intro}, and delve into the background and related results in subsection \ref{subsection_background_related_results}.

\subsection{Overview of results and techniques.}
\label{subsection_overview_introduction}

In this article we study stratified spaces with wedge K\"ahler metrics (roughly these are ``asymptotically iterated conic metrics") which admit K\"ahler Hamiltonian Morse functions. This includes many algebraic varieties including so called ``cusp" curves (see Subsection \ref{subsection_LCTGMetrics}), singular algebraic surfaces appearing in the boundary of the moduli space of K3 surfaces and conifolds, some of which we study in Section \ref{Subsection_examples}. 
Moreover we assume that near critical points the metric is \textit{locally conformally totally geodesic}, which roughly means that the metric is conformal to a conic metric (product type metric) near singular strata, up to a prescribed asymptotic behaviour. Many toric varieties including the examples we study in this article admit metrics which satisfy this condition.

Given a twisted Dolbeault complex $\mathcal{P}_{\alpha}(X)=(L^2\Omega^{0,\cdot}(X;E), \mathcal{D}_{\alpha}(P), P=\overline{\partial}_E)$ where $E$ is a Hermitian bundle on a resolution of a stratified pseudomanifold $X$, if there is an isometric K\"ahler circle action generated by a Hamiltonian vector field $V$ that lifts to fiberwise linear action on $E$, there is an induced action $\sqrt{-1}L_V$ on sections in the domain. Here $\alpha$ corresponds to various choices of domains, where for instance $\alpha=0,1/2,1$ correspond to the maximal, vertical Atiyah-Patodi-Singer (VAPS) and minimal domains, which we introduce in Subsection \ref{subsection_abstract_hilbert_complexes}.
We define equivariant subcomplexes $\mathcal{P}^{\mu}_{\alpha}(X)$ by restricting the operator to eigensections of $\sqrt{-1}L_V$ with eigenvalue $\mu$. We also define local complexes $\mathcal{P}_{\alpha,B}(U_a)$ and local equivariant complexes $\mathcal{P}^{\mu}_{\alpha,B}(U_a)$ in \textit{fundamental neighbourhoods} $U_a$ of critical points $a$ of the K\"ahler Hamiltonian $h$, where the choice of domain indicated by $B$ and the cohomology groups depend on the normal form of the Hamiltonian near the critical point $h$.
Roughly, it is the restricted complex of $\overline{\partial}_E$ (respectively $\overline{\partial}^*_E$) in neighbourhoods where the metric gradient of $h$ is attracting (expanding), and a suitable product domain if there are directions in which the gradient flow of $h$ is both attracting and expanding restricted to a product decomposition $U_a=U_{a,s} \times U_{a,u}$ near critical points.
We denote the geometric endomorphism corresponding to the circle action by $T_{\theta}$, for $\theta \in S^1$.

\begin{remark}[convention]
\label{Remark_convention_isolated_fixed_points}
In this article we restrict to actions of compact connected Lie groups $G$ for which the action by generic elements have only isolated fixed points on $\widehat{X}$, and we refer to these as \textbf{actions with isolated fixed points}. It is clear that for any group action, the identity map induced by the action of the identity element of the group fixes the entire space, but our results apply as long as generic elements have isolated fixed points. 
\end{remark}

If the equivariant complexes $\mathcal{P}^{\mu}_{\alpha}$ are Fredholm for every $\mu$, then we call the complex $\mathcal{P}_{\alpha}$ \textit{\textbf{transversally Fredholm}}, and we will restrict our study to such complexes. Given a stratified pseudomanifold $\widehat{X}$ with isolated conic singularities the complexes $\mathcal{P}_{\alpha}(X)$ are Fredholm and that local complexes $\mathcal{P}_{\alpha,B}(U_a)$ are transversally Fredholm, as is also the case for $\alpha=1/2$, $\alpha=1$ for general singularities 
and we will formalize this in Proposition \ref{Proposition_equivariant_Fredholm_complexes}, and discuss the discreteness of the spectrum there as well.
It can be verified explicitly that the transversally Fredholm condition holds for many other examples of complexes with various constant values of $\alpha \in [0, 1]$.

Given such a complex, we have an associated \textit{Witten deformed complex}, where the Hilbert spaces on the space are the same, but the Dolbeault operator $P$ is deformed by a conjugation of $e^{\varepsilon h}$ where $\varepsilon$ is a parameter, yielding the deformed operator $P_\varepsilon$, and in particular the domains of the operators are different in general, though there is a canonical isomorphism given by the deformation. Taking the adjoint $P^*_\varepsilon$, we can construct a deformed Dirac-type operator $D_\varepsilon=P_\varepsilon+P^*_\varepsilon$, the square of which we call the deformed Laplace-type operator and denote it by $\Delta_{\varepsilon}$. We denote the Witten deformed complex by $\mathcal{P}_{\alpha,\varepsilon}(X)=(L^2\Omega^{0,\cdot}(X;E), \mathcal{D}_{\alpha}(P_{\varepsilon}), P_{\varepsilon})$, and the equivariant complexes by $\mathcal{P}^{\mu}_{\alpha,\varepsilon}(X)=(L^2_{\mu}\Omega^{0,\cdot}(X;E), \mathcal{D}^{\mu}_{\alpha}(P_{\varepsilon}), P_{\varepsilon})$. We note that the case where $\varepsilon=0$ corresponds to the \textit{undeformed} complex. We denote the corresponding local complexes by $\mathcal{P}^{\mu}_{\alpha,\varepsilon,B}(U_a)$. In Subsection \ref{subsection_deformed_Hilbert_complexes} we will show that the deformed and undeformed complexes are isomorphic. The following is the main theorem of this article, a restatement of Theorem \ref{theorem_small_eig_complex}.

\begin{theorem}[holomorphic Witten instanton complex]
\label{theorem_small_eig_complex_intro}

Let $X$ be the resolution of a stratified pseudomanifold of dimension $2n$ with a K\"ahler wedge metric and a stratified K\"ahler Hamiltonian Morse function $h$ corresponding to an isometric $S^1$ action generated infinitesimally by a wedge vector field $V$, where the metric is locally conformally totally geodesic at fundamental neighbourhoods of critical points of $h$. Let $E$ be a Hermitian vector bundle on $X$, to which the action lifts as a bundle action, yielding a geometric endomorphism $T^{\mathcal{P}_{\alpha}(X)}_{\theta}=T_{\theta}$ on the Dolbeault complex $\mathcal{P}_{\alpha,\varepsilon}(X)$ introduced above, which we assume is Fredholm, and where local complexes induced at isolated critical points are transversally Fredholm. Let $\mu$ be an eigenvalue of $\sqrt{-1}L_V$.

For any integer $0 \leq q \leq n$, let $
\mathrm{F}_{\alpha, \varepsilon, \mu, q}^{[0, \kappa]} \subset L^2_{\mu}\Omega^{0,q}(X;E)$ denote the vector space generated by the eigenspaces of $\Delta_{\varepsilon}$ associated with eigenvalues in $[0, \kappa]$.
Denoting $\kappa=\varepsilon^{f^+}$ there exists some $f^+ <1/2$ 
and some $\varepsilon_0>0$ such that when $\varepsilon>\varepsilon_0$, $\mathrm{F}_{\alpha, \varepsilon, \mu, q}^{[0, \kappa]}$ has the same dimension as 
\begin{equation}
    \sum_{a \in Cr(h)} \dim \mathcal{H}^q(\mathcal{P}^{\mu}_{\alpha,B,\varepsilon}(U_{a})),
\end{equation}
and form a finite dimensional subcomplex of $\mathcal{P}^{\mu}_{\alpha,\varepsilon}(X)$ :
\begin{equation}
    \label{small_eigenvalue_complex_intro}
\left(\mathrm{F}_{\alpha,\varepsilon,\mu, q}^{[0, \kappa]}, P_{\varepsilon}\right): 0 \longrightarrow \mathrm{F}_{\alpha,\varepsilon,\mu, 0}^{[0, \kappa]} \stackrel{P_{\varepsilon}}{\longrightarrow} \mathrm{F}_{\alpha,\varepsilon,\mu, 1}^{[0, \kappa]} \stackrel{P_{\varepsilon}}{\longrightarrow} \cdots \stackrel{P_{\varepsilon}}{\longrightarrow} \mathrm{F}_{\alpha,\varepsilon,\mu, n}^{[0, \kappa]} \longrightarrow 0
\end{equation}
which is quasi-isomorphic to $\mathcal{P}^{\mu}_{\alpha,\varepsilon}(X)$.
\end{theorem}

In the smooth setting, the result is implicit in \cite{wu1998equivariant}, where the \textit{\textbf{``smaller"}} eigenvalues in $[0, \kappa]$ are shown to be small with $f^+=-1/2$ using results in \cite{bismut1991complex} (see Remark \ref{Remark_operator_estimates_smooth}). We will discuss this complex in more detail in Subsection \ref{subsection_construction_summary_intro}.

In this setting, we denote by $T_{s,\theta}$ the endomorphism (not necessarily realizable as a geometric endomorphism) on the cohomology groups, which acts on sections in the cohomology group by $\lambda=se^{i\theta}$ (observe that $T_{\theta}=T_{1,\theta}$) where $\lambda \in \mathbb{C}^*$ (see Remark \ref{remark_local_extention_c_star}). 
Then we notate the generating series for the dimensions of the cohomology groups of each equivariant sub complex as 
\begin{equation}
\label{equation_renormalized_trace_definition}
    {Tr} (T_{s,\theta}|_{\mathcal{H}^{q}(\mathcal{P}_{\alpha,B}(U_a))}):= \sum_{\mu} \mathrm{dim}(\mathcal{H}^{q}(\mathcal{P}^{\mu}_{\alpha,B}(U_a)) \lambda^{\mu} 
\end{equation}
which is a power series that converges for $s<1$ (see Remark \ref{remark_convergence_of_generating_series}).
The following is a restatement of Theorem \ref{Theorem_strong_Morse_Dolbeault_version_2}.

\begin{theorem}[Strong form of the holomorphic Morse inequalities]
\label{Theorem_strong_Morse_Dolbeault_version_2_intro}
Let $X$ be the resolution of a stratified pseudomanifold of dimension $2n$ with a K\"ahler wedge metric and a stratified K\"ahler Hamiltonian Morse function $h$ corresponding to an isometric $S^1$ action generated infinitesimally by a wedge vector field $V$. Let $E$ be a Hermitian vector bundle on $X$, to which the action lifts as a fiberwise linear action, yielding a geometric endomorphism $T^{\mathcal{P}_{\alpha}(X)}_{\theta}=T_{\theta}$ on the Dolbeault complex $\mathcal{P}_{\alpha}(X)=(L^2\Omega^{0,\cdot}(X;E), \mathcal{D}_{\alpha}(P), P=\overline{\partial}_E)$ which we assume is Fredholm, and where local complexes induced at isolated critical points are transversally Fredholm. Let $\mu$ be an eigenvalue of $\sqrt{-1}L_V$. 
Then we have 
\begin{equation}
\label{equation_primitive_Morse_inequality_equivariant_1}
    \Big( \sum_{a \in Crit(h)}  \sum_{q=0}^n b^q \mathrm{dim}(\mathcal{H}^{q}(\mathcal{P}^{\mu}_{\alpha,B}(U_a)) \Big) = \sum_{q=0}^n b^q \mathrm{dim}(\mathcal{H}^{q}(\mathcal{P}^{\mu}_{\alpha}(X))) + (1+b) \sum_{q=0}^{n-1} Q^{\mu}_q b^q
\end{equation}
where $Q^{\mu}_q$ are non-negative integers.
Equivalently we have an equality of power series
\begin{equation}
\label{equation_Morse_inequalities_intro_11}
    \Big( \sum_{a \in Crit(h)}  \sum_{q=0}^n b^q {Tr} (T_{s,\theta}|_{\mathcal{H}^{q}(\mathcal{P}_{\alpha,B}(U_a))}) \Big) = \sum_{q=0}^n b^q {Tr} (T_{s,\theta}|_{\mathcal{H}^{q}(\mathcal{P}_{\alpha}(X))}) + (1+b) \sum_{q=0}^{n-1} Q_q b^q 
\end{equation}
where the $Q_q$ are power series in the variable $\lambda=se^{i\theta}$
converging for $|\lambda|=s<1$. The coefficients of this series are non-negative integers.
\end{theorem}

The theorem is referred to as an inequality since the terms $Q_q^{\mu}$ are non-negative integers. 
We call the terms on the left hand side of equation \eqref{equation_Morse_inequalities_intro_11} corresponding to each critical point $a$ as the \textit{\textbf{Morse polynomial at the critical point}}, and the sum over all critical points of those terms as the \textit{\textbf{Morse polynomial}}, while the first set of terms on the right hand side is called the \textit{\textbf{Poincar\'e polynomial of the complex}}, and the second set of terms, the \textit{\textbf{error polynomial}}. Here the polynomials are in the variable $b$, with coefficients that are power series. The following is a restatement of Theorem \ref{Theorem_dual_inequalities}.

\begin{theorem}
\label{Theorem_dual_inequalities_intro_version}
[Dual equivariant holomorphic Morse inequalities]
In the same setting as Theorem \ref{Theorem_strong_Morse_Dolbeault_version_2_intro}, let $(\mathcal{P}_{\alpha})_{SD}(X)$ be the Serre dual complex
\begin{equation}
    (\mathcal{P}_{\alpha})_{SD}(X):=\mathcal{R}_{\alpha}(X)=(L^2\Omega^{n,n-\cdot}(X;E^*),\mathcal{D}_{\alpha}(\overline{\partial}^*_{E^* \otimes K}),\overline{\partial}^*_{E^* \otimes K})
\end{equation}
and let $T^{\mathcal{R}_{\alpha}(X)}_{\theta}$ be the geometric endomorphism induced on this complex corresponding to the K\"ahler action, which we will denote as $T_{\theta}$ with some abuse of notation.
Then we have that
\begin{equation}
    b^n \Big( \sum_{a \in Crit(h)}  \sum_{q=0}^n b^{-q} Tr(T_{s,-\theta}|_{\mathcal{H}^{q}(\mathcal{R}_{\alpha,B}(U_a))}) \Big) = b^n \sum_{q=0}^n b^{-q} Tr(T_{s,-\theta}|_{\mathcal{H}^{q}(\mathcal{R}_{\alpha}(X))}) + (1+b) \sum_{q=0}^{n-1} \widetilde{Q}_q b^q 
\end{equation}
where the $B$ subscript denotes a choice of domain for the local complex corresponding to the K\"ahler hamiltonian $-h$ (corresponding to $T_{-\theta}$)
where the $\widetilde{Q}_q$ are power series in the variable $\lambda=se^{i\theta}$ which converge for $s=|\lambda|>1$. The coefficients of this series are non-negative integers.
Moreover
\begin{equation}
\label{equation_Poincare_polynomial_same_for_dual}
    \sum_{q=0}^n b^q Tr  (T_{1,\theta}|_{\mathcal{H}^{q}(\mathcal{P}_{\alpha}(X))})=b^n \sum_{q=0}^n b^{-q} Tr  (T_{1,-\theta}|_{\mathcal{H}^{q}(\mathcal{R}_{\alpha}(X))}).
\end{equation}
\end{theorem}

As before, we refer to the term on the left hand side as the \textit{\textbf{dual Morse polynomials}}, and the terms on the right as the \textit{\textbf{dual Poincar\'e polynomial}} and the \textit{\textbf{dual error polynomial}}.
These dual polynomials and the dual inequalities can equivalently be phrased in terms of the \textit{adjoint complexes} as well using the dualities in Proposition \ref{Proposition_duality_complex_conjugation_for_local_Lefschetz_numbers}.

Taken together these inequalities are very restrictive. Since equation \eqref{equation_Poincare_polynomial_same_for_dual} shows that the Poincar\'e polynomials are the same for the inequalities and the dual inequalities, and since the power series converge for opposite regimes of $s=|\lambda|$, one can conclude that only finitely many terms $c_{\mu}\lambda^{\mu}$ with common powers $\mu$ of $\lambda$ will appear in the Morse polynomial and the dual Morse polynomial, reducing it to what we call a \textit{\textbf{classical Morse polynomial}}.

\begin{definition}[Classical Morse polynomial]
\label{definition_classical_Morse_polynomial}
In the setting of Theorem \ref{Theorem_strong_Morse_Dolbeault_version_2_intro}, let $c_{\mu,q,1} (c_{\mu,q,2})$ be the coefficient of the power series with power $\lambda^{\mu}$ of the polynomial corresponding to the monomial $b^q$ of the Morse (dual Morse) polynomial. Let $c_{\mu,q}:=\min\{c_{\mu,q,1}, c_{\mu,q,2} \}$.
Then the classical Morse polynomial is 
\begin{equation}
    \sum_{q=0}^n \sum_{\mu} b^q c_{\mu,q} \lambda^{\mu}
\end{equation}
\end{definition}

The discussion immediately above this definition shows that for a given $q$, $c_{\mu,q} >0$ only for finitely many $\mu$, and that the \textit{\textbf{classical Morse inequality}}
\begin{equation}
    \sum_{q=0}^n \sum_{\mu} b^q c_{\mu,q} \lambda^{\mu}=\sum_{q=0}^n b^q Tr  (T_{1,\theta}|_{\mathcal{H}^{q}(\mathcal{P}_{\alpha}(X))}) + (1+b) \sum_{q=0}^{n-1} b^q d_{q,\mu} \lambda^{\mu}
\end{equation}
holds where $d_{q,\mu}$ are non-negative integers, positive only for finitely many $\mu$.
Then one can invoke the lacunary principle \cite[Theorem 3.39]{banyaga2004lectures} of Morse theory to get refined information on the cohomology groups.
We provide the example of a spinning sphere that Witten uses to illustrate this method in \cite{witten1984holomorphic}.

\begin{example}[Spinning the 2-sphere]
\label{Example_spinning_sphere_intro}
    Consider $\mathbb{CP}^1$ with projective coordinates $[Z_0:Z_1]$, equipped with the Hamiltonian $\mathbb{C}^*$ action $(\lambda)[Z_0:Z_1]=[\lambda Z_0: Z_1]$ where $\lambda=se^{i\theta} \in \mathbb{C}^*$. 
    This has two fixed points at $[0:1]$ and $[1:0]$.

    Let us consider the holomorphic Morse inequalities for the trivial bundle in the form of a generating series. The local cohomology (converging for $|\lambda|<1$) at $[0:1]$, is simply the $L^2$ bounded holomorphic functions in a neighbourhood of the fixed point and has Schauder basis $\{1,z,z^2,...\}$ where $z=Z_0/Z_1$.    
    The trace of the geometric endomorphism $T_{s,\theta}$ on the local cohomology group yields $\sum_{k=0}^{\infty} \lambda^k$  (converging for $|\lambda|<1$). Similarly the contribution $\sum_{k=1}^{\infty} \lambda^k$ is the trace of the geometric endomorphism over the local cohomology at $[1:0]$, which has a Schauder basis given by the anti-holomorphic one forms $\{\overline{y}^k\overline{dy}\}_{k \in \mathbb{N}}$ where $y=1/z$.   
    The equivariant holomorphic Morse inequalities for the trivial bundle are encoded by 
    \begin{equation}
    \label{equation_timbuktoo_1}
        \sum_{k=0}^{\infty} \lambda^k + b \sum_{k=1}^{\infty} \lambda^k=1 + (1+b) \sum_{k=1}^{\infty} \lambda^k.
    \end{equation}
        
    Similarly we can show that the dual Morse inequalities are encoded by
    \begin{equation}
    \label{equation_timbuktoo_2}
        b \sum_{k=1}^{\infty} \lambda^{-k} + \sum_{k=0}^{\infty} \lambda^{-k}=1 + (1+b) \sum_{k=1}^{\infty} \lambda^{-k},
    \end{equation}
    (converging for $|\lambda|>1$) and each term can be computed using explicit descriptions of the local cohomology groups as earlier.

    We see that the classical Morse polynomial (see Definition \ref{definition_classical_Morse_polynomial}) is simply $1$, whence the lacunary principle shows that this corresponds to the Poincar\'e polynomial (since there are no $(1+b)$ factors).    
\end{example}

Observe that the left hand side of equation \eqref{equation_timbuktoo_1} is a Laurent expansion for the first expression in
\begin{equation}
    \frac{1}{(1-\lambda)}+b\frac{\lambda}{(1-\lambda)}, \hspace{5mm} b\frac{\lambda^{-1}}{(1-\lambda^{-1})}+\frac{1}{(1-\lambda^{-1})}
\end{equation}
and substituting for $b$ by $1/b$ and multiplying by an overall factor of $-b=(-b)^n$ yields the second expression
which is the formal sum of the left hand side of equation \eqref{equation_timbuktoo_2}. 
The Morse polynomials and the dual Morse polynomials correspond to series expansions that converge for the different connected components of $\{ \lambda \in \mathbb{C} | |\lambda| \neq 1 \}$.
\textbf{We compute Morse polynomials for singular examples in Subsection \ref{Subsection_examples}.}

\begin{corollary}[Weak Morse inequalities]
\label{Corollary_weak_Morse_inequalities}
In the same setting as Theorem \ref{Theorem_strong_Morse_Dolbeault_version_2_intro}, we have that 
\begin{equation}
\label{equation_weakMorse_inequalities_1}
    \Big( \sum_{a \in Crit(h)}  \mathrm{dim}(\mathcal{H}^{q}(\mathcal{P}^{\mu}_{\alpha,B}(U_a)) \Big) \geq  \mathrm{dim}(\mathcal{H}^{q}(\mathcal{P}^{\mu}_{\alpha}(X))). 
\end{equation} 
\end{corollary}

This is immediate from the statement of Theorem \ref{Theorem_strong_Morse_Dolbeault_version_2_intro} since the coefficients of the power series of the error polynomial are positive, and can be equivalently stated in terms of power series as has been done for the smooth setting in \cite{mathai1997equivariant,wu1998equivariant}.
The Serre dual complex gives dual weak Morse inequalities.

When $b=-1$ the error polynomial vanishes and we recover the \textit{\textbf{holomorphic Lefschetz fixed point theorem}} that we proved in \cite{jayasinghe2023l2} when the cohomology of the complex matches that of the VAPS domain in the case of non-isolated singularities.
The direct sum of the cohomology groups of the Dolbeault complexes with minimal domain twisted by $E=\Lambda^p (^wT^* X^{1,0})$ for all $p=0,..,n$ is isomorphic to the de Rham cohomology on spaces satisfying the \textit{Witt condition}.

We proved the holomorphic Lefschetz fixed point theorem in \cite{jayasinghe2023l2} primarily for the VAPS domain for $\overline{\partial}_E$, which corresponds to $\alpha=1/2$ in this article, and observed that it holds for the cases of $\alpha=0$ (the maximal domain), and $\alpha=1$ (the minimal domain) for the case of isolated singularities together with considerations in \cite{Bei_2011_Perversities}.
Building on this we computed other equivariant invariants, including equivariant Hirzebruch $\chi_y$, signature invariants, where we restricted to the case when the cohomology of the VAPS domain equalled that of the minimal domain for the $\overline{\partial}_E$ domain. Here we study the minimal domain and certain other domains directly, and explore more general invariants in Subsections \ref{subsection_Poincare_Hodge_polynomial} and \ref{subsubsection_spin_Morse_inequalities}.

\begin{remark}
\label{remark_density_of_circle_actions}
The results that we prove extend naturally to actions of compact connected Lie groups, as in the proof of the theorem on page 22 of \cite{atiyah1970spin}, and the proof of Theorem 2.2 of  \cite{wu1996equivariant}. Therefore we study circle actions, for the most part. 

The two inequalities in Theorem \ref{Theorem_strong_Morse_Dolbeault_version_2_intro} and Theorem \ref{Theorem_dual_inequalities_intro_version} correspond to the decomposition of the Lie algebra of $S^1$, $\mathbb{R}$, into the two cones $\mathbb{R}^+ \cup \mathbb{R}^-$. 
For general Lie group actions we can obtain more restrictive inequalities using the inequalities corresponding to each cone of the Lie algebra following the approach in \cite{wu1996equivariant}, and we will study this in a follow up article. Related ideas are already used in the proof of Theorem \ref{Theorem_harmonic_spinors_vanish_singular}.
\end{remark}

In \cite{jayasinghe2023l2} we showed that under certain assumptions, the Lefschetz-Riemann-Roch formulas for Hilbert complexes that we proved correspond to those of Baum-Fulton-Quart \cite{baum1979lefschetz,Baumformula81}. 
Theorem \ref{Theorem_strong_Morse_Dolbeault_version_2_intro} gives the corresponding equivariant holomorphic Morse inequalities. It is natural to formulate holomorphic Morse inequalities in singular cohomology, and equivariant complexes studied more generally in \cite{baum1979lefschetz}, especially in light of more algebraic proofs in the smooth setting as in \cite{wu2003instanton}. We explore the conjectural holomorphic Morse inequalities 
in Subsection \ref{Subsection_examples}, along with a discussion of the smooth case and singular examples. We discussed holomorphic Lefschetz numbers in $L^2$ cohomology in great detail and in many examples in \cite{jayasinghe2023l2}, and we discuss enough details to enable one to work out the holomorphic Morse inequalities in those and similar examples.
Hilbert complexes which have Todd classes that match those of the complexes studied in \cite{baum1979lefschetz} have been constructed in \cite{LottHilbertconplex,RuppenthalSerre2018} and we briefly discuss Morse inequalities for these as well.

In Subsection \ref{subsection_Poincare_Hodge_polynomial}, we study the equivariant Poincar\'e Hodge polynomials corresponding to the Dolbeault complexes we study, in particular proving a Lefschetz-Hodge theorem in Theorem \ref{Theorem_L2_Lefschetz_Hodge_index}, and studying dualities.

We then turn our attention to rigidity results, starting with the following theorem. We note that $L^2$ de Rham cohomology is isomorphic to middle perversity intersection cohomology on these spaces. The following is a restatement of Theorem \ref{Theorem_instanton_correspondence_deRham_Dolbeault_non_intro}.

\begin{theorem}
\label{Theorem_instanton_correspondence_deRham_Dolbeault}
Let $X$ be a resolved stratified pseudomanifold with a wedge K\"ahler metric and a K\"ahler Hamiltonian morse function $h$. Then the Witten deformed de Rham complex in Theorem 6.33 of \cite{jayasinghe2023l2} is isomorphic as a Hilbert complex to the direct sum over $p$ of the Witten deformed complexes given in Theorem \ref{theorem_small_eig_complex_intro} for the equivariant Dolbeault complexes $^p\mathcal{P}^{\mu=0}_{\alpha}=(L^2\Omega_{\mu=0}^{p,\cdot}(X;\mathbb{C}),\mathcal{D}^{\mu=0}_{\alpha}(P),P)$ where $P=\overline{\partial}_{\Lambda^{p,0}}$ for any $\alpha$. In particular the Witten deformed de Rham complex inherits the Hodge bi-grading.
\end{theorem}

Basically this shows that for any of the domains we work with, the $\mu=0$ equivariant subcomplex is isomorphic to the de Rham complex, which in particular shows that the K\"ahler group actions we study act trivially on $L^2$ de Rham cohomology. The phenomenon that the index of operators restricted to equivariant Hilbert complexes vanishes for all $\mu \neq 0$ is usually called \textit{\textbf{rigidity}} (see, e.g. \cite{bott1989rigidity}). This is not the case for singular cohomology, as seen by the examples in Subsection \ref{subsubsection_conjectural_inequalities}, including the case of the cusp curve, and the variety $Z^4-X^3Y=0$ in $\mathbb{CP}^3$.

We denote the Dolbeault complexes $^p\mathcal{P}_{\alpha}=(L^2\Omega^{p,\cdot}(X;E),\mathcal{D}_{\alpha}(P),P)$ where $P_E=\overline{\partial}_{\Lambda^{p,0}\otimes E}$.
Given a circle action of the type that we study in this article, we define the \textbf{\textit{equivariant Poincar\'e Hodge polynomial}} for the Dolbeault complex to be 
\begin{equation}
\label{equivariant_Poincare_Hodge_definition}
    \chi_{y,b}(\mathcal{P}_{\alpha},T_{\theta}):=\sum_{p=0}^n y^p \sum_{q=0}^n b^q TrT_{\theta}|_{\mathcal{H}^{q}(^p\mathcal{P}_{\alpha})}
\end{equation}
for global complexes as well as local complexes $^p\mathcal{P}_{\alpha}(U)=(L^2\Omega^{p,\cdot}(U;E),\mathcal{D}_{\alpha}(P_U),P_U)$, where we abuse notation to denote the Dolbeault complexes for all $p$ with fixed $\alpha$ in the left hand side of the above equality. The following is a restatement of Theorem \ref{theorem_rigidity_Poincare_Hodge_non_intro}.

\begin{theorem}[Rigidity of Poincar\'e Hodge polynomials]
\label{theorem_rigidity_Poincare_Hodge}
In the same setting as Theorem \ref{Theorem_instanton_correspondence_deRham_Dolbeault}, for $E=\mathbb{C}$, $\alpha=1$ the \textbf{global equivariant Poincar\'e Hodge polynomials are rigid}. 
Moreover we have that 
\begin{equation}
\label{equation_Rigidity_deRham_1}
    \chi_{y,b}(\mathcal{P}^{\mu=0}_{\alpha}(X),T_{\theta})=\chi_{y,b}(\mathcal{Q}(X)):=\sum_{p,q=0}^n y^p  b^q \dim {\mathcal{H}^{p,q}(\mathcal{Q}(X))}
\end{equation}
and for fundamental neighbourhoods $U_a$ of critical points,
\begin{equation}
\label{equation_Rigidity_deRham_2_intro}
    \chi_{y,b}(\mathcal{P}^{\mu=0}_{\alpha,B}(U_a),T_{\theta})=\chi_{y,b}(\mathcal{Q}_B(U_a),T_{\theta})):=\sum_{p,q=0}^n y^p b^q \dim {\mathcal{H}^{p,q}(\mathcal{Q}_B(U_a))}
\end{equation}
for any $\alpha$,
and these satisfy {\textit{\textbf{de Rham type Poincar\'e Hodge inequalities}}} 
\begin{equation}
\label{equation_Rigidity_deRham_3_intro}
    \sum_{a \in crit(h)} \chi_{y,b}(\mathcal{Q}_B(U_a),T_{\theta}))=\chi_{y,b}(\mathcal{Q}(X))+(1+b) \sum_{q=0}^{n-1} R_q b^q,
\end{equation}
where $\mathcal{Q}$ is the Witten deformed de Rham complex with the Hodge bi-grading (see Remark \ref{Remark_minimal_domain_equivalences}).
\end{theorem}

In particular, \textbf{the formula in equation \ref{equation_Rigidity_deRham_3_intro} generalizes formulas in \cite{witten1982supersymmetry} for the signature invariant in the smooth setting when there are Killing vector fields on a smooth orientable manifold. This is when $b=-1, y=1$.}
In the smooth setting this formula for the signature in the presence of Killing vector fields has contributions of $\pm 1$ at zeros of the vector field. This is not always the case in the singular setting, as we show in Example \ref{example_conifold_Poincare_poly} where there is an isolated fixed point contributing $0$.
The case of $y=-1$ captures the Morse inequalities for the de Rham complex in \cite{jayasinghe2023l2}.

Rigidity is a crucial ingredient in the study of the equivariant signature theorem, and formulas derived using it are of great interest in mathematics and physics. In Subsection \ref{subsection_Nut_charge_signature} we present a \textbf{third formula for the equivariant signature}, following ideas in \cite{gibbons1979classification} that were used by Gibbons and Hawking in classifying gravitational instantons, and we extend formulas used to compute \textbf{NUT charge} contributions for certain gravitational actions and entropy functionals by them.

In Subsection \ref{subsubsection_spin_Morse_inequalities} we explore the rigidity of complexes twisted by fractional powers of the canonical bundle $K^{1/N}=L$ when such line bundles exist, the case where $N=2$ corresponding to spin Dirac complexes.
We show in Proposition \ref{Proposition_action_lifts_to_lifted_bundle} that given an isometric K\"ahler Hamiltonian $S^1$ action on a stratified pseudomanifold with a wedge K\"ahler metric, the group action lifts to one on $L$, after which we prove the following result (a restatement of Theorem \ref{Theorem_harmonic_spinors_vanish_non_intro}) in the smooth case.

\begin{theorem}
\label{Theorem_harmonic_spinors_vanish}
Let $X$ be a \textbf{smooth} K\"ahler manifold with a Hermitian line bundle $L$ such that $L^{\otimes N}=K_X$ where $K_X$ is the canonical bundle, where $N>1$ is an integer. Let $E=L^{k}$, where $0< k <N$ is an integer. We consider the complex $\mathcal{P}=(L^2\Omega^{0,q}(X;L), \mathcal{D}_{\alpha}(P),P)$ where $P=\overline{\partial}_E$.

If there is a K\"ahler action of a compact connected Lie group G with isolated fixed points (with at least one fixed point), then the cohomology of the Dolbeault complex twisted by the bundle $E$ vanishes in all degrees. In particular the twisted Dirac operator is invertible.
\end{theorem}

We also prove a generalization to singular spaces for the case of the spin Dirac operator in Theorem \ref{Theorem_harmonic_spinors_vanish_singular}, with additional assumptions, where we study torus actions instead of circle actions.

Lastly we study \textbf{k-Rarita Schwinger} (k-RS) operators in Subsection \ref{subsection_RS_ops}, which cover various notions of Rarita-Schwinger operators studied in the literature, including \cite{witten1983fermion} where the Witten deformation of Dolbeault complexes was first studied.
The value of $k=-1$ gives the operator in \cite{witten1983fermion,alvarez1984gravitational}, while that for $k=+1$ gives the version in \cite{HommaSemmelmannRaritaSchwinger2019,RafeManyRarita2021}. We extend a formula in \cite{witten1983fermion} to the case of general $k$ after correcting a minor typo, and investigate Morse inequalities, and extensions to singular spaces.

We end this summary of results by observing that many other equivariant formulas including those used to compute various partition functions in physics (some of which were briefly studied in \cite{jayasinghe2023l2}) can now be generalized to singular spaces with the work in this article, some of which we discuss in Subsection \ref{subsection_background_related_results}.

\subsection{Summary of construction of the Witten instanton complex}
\label{subsection_construction_summary_intro}

The proof of Theorem \ref{Theorem_strong_Morse_Dolbeault_version_2_intro} hinges on an analytic construction of a \textit{\textbf{holomorphic Witten instanton complex}} in Theorem \ref{theorem_small_eig_complex_intro} which builds on ideas in the smooth setting developed in \cite{witten1984holomorphic,bismut1991complex,wu1998equivariant} with a few technical twists.
Our construction of the instanton complex in the $L^2$ de Rham case for stratified spaces with wedge metrics in \cite{jayasinghe2023l2} was based on those ideas, with differences arising in technicalities.

Roughly the idea is as follows.
The null spaces of model operators in neighbourhoods of critical points of K\"ahler Hamiltonian Morse functions for \textit{Witten deformed} Laplace-type operators $\Delta_{\varepsilon}$ of local complexes are shown to be isomorphic to \textit{local cohomology groups} of the Hilbert complexes restricted to fundamental neighbourhoods of fixed points. In the smooth case, estimates are used to show that the small eigenvalue eigensections of $\Delta_{\varepsilon}$ on unions of local neighbourhoods approximate the global eigensections of 
$\Delta_{\varepsilon}$ on $X$ with eigenvalues which decay (of order $\varepsilon^{-1}$) as $\varepsilon$ goes to $\infty$, while the other eigenvalues of $\Delta_{\varepsilon}$ grow of order $\varepsilon$ at the local and global levels. This yields a spectral gap $[C_1\varepsilon^{-1}, C_2 \varepsilon]$ for fixed $C_1,C_2 >0$ for $\Delta_{\varepsilon}$ on $X$.

In the setting of this article we prove weaker estimates which give a much weaker spectral gap, namely that there is a subcomplex of the Witten deformed Dolbeault complex that corresponds to eigensections where the eigenvalues grow slower than  $\varepsilon^{2f}$ where $f<1/2$ as $\varepsilon$ goes to $\infty$. This yields a spectral gap $[C_1\varepsilon^{2f}, C_2 \varepsilon]$ for fixed $C_1,C_2 >0$ for $\Delta_{\varepsilon}$ which we use to construct a \textit{\textbf{``smaller" eigenvalue complex/ Witten instanton complex}}.
This suffices for the proof of the Morse inequalities.

The main ingredients are the following results.

\begin{enumerate}
    \item Proposition \ref{Proposition_model_spectral_gap_modified_general}, a spectral gap result for the model harmonic oscillator on tangent cones, for equivariant complexes.       

    \item A \textit{polynomial expansion} of the operator near the critical points of the K\"ahler Morse Hamiltonian, given in Lemma \ref{Lemma_operator_estimates}. This step was be avoided in \cite{jayasinghe2023l2} for the de Rham case by perturbing the metric near the critical points to be the model metric.

    \item Estimates for certain \textit{global} operators in Proposition \ref{proposition_Zhangs_Morse_inequalities_intermediate_estimates} obtained using information from the analysis of the local operator. Estimates in Proposition \ref{Propostion_growth_estimate_witten_deformed} showing the vanishing of harmonic forms of the deformed operator away from critical points.

    \item Parametric resolvent estimates in Lemma \ref{Lemma_inequality_spectral_for_Witten_deformation} and Proposition \ref{Proposition_small_eig_estimate_and_dimension} showing that the dimension of the vector space of approximate solutions is equal to the dimension of the vector space of \textit{``smaller"} eigenvalue eigensections, leading to the construction of the Witten instanton complex in Theorem \ref{theorem_small_eig_complex_intro}.
\end{enumerate}

The first ingredient follows from an explicit description of an orthonormal basis of eigensections on the tangent cone using separation of variables and Sturm-Liouville theory, together with a scaling argument. The exact computation of an orthonormal basis of eigensections was carried out in detail for the case of the Dolbeault complex in \cite{jayasinghe2023l2} without Witten deformation. The case of the de Rham complex was handled there as well.

The key ingredients necessary for the second numbered ingredient were proven in the smooth setting in \cite{bismut1991complex} in a generality that holds for Morse-Bott critical point sets, where the asymptotics of the metric in the smooth setting allows one to prove much stronger estimates for the Dirac-type operator. 
In fact we keep careful track of the constants in the crucial estimates in Section \ref{Section_proof_of_main_result}. In particular for wedge metrics that are asymptotically $\delta$ with $\delta \geq 1$ there are bona fide \textit{\textbf{small eigenvalue complexes}}, with the small eigenvalues going to $0$ for $\delta$ strictly greater than $1$. We also see that for locally conformally product type metrics where the conformal factor is radial (see Definition \ref{definition_locally_conformally_pt_wedge}) that there is a small eigenvalue complex. \textit{All the examples we study in this article can be easily seen to be of this type}.

The third and fourth numbered ingredients are straightforward generalizations of results for the de Rham case in \cite{jayasinghe2023l2,Zhanglectures} for equivariant complexes, accommodating the polynomial approximation of the operator. The third ingredient can be thought of as the \textit{glue} that uses the local model operator and its properties to prove estimates for the global operator.

\subsection{Background and related results}
\label{subsection_background_related_results}

The Atiyah-Bott-Berligne-Vergne localization theorem and similar results relate important global and local quantities. Supersymmetric localization, equivariant cohomology, Bott residue formulae and the Atiyah-Bott-Lefschetz fixed point theorem are all instances of such results, and are tightly woven into the fabric of many braches of mathematics and physics (see \cite[\S 1.2]{jayasinghe2023l2} for a survey). Ideas of Witten were key in understanding the role of supersymmetry in such localization results. A large number of localization results have been proven for $\mathbb{Z}_2$ graded Dirac type operators, related to equivariant K-theory and K-homology (see \cite{baum1979lefschetz}), even when there are richer $\mathbb{Z}$ gradings of complexes involved.

While studying supersymmetry breaking (\cite{witten1981dynamical,witten1982constraints}), Witten gave a groundbreaking take on Morse theory in \cite{witten1982supersymmetry}, where the $\mathbb{Z}$ grading of the de Rham complex plays a more prominent role as opposed to results such as the Lefschetz fixed point theorem (which can be thought of as $\mathbb{Z}_2$ graded versions). 
In \cite{bott1988morse} Bott suggested the name \textit{Witten complex} (we followed Zhang's nomenclature in \cite[\S 6.3]{Zhanglectures}, c.f. \cite{wu2003instanton}) and discusses the importance of uniting ideas in physics such as tunneling and topology.
Witten went beyond the case of the de Rham complex and formulated results for other operators such as the signature operator when there are Killing vector fields on orientable spaces, outlining novel proofs for the Hopf-index theorem and for formulas for the equivariant signature.
In \cite{witten1984holomorphic} he expanded on these ideas for the case of the Dolbeault complex on K\"ahler manifolds with Hamiltonian circle actions acting by isometry, which were also used to study Fermion Quantum numbers in Kaluza Klein reduction in \cite{witten1983fermion}.

The technique of Witten deformation in the smooth complex setting was used in the study of Quillen metrics in \cite{bismut1991complex} by Bismut and Lebeau, from which many techniques and results were used in the proof of the holomorphic Morse inequalities in the four articles \cite{mathai1997equivariant,wu1996equivariant,wu1998equivariant,wu2003instanton} authored by Wu, Mathai and Zhang. Some pertinent results of \cite{bismut1991complex} in the smooth case were worked out for non-isolated fixed point sets which are Morse-Bott critical points of Hamiltonians. In the case of stratified spaces there are natural Hamiltonian actions which have fixed point sets which are stratified subspaces, which lead to additional technical complications and we postpone a treatment of that case to a later article.

In the case of $S^1$ actions, we have inequalities and dual inequalities. These correspond to the two action chambers of the Lie algebra $\mathbb{R}$ of $S^1$ in the formulation for more general compact connected Lie groups $G$ in \cite{wu1996equivariant} acting on smooth spaces. There are no extra analytic considerations in the proof of this more general situation even in the singular case, and we focus on the analysis in this article. 

In \cite[\S 4]{wu1996equivariant}, it was shown that the strong form of the inequalities do not hold for general Dolbeault-Dirac complexes for actions that do not preserve a K\"ahler structure with an explicit counterexample, 
for which the weak form of the inequalities in Corollary \ref{Corollary_weak_Morse_inequalities} continue to hold. In \cite{wu2003instanton} the strong form of the inequalities were proven for group actions on complex manifolds that correspond to a Bialynicki-Birula decomposition of the space.

These techniques have been used to prove localization results for spin$^{\mathbb{C}}$ Dirac complexes on symplectic (and even almost complex manifolds \cite{ParadanLocRiemRoch2001}) manifolds and prove versions of the Guillemin-Sternberg conjecture (quantization commutes with reduction) (see \cite{vergne1994quantification,meinrenken1998symplectic,tian1998holomorphic,TianZhangQUANTIZATION1998,braverman2000index}). The equivariant localization formulas for spin$^{\mathbb{C}}$ Dirac complexes in \cite{jayasinghe2023l2} are more general in the case of isolated fixed points since we did not assume that the actions were of group actions which were compact connected, or that they were (generalized) Hamiltonian.

In \cite{Kirwan_Morse_21}, the Morse inequalities have been extended to smooth manifolds with critical point sets that are stratified spaces.
Morse inequalities for the de Rham complex on spaces with conic and conformally conic metrics have been studied in \cite{LudwigMorse2013,ludwig2013analytic} by constructing an instanton complex, and in \cite{Jesus2018Wittens,Jesus2018Wittensgeneral} avoiding it.

In \cite{jayasinghe2023l2}, we constructed the de Rham instanton complex for Witt spaces with wedge metrics which were of \textit{product type} near neighbourhoods of critical points of the stratified Morse functions considered in that article. That assumption can be removed using the technical results in Section \ref{Section_proof_of_main_result}.
This is useful in extending studies such as those in \cite{DangRiviere2021Witten} where the Witten deformed Laplace type operators are studied without perturbing the metric. The methods developed in this article are related to semi-classical analysis on stratified spaces (where the standard semi-classical parameter corresponds to $\hbar=1/\varepsilon$ where $\varepsilon$ is the deformation parameter in this article), for instance as in \cite{DangRiviere2021Witten} and references therein including connections to Policott Ruelle resonances and analytic torsion via the Freed conjecture.  Analytic torsion and the Cheeger-M\"uller theorem have also been studied widely using the techniques of Witten deformation following the work of \cite{BismutZhangCheegerMuller1992} (see also \cite{ludwig2020extension} in the singular setting). 

The techniques developed here are also related to techniques needed to generalize the asymptotic holomorphic Morse inequalities of Demailly, which have been investigated on some spaces with singularities and with singular bundle metrics (see \cite{MarinescuHolMorse2023,MaMarinescu2007HolMorseBook}). Observations of Demailly in \cite{demailly1991holomorphic} can be used to see that candidate asymptotic holomorphic Morse inequalities can be formulated as soon as there is a Hirzebruch-Riemann-Roch formula, which for the VAPS domain ($\alpha=1/2$) is given by the index theorem in \cite{Albin_2017_index}. This is similar to how the candidate equivariant holomorphic Morse inequalities can be formulated once there is a Lefschetz-Riemann-Roch theorem (see Subsection \ref{subsubsection_conjectural_inequalities}).

In the setting of Theorem \ref{Theorem_strong_Morse_Dolbeault_version_2_intro}, the holomorphic Morse inequalities are a generalization of the holomorphic Lefschetz fixed point theorem. Thus it can be used to strengthen certain results that are usually tackled by the Lefschetz fixed point theorem. This includes questions of quantization commutes with reduction, which in \cite{braverman1999holomorphic,tian1998holomorphic}, where certain inequalities associated to reduction were proven to be equalities. 

Since Witten deformation is easier to use than heat kernel techniques, it has also been used to study spin$^{\mathbb{C}}$ Dirac operators and quantization, even for transversally elliptic operators on non-compact smooth manifolds (see for instance \cite{TianZhangQUANTIZATION1998,tian1998holomorphic,braverman2000index,zhang1998holomorphic}).
Novikov generalized Witten deformation in the de Rham case by taking closed one forms $dh$ which are not necessarily exact (where $h$ is Morse) and a version for the Dolbeault complex has been investigated in \cite{JesusGilkeylocalDolbeault2021}. Moreover there are deformations of Dolbeault complexes related to Landau-Ginzburg models \cite{witten1993phases,dai2023witten}.

Localization formulas are crucial to understand and compute various partition functions that are widely studied in mathematical physics \cite{pestun2017localization}, including the Nekrasov partition function (\cite{nekrasov2003seiberg,NekrasovABCDinsta}).
The ideas of Witten deformation compatible with equivariant actions were also used in the thesis of Pestun (see footnote 6 of \cite{pestun2012localization} and relevant deformations) in computing instanton contributions to path integrals of integrable quantum field theories, and his method (dubbed Pestunization) have been extended broadly since then (see, e.g., \cite{festuccia2020transversally,festuccia2020twisting,hekmati2023equivariant,mauch2022index}). In \cite{jayasinghe2023l2} we discussed extensions of such theories to the singular setting, how some symmetries of path integrals that are of physical importance fail to hold on certain singular spaces, and the work in this article brings us a step closer to more rigorous studies of related work in the singular setting (see Remark \ref{Remark_instanton_computations}). Briefly, the self dual and anti-self dual complexes have the same (equivariant) indices as $1/2(\chi_{-1} \pm  \chi_{1})$, which can be algorithmically computable with the techniques we introduce on various spaces. These formulas also show up in the study of BPS invariants and more sophisticated studies of blackhole entropy related to supersymmetry and the ADS/CFT correspondence (see, e.g., \cite[\S 4.2]{gupta2015supersymmetric}) than what we consider in subsection \ref{subsection_Nut_charge_signature}.

We prove our main results under the assumption that the complexes for the algebraic domains are transversally Fredholm. Similar domains have been studied for other operators in \cite{Vertman2015Heatalgebraic,VertmanBahaudDryden2015Heatalgebraic} at the global level, and it is plausible that similar techniques can be used to show that the complexes we study here are transversally Fredholm, if not Fredholm. 

In the smooth setting Witten uses the holomorphic Morse inequalities to prove that for smooth manifolds admitting Hamiltonian K\"ahler circle actions with isolated fixed points, $\mathcal{H}^{p,q}(X)=0$ for $p \neq q$ (in fact Carrell and Liebermann had proven this for arbitrary holomorphic circle actions on K\"ahler manifolds \cite{carrell1973holomorphic}).
This is an analog of the Bochner-Kodaira-Nakano estimates on smooth manifolds, which shows that assumptions on the curvature of line bundles imply vanishing results on the cohomology of bundles. 
For instance is well known that if there is a trivializing section of the canonical bundle of a smooth K\"ahler manifold $X^{2n}$ which gives rise to a non-trivial element in $\mathcal{H}^{0,n}(X)$, then it is a Calabi-Yau admitting a Ricci flat metric and does not admit K\"ahler circle actions.

Perhaps the simplest example of rigidity on smooth manifolds is for the global Lefschetz number of the trivial bundle of a complex manifold for a holomorphic circle action, which follows by results of \cite{carrell1973holomorphic} (c.f., page 330 \cite{witten1984holomorphic}) described above. Rigidity for the structure sheaf fails for certain non-normal algebraic varieties in the Baum-Fulton-MacPherson theory, and for domains other than the minimal domain (see Examples 7.33 and 7.37 of \cite{jayasinghe2023l2}, Subsection \ref{subsubsection_conjectural_inequalities}).

In \cite{witten1983fermion} Witten shows how to deform twisted Dolbeault complexes to compute the equivariant indices of spin Dirac and Rarita-Schwinger operators, and explores rigidity questions following \cite{atiyah1970spin}, exploring consequences in supergravity, continued in later work (see for instance \cite{alvarez1984gravitational} where the equivariant index and the spin Dirac and Rarita-Schwinger operators are key characters). Motivated by Witten's questions rigidity phenomena were studied by Landweber and Stong, leading to their elliptic genus and Witten's conjectures on the rigidity of the Dirac operator on loop spaces (see \cite{landweber1988circle,zagier1988note,bott1989rigidity}).

Our proof of the vanishing of harmonic spinors using the holomorphic Morse inequalities fits into a circle of results related to positive scalar curvature, group actions and harmonic spinors. It is well known that the Lichnerowicz inequality implies that there are no harmonic spinors on spaces with positive scalar curvature metrics, and it is known that isometric group actions are abundant on geometries with positive and zero scalar curvature.
The relation between fixed points, spinors and positive scalar curvature has been explored, for instance in \cite{wiemeler2016circle,lawson1974scalar}.

In the smooth setting, rigidity results on complete intersections have been studied in \cite{dessai2017complete} where the existence of circle actions on spaces with positive curvature metrics is explored.
One can use arguments similar to those in the proof of Theorem \ref{Theorem_harmonic_spinors_vanish_singular} for complete intersections with canonical singularities and positivity conditions for the canonical bundle to study the rigidity of Dolbeault complexes twisted by other powers of the canonical bundle.
Elliptic genera have been studied on singular spaces with restrictions on the singularities in \cite{weber2016equivariant,donten2018equivariant,borisov2003elliptic,BurtTotaro2000singularelliptic}, but as we observed in \cite{jayasinghe2023l2}, these are different from the $L^2$ versions in general. Rigidity and dualities of $\chi_y$ invariants that can be proven from properties of elliptic genera in the smooth setting hold for $L^2$ $\chi_y$ invariants with the minimal domain for the Dolbeault complex. 

Many important theories in mathematics and physics take the fact that cohomology is rigid under the action of continuous Lie groups for granted, not true in localization theories generalizing formulas of \cite{baum1979lefschetz}, for instance in \cite{equivariant_intersection_edidin}.
The rigidity of the signature complex is intimately tied to the fact that group actions act trivially on the cohomology of the de Rham complex on smooth manifolds, as observed by Witten in \cite[\S 3]{witten1982supersymmetry} where he explains how this leads to important formulas in mathematics and physics for the signature in terms of the weights at fixed points.
This is an important difference between $L^2$ cohomology and singular cohomology on singular spaces, in light of Theorem \ref{Theorem_instanton_correspondence_deRham_Dolbeault} and counterexamples in singular cohomology we study in this article. We explore this in Subsection \ref{subsection_Nut_charge_signature} as well (see Remarks \ref{Remark_considerations_integrals_cup_products}, and \ref{Remark_formulas_NUTs_other_cohom_theory}). In particular we show that some important equivariant invariants can be computed independent of the choices of domain (see Remark \ref{Remark_equivariant_volumes}) and can be computed even using the algebraic tools in \cite{baum1979lefschetz,equivariant_intersection_edidin}.

Witten also highlights the importance of picking orientations in deriving his signature formulas (see Remark \ref{Remark_orientation_signature_Hamiltonian}) and we observe that there are many publications which have sign errors in important formulas due to failures in accounting for co-orientations of fixed point sets and critical point sets (see Remark \ref{Remark_orientation_issues}, \cite{Orientation_Morse_Bott_Rot_2016}).

We use rigidity of $L^2$ cohomology to prove Theorem \ref{theorem_rigidity_Poincare_Hodge} establishing rigidity for the $\chi_{y,b}$ polynomials, generalizing results in \cite{KosniowskiLefschetzApplication1970}, and \cite[\S 3]{atiyah1970spin} for complex manifolds in our singular K\"ahler setting, and we observe that similar arguments together with the holomorphic Lefschetz fixed point theorem in \cite{jayasinghe2023l2} can be used to establish rigidity of the $\chi_{y}$ invariant for Hamiltonian actions when there is only an almost complex structure, and when the actions are locally K\"ahler Hamiltonian.
It also generalizes some formulas studied by Witten for the signature computed in terms of local quantities at fixed points, and give a third formula for the index in Subsection \ref{subsection_Nut_charge_signature} extending work of \cite{gibbons1979classification}, where we extend formulas for NUT charges that appear in the gravitational action and entropy functionals on gravitational instantons. We chose this example since it showcases how paying closer attention to the equivariant formulas, already studied and used in that article, would have made certain sign errors clear.
We observe in Remark \ref{Remark_orientation_issues} that the cumbersome coordinate computations used in \cite{gibbons1979classification} leads to sign issues which should have been clear to the authors already by crucial physical arguments drawn from the equivariant signature formula that they already make in that highly influential article.

The more recent study of gravitational instantons in \cite{aksteiner2023gravitational} uses methods and results in \cite{Donghooncircleactions_2018} which are proven as consequences of rigidity, and gives a correct derivation of the formulas which we emphasize were already clear from the equivariant signature formula. 
We also observe that while the K\"ahler case seems special, conjecture 3 of \cite{aksteiner2023gravitational} shows that for certain physical applications it suffices to understand such equivariant formulae for Hermitian metrics, which can be obtained using results in \cite{jayasinghe2023l2}.

Many physical consequences are drawn from such formulae and there is significant interest in generalizing this to singular spaces, both compact and non-compact with many interesting work done by Hawking and his collaborators on those general settings (see \cite{toricgravitationalinstantonsBiquardGauduchon_2023,de2023comments} and references in the latter). One sees glimpses of derivations of equivariant quantities in sections 5, 6 of \cite{gibbons1979classification}, and relevant invariants getting physical labels (see ``mass" of the bolt in equation (6.6), the weights of actions corresponding to surface gravities).

The equivariant $\chi_{y,b}$ polynomials, equivariant $\chi_y$ polynomials and equivariant signature invariants appear in many related applications including classifying group actions with prescribed numbers of fixed points \cite{wiemeler2016circle}, group actions on 4 manifolds \cite{Donghooncircleactions_2018} in addition to classifying gravitational instantons.
The ideas of \cite{gibbons1979classification} now appear in the study of supergravity theories that embrace the now better understood theory of equivariant localization on orbifolds \cite{Equivariant_loc_Sparks_2024}, and extending such work to stratified pseudomanifolds will require a more careful study of localization as we discuss in Subsection \ref{subsection_Nut_charge_signature}.

The Rarita-Schwinger operator is also important in the study of supergravity and to understand gravitational anomalies (\cite{witten1983fermion,alvarez1984gravitational}). There the lower order parts of the Rarita-Schwinger operator are not accounted for since only the indices are studied, and the local contributions at fixed points are computed using localization formulas for twisted Dolbeault complexes.

There are various versions of operators called Rarita-Schwinger operators studied in the literature (see for instance \cite{HitchinStableforms2001,HommaSemmelmannRaritaSchwinger2019,RafeManyRarita2021}) where the differences of the indices of the relvant operators can be attributed to various choices of gauge fields and we refer to the appendix of \cite{freed2002k} for an explanation. The results written for versions of these operators studied independently cannot be used directly for other versions (though they can certainly be adapted), and can lead to confusion due to the use of the same name for different operators. Perhaps one also has to be careful about whether the number of elements in the kernel of the Rarita-Schwinger operator is affected by the lower order terms since only the index is a stable homotopy invariant.
Our investigation of these operators is simple, and adopts methods in \cite{witten1983fermion} for the study of equivariant indices for other versions, and we indicate how some methods extend to singular spaces.
These methods may not be useful for versions on odd dimensional spaces such as studied in 11 dimensional $M$-theory.
Eta invariants and boundary conditions for various versions of Rarita-Schwinger operators are also studied in various contexts (see \cite{bar2022boundary,debray2022anomaly}) and since local Lefschetz numbers can be written in terms of equivariant eta invariants (see \cite[\S 3]{weiping1990note} for the Spin Dirac operator on a cone) these can be computed using the localization formulas we present in our singular setting.

\begin{justify}
\textbf{Acknowledgements:} 
I thank my advisor Pierre Albin for many discussions on the subject of this article, for ideas and valuable comments.
A reading group on the Cheeger-M\"uller theorem organized by him proved to be useful, and I thank all participants. Many conversations with Gabriele La Nave and Hadrian Quan on wedge symplectic structures proved useful in this work.
I thank Dan Berwick-Evans, Eugene Lerman, Nachiketa Adhikari, and Donghoon Jang for some useful discussions. I thank Jes\'us \'Alvarez L\'opez for his interest and comments on this work. I thank Edward Witten for enlightening answers for some questions on Rarita Schwinger operators, and thank Lashi Bandara and Alberto Richtsfeld as well for useful discussions on the topic.
I thank Dhantha Gunarathna for hosting me in the summer of 2023 when part of this work was developed.
I was partially supported by Pierre Albin's NSF grant DMS-1711325.
\end{justify}

\section{Background}

In this section, we review stratified spaces and their resolutions, as well as wedge K\"ahler structures and the spin$^{\mathbb{C}}$ Dirac operator, before discussing local structures in neighbourhoods of fixed points of the circle actions that we study.

\subsection{Stratified pseudomanifolds with K\"ahler structures}

We review some of the background for stratified spaces and wedge K\"ahler geometry, referring the reader to \cite{Albin_2017_index} and \cite[\S 2]{jayasinghe2023l2} for more details.

\subsubsection{Stratified pseudomanifolds}

All topological spaces we consider will be Hausdorff, locally compact topological spaces with a countable basis for its topology. Recall that a subset $W$ of a topological space $V$ is locally closed if every point $a \in W$ has a neighborhood $\mathcal{U}$ in $V$ such that $W \cap \mathcal{U}$ is closed in $\mathcal{U}$. A collection ${S}$ of subsets of $V$ is locally finite if every point $v \in V$ has a neighborhood that intersects only finitely many sets in ${S}$.

A Thom-Mather stratified pseudomanifold is a topologically stratified space (see Definition 4.11 of \cite{Kirwan&woolf_2006_book}) with additional conditions and we refer the reader to \cite[\S 5]{veronabook_1984} for more details on Thom-Mather stratified spaces.

\begin{definition}
    \label{Thom-Mather stratified space}    
A Thom-Mather stratified space $\widehat{X}$ is a metrizable, locally compact, second countable space which admits a locally finite decomposition into a union of locally closed strata $\mathcal{S}(\widehat{X})=\left\{Y_{\alpha}\right\}$, where each $Y_{\alpha}$ is a smooth manifold, with dimension depending on the index $\alpha$. We assume the following:
\begin{enumerate}
    \item If $Y_{\alpha}, Y_{\beta} \in \mathcal{S}(X)$ and $Y_{\alpha} \cap \overline{Y_{\beta}} \neq \emptyset$, then $Y_{\alpha} \subset \overline{Y_{\beta}}$.

    \item Each stratum $Y$ is endowed with a set of `control data' $T_{Y}, \pi_{Y}$ and $\rho_{Y}$; here $T_{Y}$ is a neighbourhood of $Y$ in $X$ which retracts onto $Y, \pi_{Y}: T_{Y} \longrightarrow Y$ is a fixed continuous retraction and $\rho_{Y}: T_{Y} \rightarrow[0,2)$ is a `radial function' in this tubular neighbourhood such that $\rho_{Y}^{-1}(0)=Y$. Furthermore, we require that if $\widetilde{Y} \in \mathcal{S}(X)$ and $\widetilde{Y} \cap T_{Y} \neq \emptyset$, then
\begin{equation}    
    \left(\pi_{Y}, \rho_{Y}\right): T_{Y} \cap \widetilde{Y}  \setminus Y \longrightarrow Y \times[0,2)
\end{equation}
is a proper differentiable submersion.
\item If $W, Y, \widetilde{Y} \in \mathcal{S}(X)$, and if $p \in T_{Y} \cap T_{\widetilde{Y}} \cap W$ and $\pi_{\widetilde{Y}}(p) \in T_{Y} \cap \widetilde{Y}$, then $\pi_{Y}\left(\pi_{\widetilde{Y}}(p)\right)=\pi_{Y}(p)$ and $\rho_{Y}\left(\pi_{\widetilde{Y}}(p)\right)=\rho_{Y}(p)$.

\item If $Y, \widetilde{Y} \in \mathcal{S}(X)$, then
$$
\begin{aligned}
Y \cap \overline{\widetilde{Y}} \neq \emptyset & \Leftrightarrow T_{Y} \cap \widetilde{Y} \neq \emptyset, \\
T_{Y} \cap T_{\widetilde{Y}} \neq \emptyset & \Leftrightarrow Y \subset \overline{\widetilde{Y}}, Y=\widetilde{Y} \quad \text {or } \widetilde{Y} \subset \overline{Y} .
\end{aligned}
$$

\item For each $Y \in \mathcal{S}(X)$, the restriction $\pi_{Y}: T_{Y} \rightarrow Y$ is a locally trivial fibration with fibre the cone $C\left(Z_{Y}\right)$ over some other stratified space $Z_{Y}$ (called the \textbf{\textit{link}} over $Y)$, with atlas $\mathcal{U}_{Y}=\{(\phi, \mathcal{U})\}$ where each $\phi$ is a trivialization 
\begin{equation}
\label{equation_chartlike_map}
\pi_{Y}^{-1}(\mathcal{U}) \rightarrow \mathcal{U} \times C\left(Z_{Y}\right),    
\end{equation}
and the transition functions are stratified isomorphisms 
of $C\left(Z_{Y}\right)$ which preserve the rays of each conic fibre as well as the radial variable $\rho_{Y}$ itself, hence are suspensions of isomorphisms of each link $Z_{Y}$ which vary smoothly with the variable $y \in \mathcal{U}$.

If in addition we let $\widehat{X}_{j}$ be the union of all strata of dimensions less than or equal to $j$, and require that 

\item $\widehat{X}_{n-1}=\widehat{X}_{n-2}$ and $\widehat{X} \backslash \widehat{X}_{n-2}$ is dense in $\widehat{X}$, 
then we say that $\widehat{X}$ is a stratified pseudomanifold.
\end{enumerate}
\end{definition}

We remark that this ensures that $\widehat{X}_j \backslash \widehat{X}_{j-1}$ is a smooth manifold of dimension $j$, and the connected components of this are called the strata of \textit{\textbf{depth}} $j$. We define the \textit{\textbf{regular part}} of $\widehat{X}$ to be $X^{reg}:=\widehat{X}^n \setminus \widehat{X}^{n-2}$. 
The definition shows that it is natural to study the topology of these spaces using iterated conic metrics, which we call wedge metrics and introduce more formally later in this section.
The following definition is based on the first proposition of \cite[\S 4.1]{goresky1980intersection}.
\begin{definition}
\label{definition_normal_pseudomanifold}
    Given a stratified pseudomanifold $\widehat{X}$ of dimension $n$, if the link at each $x \in \widehat{X}_{n-2}$ is connected, it is called a \textit{\textbf{normal pseudomanifold}}.
\end{definition}
Goresky and MacPherson explore this definition in \cite[\S 4]{goresky1980intersection} to which we refer the reader. The metric completion of $\widehat{X}^{reg}$ with respect to a wedge metric corresponds to the topological normalization of non-normal pseudomanifolds. This topological normalization is unique for a given pseudomanifold.
We explored how the $L^2$ Atiyah-Bott-Lefschetz fixed point theorem holds without the normal assumption in \cite{jayasinghe2023l2} and the constructions in this article hold without it as well.
There is a functorial equivalence between Thom-Mather stratified spaces and manifolds with corners and iterated fibration structures (see Proposition 2.5 of \cite{Albin_signature}, Theorem 6.3 of \cite{Albin_hodge_theory_cheeger_spaces}).

\subsubsection{Manifolds with corners with iterated fibration structures}

In section 1 of \cite{Albin_2017_index}, there is a detailed description of iterated fibration structures on manifolds with corners. We explain some of these structures from said article which we will use and refer the reader to the source for more details. In \cite{kottke2022products}, these are referred to as manifolds with fibered corners that are interior maximal.

An n-dimensional manifold with corners $X$ is an $n$-dimensional topological manifold with boundary, with a smooth atlas modeled on $(\mathbb{R}^+)^n$ whose boundary hypersurfaces are embedded. We denote the set of boundary hypersurfaces of $X$ by $\mathcal{M}_1(X)$. A collective boundary hypersurface refers to a finite union of non-intersecting boundary hypersurfaces.

\begin{definition}
\label{iterated_fibration_structure}
An \textit{\textbf{ iterated fibration structure}} on a manifold with corners $X$ consists of a collection of fiber bundles
\begin{center}
    $Z_Y -\mathfrak{B}_Y \xrightarrow{\phi_Y} Y$
\end{center}
where $\mathfrak{B}_Y$ is a collective boundary hypersurface of $X$ with base and fiber manifolds with corners such that:

\begin{enumerate}
    \item Each boundary hypersurface of $X$ occurs in exactly one collective boundary hypersurface $\mathfrak{B}_Y$.
    \item If $\mathfrak{B}_Y$ and $\mathfrak{B}_{\widetilde{Y}}$ intersect, then dim $Y \neq$ dim $\widetilde{Y}$, and we write $Y < \widetilde{Y}$ if dim $Y < \text{dim} \widetilde{Y}$.
    \item If $Y < \widetilde{Y}$, then $\widetilde{Y}$ has a collective boundary hypersurface $\mathfrak{B}_{Y\widetilde{Y}}$ participating in a fiber bundle $\phi_{Y\widetilde{Y}} : \mathfrak{B}_{Y\widetilde{Y}} \rightarrow Y$ such that the diagram 
    
\[\begin{tikzcd}
	{\mathfrak{B}_Y \cap \mathfrak{B}_{\widetilde{Y}}} && {\mathfrak{B}_{Y\widetilde{Y}} \subseteq \widetilde{Y}} \\
	& Y
	\arrow["{\phi_{\widetilde{Y}}}", from=1-1, to=1-3]
	\arrow["{\phi_Y}"', from=1-1, to=2-2]
	\arrow["{\phi_{Y\widetilde{Y}}}", from=1-3, to=2-2]
\end{tikzcd}\]
commutes.
\end{enumerate}
\end{definition}

The base can be assumed to be connected but the fibers are in general disconnected.
As we mentioned above, there is an equivalence between Thom-Mather stratified spaces and manifolds with corners with iterated fibration structures. 

If we view a cone over a link $Z$ as the quotient space of $[0,1]_x \times Z$ under the identification where the points of the link at $\{x=0\} \times Z$ are identified, then the quotient map is a blow-down map.
More generally, there is an inductive desingularization procedure which replaces the Thom-Mather stratified space with a manifold with corners with iterated fibration structures.
This corresponds to a \textit{\textbf{blow-down}} map $\beta : X \rightarrow \widehat{X}$, which satisfies properties given in Proposition 2.5 of \cite{Albin_signature}.
See Remark 3.3 of \cite{Albin_2017_index} for an instructive toy example.

Under this equivalence, the bases of the boundary fibrations correspond to the different strata, which we shall denote by
\begin{equation}
    \mathcal{S}(X) = \{Y: Y \textit{ is the base of a boundary fibration of X} \}.
\end{equation}
The bases and fibers of the boundary fiber bundles are manifolds with corners with iterated fibration structures (see for instance Lemma 3.4 of \cite{albin2010resolution}). The condition dim $Z_Y > 0$ for all $Y$ corresponds to the category of pseudomanifolds within the larger category of stratified spaces. The partial order on $\mathcal{S}(X)$ gives us a notion of depth
\begin{center}
$ depth_X(Y) = max \{ n \in \mathbb{N}_0 : \exists Y_i \in \mathcal{S}(X)$ s.t. $Y=Y_0 <Y_1 < ... <Y_n \}$.
\end{center}
The depth of $X$ is then the maximum of the integers $depth_X(Y)$ over $Y \in \mathcal{S}(X)$. 

We now introduce some auxiliary structures associated to manifolds with corners with iterated fibration structures.
If $H$ is a boundary hypersurface of $X$, then because it is assumed to be embedded, there is a smooth non-negative function $\rho_H$ such that 
$\rho_{H}^{-1}(0)=H$ and $d\rho_H$  does not vanish at any point on $H$. We call any such function a \textbf{\textit{boundary defining function for $H$}}.
For each $Y \in \mathcal{S}(X)$, we denote a \textbf{\textit{collective boundary defining function}} by
\begin{center}
    $\rho_Y = \prod_{H \in \mathfrak{B}_Y} \rho_H$,
    and by
    $\rho_X = \prod_{H \in \mathcal{M}_1(X)} \rho_H$
\end{center}
 a \textbf{\textit{total boundary defining function}}, where $\mathcal{M}_1(X)$ denotes the set of boundary hypersurfaces of $X$.

When describing the natural analogues of objects in differential geometry on singular spaces, the iterated fibration structure comes into play. For example,
\begin{equation}
\label{equation_smooth_functions_on_stratified_spaces}
\mathcal{C}_{\Phi}^{\infty}(X)= \{ f \in \mathcal{C}^{\infty}(X) : f \big|_{\mathfrak{B}_Y} \in 
\phi_{Y}^*\mathcal{C}^{\infty}(Y) \text{ for all } Y \in \mathcal{S}(X) \}
\end{equation}
corresponds to the smooth functions on X that are continuous on the underlying
stratified space. 


\subsubsection{Wedge metrics and related structures}
\label{subsubsection_wedge_metrics_related_structures}

On Thom-Mather stratified pseudomanifolds we can define metrics which are \textit{locally conic}. For instance, consider the model space $\widehat{X}=\mathbb{R}^k \times {Z}^+$, where ${Z}^+$ is a cone over a smooth manifold $Z$. The resolved manifold with corners that corresponds to the blowup of this space is $X=\mathbb{R}^k \times [0,\infty)_x \times Z$. We have the model wedge metric
\begin{equation}
\label{equation_definition_homogeneous_metric_cone}
    g_{w} = g_{\mathbb{R}^k} + dx^2 + x^2 g_Z,
\end{equation}
which is a product metric on the product space $X^{reg}$ and which we call a  \textbf{\textit{rigid/ product type}} wedge metric, degenerating as it approaches the stratum at $x=0$. These metrics are degenerate as bundle metrics on the tangent bundle but we can introduce a rescaled bundle on which they are non-degenerate. 

Formally, we proceed as follows. Let $X$ be a manifold with corners and iterated fibration structure. Consider the ‘wedge one-forms’
\begin{center}
    $\mathcal{V}^{*}_{w}= \{ \omega \in \mathcal{C}^{\infty}(X; T^*X) : \text{ for each } Y \in \mathcal{S}(X), \text{  }i^*_{\mathfrak{B}_Y} \omega (V) =0 \text{ for all } V \in \text{ker } D\phi_Y \}$.
\end{center}
We can identify $\mathcal{V}^{*}_{w}$ with the space of sections of a vector bundle which we call $\prescript{{w}}{}{T}^*X$, the wedge cotangent bundle of X, together with a map
\begin{equation}
    i_{w} : \prescript{{w}}{}{T}^*X \rightarrow T^*X
\end{equation}
that is an isomorphism over $X^{reg}$ such that,
\begin{center}
    $(i_{w})_*\mathcal{C}^{\infty}(X;\prescript{{w}}{}{T}^*X) =  \mathcal{V}^{*}_{w} \subseteq \mathcal{C}^{\infty}(X; T^*X)$.
\end{center}
In local coordinates near the collective boundary hypersurface, the wedge cotangent bundle is spanned by
\begin{center}
    $ dx$, $xdz$, $dy$
\end{center}
where $x$ is a boundary defining function for $\mathfrak{B}_Y$, $dz$ represents covectors along the fibers and $dy$ covectors along the base.
The dual bundle to the wedge cotangent bundle is known as the wedge tangent bundle, $\prescript{{w}}{}{T}X$. It is locally spanned by
\begin{center}
    $\partial_x$, $\frac{1}{x}\partial_z$, $\partial_y$
\end{center}
A \textit{\textbf{wedge metric}} is simply a bundle metric on the wedge tangent bundle.

The notion of a \textit{\textbf{wedge differential operator}} 
$P$ of order $k$ acting on sections of a vector bundle $E$, taking them to sections of a vector bundle $F$ is described on page 11 of \cite{Albin_2017_index}. We first define the \textit{\textbf{edge vector fields on X}} by
\begin{center}
    $\mathcal{V}_e = \{ V \in  \mathcal{C}^{\infty}(X;TX): V \big|_{\mathfrak{B}_Y}$ is tangent to the fibers of $\phi_Y$ for all $Y \in \mathcal{S}(X) \}$.
\end{center}
There is a rescaled vector bundle that is called the \textit{\textbf{edge tangent bundle}} $\prescript{e}{}TX$ together with a natural vector bundle map $i_e : \prescript{e}{}TX \rightarrow TX$ that is an isomorphism over the interior and satisfies
\begin{center}
    $(i_e)_* \mathcal{C}^{\infty}(X;\prescript{e}{}TX)=\mathcal{V}_e$.
\end{center}
In local coordinates near a point in $\mathfrak{B}_Y, (x, y, z)$,
a local frame for $\prescript{e}{}TX$ is given by 
\begin{center}
    $ x\partial_x, x\partial_y, \partial_z$
\end{center}
Note that the vector fields $x\partial_x$ and $x\partial_y$ are degenerate as sections of $TX$, but not as sections of $\prescript{e}{}TX$. 

The universal enveloping algebra of $\mathcal{V}_e$ is the ring $\text{Diff}^*_e(X)$ of edge differential operators. That is, these are the differential operators on $X$ that can be expressed locally as finite sums of products of elements of $\mathcal{V}_e$. They have the usual notion of degree and extension to sections of vector bundles, as well as an edge symbol map defined on the edge cotangent bundle (see \cite{Mazzeo_Edge_Elliptic_1,Albin_signature,Albin_hodge_theory_cheeger_spaces}).
Similarly, $\text{Diff}^*_e(X; E, F)$ denotes the edge differential operators acting on sections of a vector bundle $E$ and taking them to a sections on a vector bundle $F$.
The edge symbol (which can be defined on the space of edge pseudo-differential operators in general) is used to define the notion of ellipticity used in this article. We follow \cite{Albin_2017_index} and define the map
\begin{equation}
    \overline{\sigma_k} : \text{Diff}^k_e(X;E,F) \rightarrow \rho_{RC}^{-k}\mathcal{C}^{\infty}(RC(\prescript{e}{}T^*X), \pi^*hom(E,F))
\end{equation}
which is the usual symbol map where $RC(\prescript{e}{}T^*X)$ denotes the radial compactification of the edge cotangent bundle and $\pi: \prescript{e}{}T^*X \rightarrow X$ is the projection map. We denote by $\rho_{RC}$ a boundary defining function for the
boundary at radial infinity. Multiplying $\overline{\sigma_k}$ by $\rho_{RC}^k$ 
the resulting map $\sigma_k$ is called the \textbf{\textit{edge symbol}} (see Section 3.3 of \cite{Albin_2017_index}).
This fits into the following short exact sequence
\begin{equation}
\label{edge symbol sequence}
    0 \rightarrow \text{Diff}^{k-1}_e(X;E,F) \rightarrow \text{Diff}^{k}_e(X;E,F) \xrightarrow{\sigma} \mathcal{C}^{\infty}(^e\mathbb{S}^*X,\pi^*hom(E,F)) \rightarrow 0.
\end{equation}
If the edge symbol of a differential operator is invertible away from the zero section we call such an operator  \textbf{\textit{edge elliptic}}.
We define \textit{\textbf{wedge differential operators}} by 
\begin{equation*}
    \text{Diff}^k_w(X;E,F) = \rho^{-k}_X \text{Diff}^k_e(X;E,F)
\end{equation*}
following, e.g., \cite{Albin_2017_index}, where $\rho_X$ is a total boundary defining function for $X$. 
If a wedge differential operator $D_k$ of order $k$ can be written as $\rho_X^{-k}A_k$ where $A_k$ is an edge elliptic differential operator of order $k$, then $D_k$ is said to be a  \textbf{\textit{wedge elliptic}} operator.

\subsubsection{Asymptotically wedge metrics and K\"ahler structures}
\label{subsection_asymptotics_metric}

An \textit{\textbf{asymptotically $\delta$ wedge metric}} is defined inductively. On depth zero spaces, which are just smooth manifolds, a wedge metric is a Riemannian metric. Assuming we have defined asymptotically $\delta$ wedge metrics on spaces of depth less than $k$, let us assume $X$ has depth $k$.
A wedge metric $g_w$ on $X$ is an 
asymptotically $\delta$ wedge metric if, for every $Y \in \mathcal{S}(X)$ of depth $k$ there is a collar neighbourhood $\mathscr{C}(\mathfrak{B}_Y) \cong [0,1)_x \times \mathfrak{B}_Y$ of $\mathfrak{B}_Y$ in $X$, a metric $g_{w, pt}$ of the form
\begin{equation}
    g_{w, pt} = dx^2 +x^2 g_{\mathfrak{B}_Y/Y} +\phi^*_Yg_Y
\end{equation}
where $g_Y$ is an asymptotically $\delta$ wedge metric on Y, $g_{\mathfrak{B}_Y/Y}+\phi^*_Yg_Y$ is a submersion metric for $\mathfrak{B}_Y \xrightarrow{\phi_Y} Y$ and $g_{\mathfrak{B}_Y/Y}$ restricts to each fiber of $\phi_Y$ to be an asymptotically $\delta$ wedge metric on $Z_Y$ and
\begin{equation}
\label{equation_structure of the metric}
    g_{w} - g_{w, pt} \in x^{2\delta} \mathcal{C}^{\infty} (\mathscr{C}(\mathfrak{B}_Y); S^2(^{w}T^*X))).
\end{equation}
If at each step  $g_{w} = g_{w, pt}$, we say $g_w$ is a \textit{\textbf{rigid}} or \textit{\textbf{product-type wedge metric}}. Off of these collar neighborhoods of the stratum of depth $k$, the form of the metric is fixed by the induction as a $\delta$ wedge metric of lower depth.

\begin{remark}
\label{remark_totally_geodesic_wedge}
In \cite{Albin_2017_index} the case of $2\delta=1$ (called \textit{\textbf{exact wedge metrics}}) and $2\delta=2$ (called totally geodesic wedge metrics) were studied, the distinction only relevant in Getzler rescaling arguments at the strata.
\end{remark}

Given a stratified pseudomanifold with a wedge metric, a complex structure on the wedge tangent bundle is called a \textit{\textbf{wedge complex structure}}.

\begin{definition}[wedge K\"ahler structures]
\label{definition_wedge_kahler_structure}
    If there is a symplectic form $\omega$ on $X^{reg}$ extending to a non-degenerate closed wedge two form such that it is tamed by a wedge complex structure $J$, yielding a wedge metric $g=\omega(J \cdot, \cdot)$ on $X$, then we say that $(g,\omega,J)$ is a \textit{\textbf{wedge K\"ahler structure}} on $X$.
\end{definition}

As in the smooth case, any two of the three constituents of a triple $(g,\omega,J)$ determines the third.
Projective algebraic varieties with wedge metrics such as conifolds and even the cusp curve, as we studied in \cite[\S 7.3]{jayasinghe2023l2} are examples of such spaces.

We briefly review some background on conic K\"ahler metrics, referring to \cite[\S 7.1]{jayasinghe2023l2,boyer&galicki,dragomir2007differential,blair2010riemannian} for more details on complex structures on cones and their CR structures, as well as Sasaki structures of K\"ahler cones.

It is well known that a complex structure on a cone $(C(Z), dx^2 +x^2 g_Z)$ induces a CR structure on the link $(Z,g_Z)$, which in particular yields an almost contact metric structure on the link if it is smooth (see for instance Section 1.1.4 of \cite{dragomir2007differential}). In the case where the CR structure is pseudoconvex, the almost complex structure is a contact structure. 
A Sasaki structure on $M$
is equivalent to a K\"ahler structure on the metric cone over $M$.
In this case, the CR structure is pseudoconvex and has a contact structure. Taking the quotient by the action of the Reeb vector field $\xi$, one gets (in general) an orbifold $\Sigma = Z / \xi$ which has a K\"ahler structure, usually known as the transversal K\"ahler structure on the Sasaki manifold. The Reeb foliation $\mathcal{F}_\xi$ on $Z$ happens to be a taut Riemannian foliation, and the contact distribution has a splitting into holomorphic and anti holomorphic components. 
The K\"ahler form can be written as 
\begin{equation}
    2dx \wedge x\alpha + x^2 d\alpha
\end{equation}
where $\alpha$ is a contact form. Here $d\alpha$ is a transversal K\"ahler form that is non-degenerate on the contact distribution $TZ/T\mathcal{F}_{\xi}$.
This is studied broadly in complex and CR geometry and we refer to chapter 7 of \cite{boyer&galicki} for more details.
The K\"ahler form of a disc $2dr \wedge rd\theta$ corresponds to a \textit{trivial} Sasaki structure on the circle.

In the case where the boundary CR structure is pseudoconvex, it is known that the cohomology of the Dolbeault complex on a smooth manifold with boundary is finite dimensional in all degrees greater than $0$ (see Theorem 5.3.8 \cite{chen2001partial}). While local cohomology vanishes for positive degrees for fundamental neighbourhoods of the tangent space in the smooth case, this is not so in the singular space, as we showed in Proposition 7.1 of \cite{jayasinghe2023l2}.
There are many examples of spaces which have such higher cohomology groups, and some broadly studied examples are conifolds. Reeb vector fields and Lefschetz fixed point theorems for such spaces have been explored in the literature using orbifold resolutions in \cite{Martellisparksyau08,NekrasovABCDinsta}, and we studied the $L^2$ cohomology perspective in \cite[\S 7.3.4]{jayasinghe2023l2}.

A wedge symplectic form is a closed wedge 2 form, that is a smooth section $\omega$ of $\Lambda^2(^{w}T^*X)$ which satisfies $d\omega=0$.
On a collar neighbourhood 
\begin{equation}
\label{equation_collar_neighbourhood_asymptotics}
    \mathscr{C}(\mathfrak{B}_Y) \cong [0,1)_x \times \mathfrak{B}_Y
\end{equation}
of $\mathfrak{B}_Y$ in $X$, a closed wedge 2 form of the form 
\begin{equation}
    \omega_{w,pt}= dx \wedge x\alpha +x^2 \alpha +\phi^* \omega_Y
\end{equation}
is called a \textbf{\textit{product type wedge symplectic structure}}, where $\alpha$ is a contact form for the link $Z$ and $\omega_Y$ is a 2 form on $Y$.
In this article we consider \textbf{\textit{asymptotically $\delta$ wedge symplectic structures}} $\omega$ which satisfy
\begin{equation}
\label{equation_structure_of_the_symplectic_form}
    \omega - \omega_{w, pt} \in x^{2\delta} \mathcal{C}^{\infty} (\mathscr{C}(\mathfrak{B}_Y); \Lambda^2(^{w}T^*X))).
\end{equation}
for some $\delta>0$ in a collar neighbourhood as described in equation \eqref{equation_collar_neighbourhood_asymptotics}.

An integrable almost complex structure on the wedge tangent bundle is called a wedge complex structure.
Given a triple $(g_{w,pt},J_{w,pt},\omega_{w,pt})$ where $g_{w,pt}$ is a product wedge metric and $\omega_{w,pt}$ is a wedge symplectic structure, and $J_{w,pt}$ is a wedge complex structure such that $g_{w,pt}(J_{w,pt} \cdot, \cdot) = \omega_{w,pt} (\cdot, \cdot)$, we say that $J_{w,pt}$ is a \textit{\textbf{product type wedge complex structure}}. 

Using a connection in a collar neighbourhood as described in equation \eqref{equation_collar_neighbourhood_asymptotics} where we can extend the fibration $\phi$ and fix a splitting, we can write this in coordinates as  
\begin{equation}
    J_{w,pt}=\partial_x \otimes x\alpha - \frac{1}{x} \partial_{\xi} \otimes dx + J_{TZ/T\mathcal{F}_{\xi}} + J_Y.
\end{equation}
Here $x\alpha$ is a unit wedge covector field on $C(Z)$, $\frac{1}{x} \xi$ is the dual wedge vector field with respect to a product type metric $g_{w,pt}$ on the collar neighbourhood, $J_{TZ/T\mathcal{F}_{\xi}}$ is an almost complex structure on the distribution $TZ/T\mathcal{F}_{\xi}$ where $\mathcal{F}_{\xi}$ is the integrable foliation generated by the flow of the vector field $\xi$ on $Z$, and where $J_Y$ is a complex structure on $TY$ restricted to the stratum.
Given a metric product of $\mathbb{C}^k \times C(Z)$ where $Z$ has a Sasaki structure, the corresponding K\"ahler structure is a product type wedge complex structure, where $\alpha$ is a contact form and $\xi$ is the Reeb vector field for the Sasaki structure and $J_{TZ/T\mathcal{F}_{\xi}}$ is an almost complex structure on the contact distribution.

In this article we consider wedge complex structures $J$ where
\begin{equation}
\label{equation_structure_of_the_complex_structure}
    J - J_{w, pt} \in x^{2\delta} \mathcal{C}^{\infty} (\mathscr{C}(\mathfrak{B}_Y); End(^{w}TX)))
\end{equation}
for some $\delta>0$, and we call them \textbf{\textit{asymptotically $\delta$ wedge complex structures}}.

\subsection{Spin$^\mathbb{C}$ Dirac operator.}
\label{subsection_spin_c_Dirac}

We follow the conventions in \cite{wu1998equivariant}, referring to chapters 3 and 5 of \cite{duistermaat2013heat} for a more detailed construction of the spin$^\mathbb{C}$ Dirac operator where in the K\"ahler case this is also called the Dolbeault-Dirac operator.

Given a wedge K\"ahler metric, the complexified wedge tangent bundle $^{w} T_{\mathbb{C}}X$ has an orthogonal splitting into the holomorphic and anti-holomorphic tangent bundles as $^wTX^{1,0} \oplus ^wTX^{0,1}$ induced by the wedge complex structure.
Given $v \in \Gamma(^{w} T_{\mathbb{C}}X)$ where $^{w} T_{\mathbb{C}}X=^{w} TX \otimes_{\mathbb{R}} \mathbb{C}$,
we can write it as $v=v^{1,0}+v^{0,1}$ where $v^{1,0} \in \Gamma(^{w} TX^{1,0})$ and $v^{0,1} \in \Gamma(^{w} TX^{0,1})$. Then the Clifford action of the complexified Clifford algebra is defined as
\begin{equation*}
cl(v^{1,0})=\sqrt{2}(v^{1,0})^{*} \wedge, \quad cl(v^{0,1})=-\sqrt{2} i_{v^{0,1}},
\end{equation*}
where $(v^{1,0})^{*}=\xi^{0,1} \in \Gamma(^{w} T^*X^{0,1})$ corresponds to $v^{1,0}$ via the K\"ahler wedge metric $g_w$. Given a Hermitian bundle $E$ on $X$ equipped with a compatible connection, this extends to a Clifford action on $End(F)$ where $F= \oplus_{q=0}^n \Lambda^{q} (^{w} TX^{1,0}) \otimes E$ by $cl \otimes Id$, which we denote by $cl$ with abuse of notation. 

Given wedge vector fields $v, e \in \Gamma(^{w} T_{\mathbb{C}}X)$, we have the anti-commutation relation $\{cl(v), cl(e)\}=-2 g(v, e)$. The holomorphic Hermitian connection on $^{w}TX \otimes \mathbb{C}$ corresponds to the Levi-Civita connection on $^{w}TX$, and we choose a compatible connection on $F=\Lambda^{q} (^{w} TX^{1,0}) \otimes E$ which we denote by $\nabla^F$. Composing $\nabla^F$ with the Clifford action we obtain the spin$^\mathbb{C}$ Dirac operator $D$ acting on sections supported on $\mathring{X}$.
Given a (local) orthonormal frame of the wedge tangent bundle $\{e_1, e_2, \ldots, e_{2n} \}$, we can write
\begin{equation}
    D=\sum_{i=1}^{2n} cl(e_{i}) \nabla^F_{e_{i}}.
\end{equation}
It is easy to check that the symbol of this operator is given by Clifford multiplication and as a result it is wedge elliptic, and we refer the reader to \cite[\S 2]{jayasinghe2023l2} for more details.

If, for some $Y\in \mathcal{S}(X)$, we restrict to a collar neighborhood of $\mathfrak{B}_Y$ as in Subsection \ref{subsection_asymptotics_metric} where we have an asymptotically $\delta$ wedge metric, the spin$^\mathbb{C}$ Dirac operator takes the form
\begin{equation}
\label{temp_lah_di_dah}
cl(\partial_x)\nabla^F_{\partial_x} + \sum_i {cl}(\frac{1}{x}\partial_{z_i})\nabla^F_{\frac{1}{x}\partial_{z_i}} + \sum_j {cl}(\partial_{y_j})\nabla^F_{\partial_{y_j}}
\end{equation}
up to a differential operator in $x^{\delta}\text{Diff}^1_w(X;F)$.
Here $x$ is a boundary defining function for $\mathfrak{B}_Y$, and we recognize \eqref{temp_lah_di_dah} as a wedge differential operator of order one. We observe that $D$ in general is not a wedge differential operator because it is only of the form in equation \eqref{temp_lah_di_dah} up to operators $x^{\delta}\text{Diff}^1_w(X;F)$ where $\delta$ need not be an integer.

It is well known that in the K\"ahler case we can write the operator as $D=\sqrt{2}(\overline{\partial}+\overline{\partial}^*)$ where
\begin{equation}
\label{equation_d_bar_spin_c_correspondence_1}
\bar{\partial}=\frac{1}{2 \sqrt{2}} \sum_{i=1}^{2 n} cl(e_i-\sqrt{-1} J e_i) \nabla^F_{e_i}, \quad \bar{\partial}^{*}=\frac{1}{2 \sqrt{2}} \sum_{i=1}^{2 n} cl(e_i+\sqrt{-1} J e_i) \nabla^F_{e_i}.
\end{equation}
Here $J$ is a complex structure that tames the wedge K\"ahler form.
We will study self adjoint extensions of this symmetric operator in Section \ref{Section_Hilbert_complexes}.

\begin{remark}[Conformal invariance of Dirac operators]
\label{Remark_conformal_invariance_Dirac}
It is well known that Dirac operators 
are conformally convariant, and we can identify the local cohomology groups of Dirac operators when there are conformal changes of the metric, say by a factor of a function $f$. This is well known for the Dolbeault Dirac operator and even the $\overline{\partial}$ operator on smooth complex manifolds, where the space of holomorphic functions can be identified on conformal complex manifolds.

We refer to \cite[\S 4]{conformalspinHijazi_1986} for a study of the conformal covariance of the spin Dirac operator on smooth spaces, extending for more general twisted Dolbeault-Dirac operators.
In our singular setting, a similar analysis holds on $X^{reg}$ which is dense in $X$, which can be used to see the conformal invariance of Spin Dirac operators on $X$.   
\end{remark}

\subsection{Fundamental neighbourhoods of fixed points}
\label{subsection_local_structures}

In this article we study self maps with isolated fixed points and focus a lot on neighbourhoods of such fixed points.
Given any point $a$ of a stratified pseudomanifold $\widehat{X}$, we can find a neighbourhood $\widehat{U_a}$ which has a homeomorphism as in equation \eqref{equation_chartlike_map}
\begin{equation}
\label{fundamental neighbourhood}
   \widehat{\phi}:\widehat{U_a} \longrightarrow  \widehat{\phi}(\widehat{U_a}) \subset \mathbb{R}^k_{y} \times \widehat{Z}^+_{z}
\end{equation}
where the image is bounded.
We call such a neighbourhood $\widehat{U_a}$ a \textit{\textbf{fundamental neighbourhood of $a$}}. Here $\widehat{Z}_{z}$ is another stratified space in general, and $\widehat{Z}^+$ is the infinite cone over this link. We can choose $\widehat{U_a}$ and $\phi$ such that $\phi(\widehat{U_a})=\mathbb{D}^{k} \times \widetilde{C}(\widehat{Z})$, where by $\widetilde{C}(\widehat{Z})$ we denote the truncated cone $[0,1]_x \times \widehat{Z}_z / _\sim $ where the points at $\{x=0\}$ are identified. The restriction of a product type wedge metric to such a neighbourhood is as we discussed in Subsection \ref{subsubsection_wedge_metrics_related_structures}.

\begin{definition}
\label{definition_smooth_boundary_and singularities}
Given such a neighbourhood $\widehat{U_a}$, we identify it with a neighbourhood $\mathbb{D}^k_y \times {\widetilde{C}_x(\widehat{Z_z})}$. We call the set
\begin{equation}
    \partial \widehat{U_a}:=(\partial \mathbb{D}^k_y) \times {\widetilde{C}_x(\widehat{Z_z})} \cup \mathbb{D}^k_y \times \{x=1\} \times \widehat{Z}
\end{equation}
the \textit{\textbf{boundary of the fundamental neighbourhood}}
and we denote the pre-image of $\partial \widehat{U_a}$ under the blow-down map $\beta$ by $\partial {U_a}$, which we refer to as the \textbf{\textit{metric boundary}} of $U_a$. We refer to the set where $\rho_X|_{U_a}=0$ as the \textbf{\textit{singularities (singular set) of $U_a$}}.
\end{definition}

We observe that the metric boundary of such a fundamental neighbourhood has an open dense set that is smooth.
By the equivalence of Thom-Mather stratified spaces and manifolds with corners with iterated fibration structures, there exists a lift of $\widehat{\phi}$ to a map $\phi: U_a \rightarrow \phi(U_a)$, where $U_a$ is a manifold with corners with iterated fibration structures, such that $\widehat{\phi} \circ \beta=\beta \circ \phi$ where $\beta$ is the blow down map, and where $\phi$ is a diffeomorphism of manifolds with corners. \textbf{We will refer to $U_a$ as a fundamental neighbourhood of $a$ as well}, denoting the difference as needed by our notation. 
We refer to $\widehat{\phi}$ as a \textit{\textbf{diffeomorphism of the fundamental neighbourhood $\widehat{U_p}$}}. 

If $a$ is a singular point of $\widehat{X}$ contained on the stratum $Y$, and $\widehat{Z}$ is the link of $\widehat{X}$ at $a$, then the tangent cone of $\widehat{X}$ at $a$ is $T_aY \times \widehat{Z}^{+}$. Given a wedge metric on $\widehat{X}$, there is a canonical metric on the tangent cone of a point obtained by freezing coefficients and extending homogeneously.

We will denote the \textbf{\textit{resolved truncated tangent cone}} at $a$ by $\mathfrak{T}_aX$, the metric product of the unit ball in $T_aY$ and the resolved truncated cone $C_{[0,1]}(Z)=[0,1]\times Z$ which is topologically a manifold with boundary, but is metrically singular. The wedge metric obtained by freezing coefficients at $\beta^{-1}(a)$ is the metric induced on $\mathfrak{T}_aX$ which is homogeneous and is of product type.

Given a neighbourhood $\widehat{U_a} = \mathbb{D}^{k} \times {\widetilde{C}(\widehat{L})}$, where we have denoted the (singular) metric cone by ${\widetilde{C}(\widehat{L})}$ we denote the resolved space $U_a=[0,1]\times Z$ where $Z$ is the resolution of $\widehat{Z}=\mathbb{S}^{k-1} \star \widehat{L}$ which is the join of the links $\mathbb{S}^{k-1}$ of the cone $\mathbb{D}^k$ and the link $\widehat{L}$.
Indeed given $\widetilde{C}_{r_1}(\mathbb{S}^{k-1})\times \widetilde{C}_{r_2}(\widehat{L})$ where the factor $\widetilde{C}_{r_1}(\mathbb{S}^{k-1})=\mathbb{D}^k$ is smooth, the new radial variable on $\widetilde{C}(\widehat{Z})$ is $x=\sqrt{(r_1)^2+(r_2)^2}$. 

We refer to \cite[\S 2.1]{Jesus2018Wittensgeneral} for a discussion of the product of cone metrics being a cone metric (c.f., \cite[\S 3.2.1]{Jesus2018Wittens}).
where a refinement of the original stratification of $\widehat{U_a}$ and a refinement of the corresponding manifold with fibered boundary structure to the resolved space $U_a$ is presented for products of cones, which we present below.

\begin{definition}[Refined stratification]
\label{Definition_refined_stratification}
Given a fundamental neighbourhood $\widehat{U_a}=\mathbb{D}^k \times \widetilde{C}_{r_2}(\widehat{L})$ of a point $a$, we consider the stratification of the neighbourhood given by $\widetilde{C}_x(\mathbb{S}^{k-1} \star \widehat{L})=\widetilde{C}(\widehat{Z})$. 
Consider the case where $L$ is smooth. Then $U_a$ has a smooth stratum $\mathbb{D}^k \times (0,1) \times L$ and a singular stratum $\mathbb{D}^k \times L$.
We define the \textit{refined stratification} on $U_a$ as the one which has smooth part $(0,1)_{r_1} \times \mathbb{S}^k \times (0,1)_{r_2} \times L$. It has two strata of depth 1 given by $\{r_1=0\} \times (0,1)_{r_2} \times L$ and $(0,1)_{r_1} \times \mathbb{S}^k \times \{r_2=0\}$, and a stratum of depth 2 given by $\{r_1=0=r_2\}$.
It is easy to see how this definition can be extended inductively when $L$ is stratified.
\end{definition}

A simple example is when $U_a=\mathbb{D}^{k}$ is a smooth disc, whose refined stratification corresponds to viewing it as a truncated cone over the sphere, the cone point is a depth 1 stratum.

\begin{remark}[Convention for stratification on fundamental neighbourhoods]
    In the rest of this article, unless otherwise stated, whenever we take fundamental neighbourhoods $\widehat{U_a}=\widetilde{C}(\widehat{Z})$ with resolution $U_a=C(Z)$ of a point $a$ we will consider the refined stratification, which has an open dense stratum of depth $0$ given by the set
    \begin{equation}
    \label{equation_set_name_what}
        \{(x,z) \in [0,1)_x \times Z \hspace{2mm} | \hspace{2mm} x \in (0,1),  z \in Z^{reg} \}.
    \end{equation}
\end{remark}

The main reason for this convention is to simplify the notation and proofs. The main estimates that we prove for Dirac operators are proven on the dense set in \eqref{equation_set_name_what}. When defining certain domains for the Dolbeault complex globally on a space $X$, we will use the boundary defining functions corresponding to the global stratification. We will use the boundary defining functions of the \textit{original} stratification when defining local domains for sections supported on fundamental neighbourhoods in Subsection \ref{subsubsubsection_Neumann_boundary_condition}.

\begin{remark}
\label{Remark_join_metric_error_asymptotics}
    Given a fundamental neighbourhood ${U_a}=C_{r_1}(\mathbb{S}^{k-1})\times C_{r_2}(L)$ of $a$, where the factor $C_{r_1}(\mathbb{S}^{k-1})=\mathbb{D}^k$ is smooth as in the discussion above, the new radial variable on $C(Z)$ is $x=\sqrt{(r_1)^2+(r_2)^2}$. We observe that given functions $f$ which vanishes to order $\mathcal{O}(r_1^m)$ and $\mathcal{O}(r_2^{\delta})$ at $a$ for some $\delta>0$, it vanishes to order $\mathcal{O}(x^{\delta'})$ at $a$ for some $0 < \delta' = \min \{m , \delta \}$.    
\end{remark}

\subsection{Wedge K\"ahler Hamiltonians and circle actions}
\label{subsection_wegde_Hamiltonian_actions}

We now introduce the group actions in the wedge K\"ahler setting that we study, building on subsection \ref{subsection_asymptotics_metric}. It is known that the symplectic reduction of a smooth compact symplectic manifold will always yield a stratified pseudomanifold (see Theorem 2.1 of \cite{sjamaar1991stratified}), and so does K\"ahler reduction, and we study actions in a generality which includes actions coming from reductions on smooth spaces.
It is well known that K\"ahler circle actions with fixed points on smooth manifolds are Hamiltonian by Frankel's theorem, and this extends to the singular case that we study (see, e.g., \cite[\S 3.2]{Mazzeo_2015}). Therefore, if there is a K\"ahler circle action which has fixed points, it is a Hamiltonian action. 

\subsubsection{Morse functions and local normal forms.}

Let us review the normal form for K\"ahler Hamiltonian vector fields in neighbourhoods of smooth fixed points.
An $S^1$ action on a smooth K\"ahler manifold $X^{2n}$ induces an $S^1$ action on the tangent space $\mathbb{C}^n$ of an isolated simple fixed point. Since the action preserves the K\"ahler metric, it preserves the radial sets, and an $S^1$ action is induced on the sphere $S^{2n-1} \subseteq \mathbb{C}^n$ at any fixed radial distance from the fixed point. On the tangent space, this is generated infinitesimally by a holomorphic vector field which is of the form
\begin{equation}
    V_0=\sum_{j=1}^n \sqrt{-1} \gamma_j (z_j\partial_{z_j} - \overline{z}\partial_{\overline{z}_j})=\sum_{j=1}^n  \gamma_j\partial_{\theta_j}
\end{equation}
for local holomorphic coordinates $z_j$ on the tangent space where the $\gamma_j$ are all non-zero. Here the $\theta_j$ correspond to the angular coordinates in polar coordinates for the polydiscs where $z_j=r_je^{i\theta_j}$ is the standard representation of holomorphic functions in polar coordinates.
If the K\"ahler form on the tangent space is 
\begin{equation}
    \omega=\frac{\sqrt{-1}}{2} \sum_{j=1}^n dz_j \wedge \overline{dz_j}
\end{equation}
then the Hamitonian condition, $dh(\cdot)=-\omega (V_1, \cdot)$ yields that $h=h(0)+\sum_j \gamma_j |z_j|^2$. The $\gamma_j$'s are known as the \textit{\textbf{weights}} of the K\"ahler action on the tangent space of the fixed point.

Let us study the model case of a cone ${\widetilde{C}_r(\widehat{Z})}$ with a product type K\"ahler wedge metric, with K\"ahler wedge form $2dr \wedge r\alpha + r^2 d \alpha$ where $\alpha$ is a contact form for the Reeb vector field $W$ on the Sasaki link $Z$, with the K\"ahler Hamiltonian $h=r^2$. 
We observe that any vector field $W$ on the link $Z$ extends to the cone over $Z$ as $r (\frac{1}{r} W)$, that is a wedge vector field that vanishes to first order in the boundary defining function of the resolved tangent cone.
The circle action generated by the Reeb vector field $r \frac{1}{r}\widetilde{V}$ where $\widetilde{V}=\alpha^{\#}$ for a contact form $\alpha$ extends to $r \frac{1}{r}\widetilde{V}$ on the resolved tangent cone, and has an isolated fixed point at $r=0$ on the singular tangent cone, whence we call the fixed point \textit{\textbf{simple}} as in Definition 5.18 of \cite{jayasinghe2023l2}.

The following is the notion of stratified Morse function for which the de Rham Morse inequalities were proven in \cite{jayasinghe2023l2}, which we give for comparison.

\begin{definition}[stratified Morse function]
\label{definition_stratified_Morse_function}
Let $\widehat{X}$ be a stratified pseudomanifold with a wedge metric $g_w$ and a continuous function $f' \in C^0(\widehat{X})$ which lifts to a map $f \in C_{\Phi}^{\infty}(X)$ (see \eqref{equation_smooth_functions_on_stratified_spaces}) on the resolved manifold with corners, i.e., $f=f' \circ \beta$ and which is a Morse function when restricted to $\widehat{X}^{reg}$. We demand that 
the image of the set $|df|_{g_w}^{-1}(0)$ under the blow-down map $\beta$ consists of isolated points on the stratified pseudomanifold.

We call such points the \textbf{\textit{critical points of $f$}}. Moreover, at critical points $a$, we ask that there exist fundamental neighbourhoods $\widehat{U_a}=\widehat{U_{a,s}}\times \widehat{U_{a,u}}$ where the metric respects the product decomposition, with radial coordinates $r_s$ on $\widehat{U_{a,s}}$ and $r_u$ on $\widehat{U_{a,u}}$ such that restricted to $\widehat{U_{a,s}}$ (respectively $\widehat{U_{a,u}}$), $r_s$ (respectively $r_u$) is the geodesic distance to $a$, and such that the function $f'$ restricted to $\widehat{U_a}$ can be written as $r_s^2-r_u^2$. 
Then we say that $f'$ is a \textit{\textbf{stratified Morse function}}.
\end{definition}

We modify this as follows to get the functions we study in this article.

\begin{definition}[K\"ahler Hamiltonian Morse function]
\label{Definition_Kahler_Morse_function}

Let $\widehat{X}$ be a stratified pseudomanifold with a wedge K\"ahler structure $(g,J,\omega)$.
Let $h' \in C^0(\widehat{X})$ be a function which lifts to a function $h \in C_{\Phi}^{\infty}(X)$ (see \eqref{equation_smooth_functions_on_stratified_spaces}) on the resolved manifold with corners, i.e., $h=h' \circ \beta$ and which is a Morse function when restricted to $\widehat{X}^{reg}$. We demand that 
the image of the set $|dh|_{g}^{-1}(0)$ under the blow-down map $\beta$ consists of isolated points on the stratified pseudomanifold.
We call such points the \textbf{\textit{critical points of h}}.
We demand that there exists a wedge vector field $V$ satisfying $dh=-\iota_{V}\omega$ called the \textit{\textbf{Hamiltonian vector field}}, the flow of which generates a stratum preserving $S^1$ action on $\widehat{X}$ which lifts to one on $X$ preserving the wedge K\"ahler structure, which implies that $dh$ is a smooth wedge one form.

We assume that there is a decomposition of the resolved truncated tangent cone $\mathfrak{T}_aX$ given by
\begin{equation}
\label{equation_decomposition_tangent_cone_1}
    \mathfrak{T}_aX= \Pi_{j=1}^l C_{r_j}(Z_j)
\end{equation}
where $r_j$ is the distance to a given point from the cone point on each cone $C(Z_j)$. The metric induced on the tangent cone of $a$ can be obtained by freezing coefficients of the metric at $\beta^{-1}(a)$, and thus it is a well defined wedge metric on $U_a$, and we have the functions $r_j$ (with abuse of notation) on $U_a$.

Then we demand that the Hamiltonian $h$ has an expansion near each critical point $a$ restricted to a fundamental neighbourhood $U_a$ of the form $h=h_a + \mathcal{O}((r_1^2+r_2^2+...+r_l^2)^{c/2})$ where 
$c>2$ and $h_a=\sum_j \gamma_j r_j^2+C$ where $C$ is some real constant. We call $h_a$ \textit{\textbf{Hamiltonian on the tangent cone}}.

We will refer to the $\gamma_j$ as the weights of the action at the singularities.
We call the cone $U_{a,s}:=\Pi_{j: \gamma_j>0}C_{r_j}(Z_j)$ as the \textit{\textbf{stable factor/attracting factor}} of the tangent cone, and the cone $U_{a,u}:=\Pi_{j: \gamma_j<0}C_{r_j}(Z_j)$ as the \textit{\textbf{unstable factor/ expanding factor}} of the tangent cone.

We define the \textbf{\textit{Hamiltonian vector field on the tangent cone}} to be $\widetilde{V}=:\sum_j \gamma_j r_j (\frac{1}{r_j}V_j)$
where the $V_j$ are the Reeb vector fields on each link $Z_j$ with respect to the Sasaki structures corresponding to the wedge metric on $U_a$ obtained by freezing coefficients at $\beta^{-1}(a)$.
\end{definition}

We use the name Hamiltonian on the tangent cone since the function $h_a$ is indeed a well defined K\"ahler Hamiltonian Morse function on the tangent cone with the induced metric and K\"ahler form.

\begin{definition}
\label{definition_stable_unstable_radii}
In the setting of the definition above, we define the \textbf{\textit{stable radius}} at $a$ by $r^2_s:=\sum_{j: \gamma_j>0} \gamma_j r_j^2$, and the \textbf{\textit{unstable radius}} at $a$ by $r^2_u:=\sum_{j: \gamma_j<0} -\gamma_j r_j^2$.
Then we define the \textbf{\textit{normalized radius}} at $a$ by $r^2_n:=r^2_s+r^2_u$. Then $h_a=r_s^2-r_u^2$.
\end{definition}

\begin{remark}[Local extension to $\mathbb{C}^*$ action]
\label{remark_local_extention_c_star}
Given a self map $f_{\theta}(r_j,z_j)$ on the tangent cone generated by the flow at time $\theta$ of the Hamiltonian vector field, the map $f_{\theta}(s r_j,z_j)$ for $s \in \mathbb{R}$ gives an extension of the $S^1$ action for a $\mathbb{C}^*$ action, where $\lambda=se^{i\theta} \in \mathbb{C}^*$ acts on the tangent cone.    
\end{remark}

The asymptotics for the wedge K\"ahler structures discussed in Subsection \ref{subsection_asymptotics_metric} can be used to see that the Hamiltonian vector field $V$ is equal to the Hamiltonian vector field on the tangent cone $\widetilde{V}$ on a fundamental neighbourhood of the fixed point up to a wedge vector field vanishing at order $r_n^{1+\delta}$ for \textit{some} $\delta>0$ restricted to a fundamental neighbourhood, since $V=-J \nabla h$. Moreover the stratum preserving condition implies that $\iota_{d\rho_X^{\#}}dh|_{\rho_X}=0$.

\begin{remark}[Sign convention for Hamiltonian actions]
\label{Remark_sign_convention_Hamiltonians}
Here we follow the sign convention in \cite{witten1984holomorphic} (different from that in \cite{wu1998equivariant}) and take the Hamiltonian vector field to be $V=-J \nabla h$ (where we denote the gradient of $h$ by $\nabla h$), equivalently $dh=-\iota_V \omega$.

The difference in this sign convention is the source of the difference between the local deformed cohomology groups on page 326 of \cite{witten1984holomorphic}, and those in Proposition 3.2 of \cite{wu1998equivariant}.
\end{remark}

\begin{remark}[Nomenclature]
    Given a K\"ahler Hamiltonian Morse function, it is easy to see that the metric gradient flow of $h$ (as opposed to the symplectic gradient flow which gives the Hamiltonian flow) is attracting on cones where the weights $\gamma_j$ are positive and expanding on cones where the weights are negative. This is the reason we call the factors defined above attracting and expanding.
    
    The nomenclature is motivated from that used in \cite{jayasinghe2023l2} for fixed point sets, as well as the notions of stable and unstable manifolds in dynamical systems.
 
\end{remark}

It is clear that in the smooth setting, we can take the links to be circles and recover the normal form described in the smooth setting. A circle action corresponds to a family of self maps $f_{\theta}: X \rightarrow X$ where $\theta \in S^1$, that are generated by the positive time $\theta$ flow of the Hamiltonian vector field $V$. Then $f_{\theta}^{-1}=f_{-\theta}$  given by the reverse time flow of the vector field will be equal to the positive time flow of $-V$. The weights at the fixed points will have the same magnitudes but opposite signs for these two flows. 
The weights have a interpretation as the frequencies of a Hamiltonian system in the study of harmonic oscillators related to physical problems.
The following example demonstrates this in the smooth setting, which is the same one in Example \ref{Example_spinning_sphere_intro}.

\begin{example}[Rotation on the round 2 sphere]
\label{example_pitfall}
Consider a rotation on the round sphere about the axis joining the north and south poles. This is a K\"ahler action with respect to the standard K\"ahler structure of $\mathbb{CP}^1$. As in Definition \ref{Definition_Kahler_Morse_function}, we have a decomposition of the tangent space into attracting and expanding factors on a fundamental neighbourhood $U=U_s\times U_u$ at each fixed point. While one fixed point will see a clockwise rotation, the other will see an anticlockwise rotation so the $\gamma_j$ will have opposite signs at the two fixed points.
We can see this explicitly using standard spherical coordinates on the sphere. The volume form is $\omega=\sin (\phi) d\theta d\phi$ in spherical coordinates and the height function $h=\cos(\phi)$ is a Hamiltonian for the vector field $\partial_{\theta}$ generating the rotation. We can put complex coordinates $z$ at the north pole at $\phi=0$ and y=$1/z$ at the south pole. On the north pole the $\gamma$ factor is positive since the vector field $\partial_\theta=\beta^*(+i[z\partial_z-\overline{z}\partial_{\overline{z}}])$ at the tangent space of the fixed point. At the south pole we have $\partial_\theta=\beta^*(-i[y\partial_y-\overline{y}\partial_{\overline{y}}])$ and the weight is negative. If instead of the symplectic gradient of $h$, one takes the gradient with respect to the round metric, the north pole is strictly expanding while the south pole is strictly attracting. 
\end{example}

As discussed in Remark \eqref{remark_density_of_circle_actions} while we work out all the details for circle actions, the results extend to more general actions of compact connected Lie groups $G$ by well known considerations. Thus we end this section with a definition of a K\"ahler moment map in our setting.

\begin{definition}[K\"ahler moment map]
In the same setting as in Definition \ref{Definition_Kahler_Morse_function}, but for $G$ a general compact connected Lie group, we define a (stratified) \textit{\textbf{ K\"ahler moment map}} as a moment map $M$ in the usual sense on $X^{reg}$ which extends to a continuous function on $X$ which we denote
\begin{equation}
    M: X \rightarrow \mathfrak{g}
\end{equation}
where $\mathfrak{g}$ is the Lie algebra of $G$, where we demand that if $g \in G$ generates circle actions on $X^{reg}$ corresponding to a K\"ahler Hamiltonian Morse function on $X^{reg}$, they extend to K\"ahler Hamiltonian Morse function as in Definition \ref{Definition_Kahler_Morse_function}.
Then we denote a geometric endomorphism corresponding to an element $g \in S^1 < G$ as $T_g$.
\end{definition}
If the Dolbeault complex is twisted by coefficients of a holomorphic vector bundle $E$ with a Hermitian metric, we will consider group actions which lift to an action on the bundle, preserving the Hermitian metric.
\textit{\textbf{Unless otherwise stated, we will always assume that our spaces are connected, and that the Lie groups $G$ that we consider are connected.}}

\subsubsection{Infinitesimal actions and Lie derivatives}
\label{subsubsection_infinitesimal_actions_lie}

In this subsection we introduce the infinitesimal actions and Lie derivatives on sections of Hermitian bundles corresponding to K\"ahler actions of the type we study. Consider a resolution of a stratified pseudomanifold $X$ with a Hermitian bundle $E$, equipped with a compatible connection. Let $V$ be the Hamiltonian vector field corresponding to a 
K\"ahler Hamiltonian Morse function, where the action on $X$ lifts to a fiberwise linear action on $E$.
Given actions that induce self maps $f_{\theta}$, we prove our results for twisted complexes $E$ only when the action lifts to one on $E$ giving fiberwise linear isomorphisms. 
In our singular setting, as discussed earlier, we demand that the flow of the vector field preserves the strata 
and when we study Hermitian bundles $E$ on $X$ with a compatible connection $\nabla^E$ we assume that the circle action lifts to a Hermitian bundle map. We can define the Lie derivative of the vector field $V$ using its flow.

We follow the notation in \cite{wu1998equivariant} and denote the Lie derivative on $E$ by $\overline{L_V}:=\{\iota_V, \nabla^E\}$. Given a K\"ahler structure $(g,J,\omega)$ the condition that the flow of $V$ generates K\"ahler isometries is equivalent to $\overline{L_V} g,\overline{L_V} J,\overline{L_V} \omega$ all vanish as sections of the bundles on which $g,J,\omega$ are defined as sections.
Given the lifted action on $E$, we have a natural action on the sections of $E$ which sends a section $s$ to $g \circ s \circ g^{-1}$ for $g=e^{i\theta} \in S^1$. This action commutes with the operators $\overline{\partial}_E,\overline{\partial}^*_E, cl(V)$ (introduced in Subsection \ref{subsection_spin_c_Dirac}) and we have the infinitesimal generator of this action 
\begin{equation}
    L_Vs=- \lim_{\theta \rightarrow 0}\frac{g \circ s \circ g^{-1}-s}{\theta}
\end{equation}
where we use the notation introduced in  \cite[\S 3]{wu1998equivariant} (this is notated as $\hat{L}_V$ in \cite{mathai1997equivariant}).

Since the difference $\overline{L_V}-L_V$ is linear over $C_{\Phi}^{\infty}(X)$, it is given by a section of $End(E)$ over $X$.
We refer to \cite{mathai1997equivariant} for more details.

\subsection{Locally conformally totally geodesic wedge metrics}
\label{subsection_LCTGMetrics}

In order to motivate the conditions we impose on the metric at fundamental neighbourhoods of critical points, we study the following example refering to sections 7.1.2 and 7.3.5 of \cite{jayasinghe2023l2} for more details.

\begin{example}
Consider the singular algebraic variety $\widehat{V}$ given by $p=ZY^2-X^3=0$ in $\mathbb{CP}^2$ which admits the $\mathbb{C}^*$ action $(\lambda)\cdot [X:Y:Z]=[\lambda^2X:\lambda^3Y:Z]$. We consider the associated family of geometric endomorphisms on the Dolbeault complex for the trivial bundle. The action has one smooth fixed point at $[0:1:0]$ with holomorphic Lefschetz number $1/(1-\lambda^{-1})$. The other fixed point is at the singularity $a=[0:0:1]$.
On the affine chart $Z=1$, we have the variety given by the equation $y^2=x^3$, where $x=X/Z, y=Y/Z$. 
\end{example}

Consider the affine metric on the chart $Z=1$. Away from the singularity the variety $y^2=x^3$ can be parametrized by the normalization map 
\begin{equation}
    t \rightarrow (t^2,t^3)=(x,y)
\end{equation}
where $t=re^{i\alpha}$ and $\alpha \in [0, 2\pi]$.
Let us compactify the regular part of this space with the choice of boundary defining function $\rho=|x|=|t|^2$ for the boundary corresponding to the resolved manifold with boundary, where the pre-image of the boundary under the blowup-map is a ($T^{2,3}$ knotted) circle at $\rho=0$.
We can compute the pullback of the affine metric on $Z=1$ to the resolution of the singular space to see that it is a wedge metric on the resolved manifold with boundary (where the pre-image of the singular point corresponding to the real blow-up is a circle) as follows.

Let us denote $x=t^2=\rho e^{i\theta}$ where $\theta=2\alpha \in [0, 4\pi]$.
The pulled back metric is 
\begin{equation}
    4|t|^2 (dt \otimes d\overline{t}) (1+(9/4)|t|^2)=4|r|^2 (dr^2+r^2 4d\alpha^2)(1+(9/4)|r|^2)
\end{equation} 
which can also be written as 
\begin{equation}
    (dx \otimes d\overline{x})(1+(9/4)|x|)=(d\rho^2 +\rho^2 d\theta^2) (1+(9/4)|\rho|)=(d\rho^2+\rho^2 4d\alpha^2) k_c^2
\end{equation}
where $k_c^2=(1+(9/4)|\rho|)$.
We observe that this wedge metric is conformal to the metric on the tangent cone $(d\rho^2 +\rho^2 d\theta^2)$ obtained by freezing coefficients at the circle at $\rho=0$. 

Since the resolved truncated tangent cone at a fixed point is isometric to a fundamental neighbourhood with the metric obtained by freezing coefficients, we can compare the two spaces $(U_a,g_w)$ and $(U_a,g_a)$ where $g_w$ is the restriction of the metric on $X$ to the fundamental neighbourhood while $g_a$ is the metric obtained by freezing coefficients. Thus we say that the metric on a fundamental neighbourhood is conformal to that on the tangent cone if $g_w= (k^2_c) g_a$ where $k^2_c$ is a smooth function on the regular part of $U_a$, extending to a continuous function on $U_a$.

Thus the local cohomology groups of the Dolbeault-Dirac operator corresponding to the trivial bundle (for a given choice of domain) on the fundamental neighbourhood and the tangent cone can be identified. In the example of the cusp curve above, it is easy to see that the local cohomology groups for an attracting critical point at the singularity are for both the fundamental neighbourhood and the tangent cone are the $L^2$ bounded functions that admit Laurent series expansions in powers of $t$.
We refer to \cite[\S 7.1.2]{jayasinghe2023l2} for an extensive discussion on the local cohomology groups for choices of domains for the example above.

\begin{remark} 
We refer to Remark \ref{Remark_convention_product_type_metrics_only_for_local_complexes} on how to define local cohomology groups for fundamental neighbourhoods of critical points when the metric is not product type. There we observe that it is straightforward for fundamental neighbourhoods where the weights are all either positive or negative (i.e. either purely attracting or expanding). If there are weights of both types, when the metric is of product type we can identify the local cohomology with the cohomology if it were purely attracting, up to duality in the expanding factor. Thus as long as we can identify the local cohomology groups for an attracting factor for conformally related metrics where one of the metrics is product type, we can use duality to define local cohomology groups for the other metric when it is not product type.
\end{remark}

\begin{remark}
\label{Remark_special_conformal_functions}
    The action is generated on the tangent cone by $\widetilde{V}=\rho \frac{1}{\rho}\partial_{\theta}$, which for the global K\"ahler actions of the type that we study extends by assumption to a well defined vector field on the fundamental neighbourhood $U_a$. In the example above, on the fundamental neighbourhood it is generated by $V=\rho k_c \frac{1}{\rho k_c}\partial_{\theta}$, which shows that $V-\widetilde{V}=0$. This is because the conformal factor $k_c^2$ is a function of the radial variable $\rho$, which implies that $L_V k_c^2=0$ and $L_{\widetilde{V}} k_c^2=0$ and thus the Hamiltonian action generated by the weighted linear combination of Reeb vector fields on the tangent cone can be identified with that on a fundamental neighbourhood. In particular it is easy to see that the action commutes with the spin$^{\mathbb{C}}$ Dirac operators for both metrics.
    
    As discussed in the previous subsection, for general conformal factors we can only anticipate that the difference will be given by a wedge vector field $W$ that vanishes to order $\mathcal{O}(\rho^{\delta})$ where $\delta=1/2$ for the metric on the cusp curve above.
\end{remark}

\begin{definition}[Locally conformally product-type wedge metrics]
\label{definition_locally_conformally_pt_wedge}

Given a fundamental neighbourhood $(U_a,g_w)$ of an isolated critical point of a K\"ahler Hamiltonian Morse function equipped with a wedge metric $g_w$, we say that it is \textit{\textbf{locally conformally wedge product type}} if the metric $g_w$ is conformal to that on the tangent cone $(U_a,g_a)$ at the critical point (obtained by freezing coefficients as discussed above) by some conformal factor $k_c^2$ which is continuous on $X$ and smooth on $X^{reg}$, and is equal to one at the pre-image of the critical point $a \in \widehat{U_a}$ under the blow-down map, on $U_a$.

If in addition the function $k_c^2$ can be expressed as a function of the radial distance functions $r_j$ on the tangent cone (where we identify the fundamental neighbourhood with the resolved tangent cone as resolved stratified spaces), we say that the metric $g_w$ is \textit{\textbf{locally radially conformally wedge product type}}, or simply say that the \textit{\textbf{conformal factor is radial}}.
\end{definition}

For a given K\"ahler Hamiltonian Morse function on $X$, if the asymptotically $\delta$ wedge metric restricted to a small enough fundamental neighbourhood of $U_a$ of a critical point $a$ is conformal to the product type metric on the tangent cone at $a$ (obtained by freezing coefficients on the resolved link at the critical point), up to terms that are asymptotically $\delta_1$, we say that the metric is \textit{\textbf{locally conformally asymptotically $\delta_1$ wedge}} at $a$. Following the nomenclature in Subsection \ref{subsection_asymptotics_metric} if $\delta_1=1$ ($\delta_1=1/2$), we say that the metric is \textit{\textbf{locally conformally totally geodesic (exact) wedge}} at $a$.

We can identify the local cohomology groups of operators for twisted Dolbeault complexes on fundamental neighbourhoods with those on the tangent cone when the metric on the fundamental neighbourhood is conformal to that on the tangent cone. For instance, twisting by a square root bundle of the canonical bundle, one has the spin-Dirac complex which we study in subsection \ref{subsubsection_spin_Morse_inequalities}. In the smooth setting such identifications have been studied even without the K\"ahler structure in \cite[S 4]{conformalspinHijazi_1986}. Such identifications can be easily extended on twisted bundles restricted to the fundamental neighbourhood, for the wedge metric and the product type metric obtained by freezing coefficients.

\section{Operators and Dolbeault complexes}
\label{Section_Hilbert_complexes}

In this section we discuss Dolbeault complexes on stratified pseudomanifolds, their restrictions to certain subspaces of stratified pseudomanifolds, and equivariant sub-complexes of Dolbeault complexes.
First we review some facts on abstract Hilbert complexes, then set up domains for Dolbeault complexes associated to resolutions of stratified pseudomanifolds, both at the global and local levels, as well as their equivariant subcomplexes. 
We then discuss geometric endomorphisms 
on these complexes corresponding to group actions, then introduce various polynomial supertraces and duality results, building on the work in \cite{jayasinghe2023l2}.

\subsection{Abstract Hilbert complexes}
\label{subsection_abstract_hilbert_complexes}
We define Hilbert complexes following \cite{bru1992hilbert}. 

\begin{definition}
A \textbf{\textit{Hilbert complex}} is a complex, $\mathcal{P}=(H_*,\mathcal{D}(P_*),P_*)$, of the form:
\begin{equation}
    0 \rightarrow \mathcal{D}(P_0) \xrightarrow{P_0} \mathcal{D}(P_1) \xrightarrow{P_1} \mathcal{D}(P_2) \xrightarrow{P_2} ... \xrightarrow{P_{n-1}} \mathcal{D}(P_n) \rightarrow 0.
\end{equation}
where each map $P_k$ is a closed 
operator which is called the differential, such that:
\begin{itemize}
    \item the domain of $P_k$, $\mathcal{D}(P_k)$, is dense in $H_k$ which is a separable Hilbert space,
    \item the range of $P_k$ satisfies $ran(P_k) \subset \mathcal{D}(P_{k+1})$,
    \item $P_{k+1} \circ P_k = 0$ for all $k$.
\end{itemize}
\end{definition}

We will often denote such a complex by $\mathcal{P}=(H,\mathcal{D}(P),P)$ without explicitly denoting the grading, or by $\mathcal{P}=(H,P)$ where the domain is either clear by context, or denoted in some decorated notation for $P$.
We shall sometimes notate the complex as $\mathcal{P}(X)$ when the Hilbert spaces are sections of a vector bundle on the resolution of a stratified pseudomanifold $X$, and we say that the \textit{\textbf{Hilbert complex $\mathcal{P}(X)$ is associated to the space $X$}}.
The \textit{\textbf{cohomology groups}} of a Hilbert complex are defined to be $\mathcal{H}^k(\mathcal{P}):= ker(P_k)/ran(P_{k-1})$. We shall often use the notation $\mathcal{H}^k$, where the complex used is clear from the context and $\mathcal{H}^k(\mathcal{P}(X))$ when the space needs to be specified (including spaces with boundary when they come up later on). If these groups are finite dimensional in each degree, we say that it is a \textit{\textbf{Fredholm complex}}.

For every Hilbert complex $\mathcal{P}$ there is an \textit{\textbf{adjoint Hilbert complex}} $\mathcal{P}^*$, given by
\begin{equation}
\label{adjoint_complex}
    0 \rightarrow \mathcal{D}((P_{n-1})^*) \xrightarrow{(P_{n-1})^*} \mathcal{D}((P_{n-2})^*) \xrightarrow{(P_{n-2})^*} \mathcal{D}((P_{n-3})^*) \xrightarrow{(P_{n-3})^*} ... \xrightarrow{(P_{1})^*} H_0 \rightarrow 0
\end{equation}
where the differentials are $P_k^*: Dom(P^*_k) \subset H_{k+1} \rightarrow H_k$, the Hilbert space adjoints of the differentials of $\mathcal{P}$. That is, the Hilbert space in degree $k$ of the adjoint complex $\mathcal{P}^*$ is the Hilbert space in degree $n-k$ of the complex $\mathcal{P}$, and the operator in degree $k$ of $\mathcal{P}^*$ is the adjoint of the operator in degree $(n-1-k)$ of $\mathcal{P}$. The corresponding cohomology groups of $\mathcal{P}^*(H,P^*)$ are $\mathcal{H}^k(H, (P)^*) := ker(P^*_{n-k-1})/ran(P^*_{n-k})$.
For instance, in the case of the de Rham complex $(L^2\Omega^k(X),d_{\max})$, the adjoint complex $\mathcal{P}^*$ is the complex $\mathcal{Q}=(L^2\Omega^{n-k}(X),\delta_{\min})$, since the operators $d$ and $\delta$ are formal adjoints of each other.

The main complexes we focus on this section are twisted Dolbeault complexes where the Hilbert spaces are $H_q=L^2\Omega^{0,q}(X;E)$ with operators $P=\overline{\partial}_E$ with a choice of domain $\mathcal{D}(P)$, which we study later in this section. We can form a two step complex where the Hilbert spaces are $H^+ =\bigoplus_{q=even} H_q$, and $H^- = \bigoplus_{q=odd} H_q$.
This leads to a \textit{\textbf{wedge Dirac complex}} as introduced in Definition 3.3 \cite{jayasinghe2023l2}, 
\begin{equation}
    0 \rightarrow \mathcal{D}(D^{+}) \xrightarrow{D^+} \mathcal{D}(D^{-}) \rightarrow 0
\end{equation}
where $D^{\pm}$ is the spin$^\mathbb{C}$-Dirac operator restricted to the spaces, together with the domain for the operator $D= \sqrt{2} (P+P^*)$ given by
\begin{equation}
    \label{Domain_Dirac_first}
    \mathcal{D}(D)=\mathcal{D}(P) \cap \mathcal{D}(P^*).
\end{equation}
There is an associated \textbf{\textit{Laplace-type operator}} $\Delta_k = P_{k}^*P_k+P_{k-1}P_{k-1}^*$ in each degree, which is a self adjoint operator with domain
\begin{equation}
\label{Laplacian_P_type}
\mathcal{D}(\Delta_k) = \{ v \in \mathcal{D}(P_k) \cap \mathcal{D}(P_{k-1}^*) : P_k v \in \mathcal{D}(P_k^*), P^*_{k-1} v \in \mathcal{D}(P_{k-1}) \},
\end{equation}
and with nullspace
\begin{equation}
   \widehat{\mathcal{H}}^k(\mathcal{P}):= ker(\Delta_k) = ker(P_k) \cap ker(P_{k-1}^*).
\end{equation}
The Kodaira decomposition which we present below in Proposition \ref{Kodaira_decomposition} identifies this with the cohomology of the complex $\mathcal{H}^k(\mathcal{P})$.
We observe that this Laplace-type operator can be written as the square of the associated Dirac-type operator $D=(P+P^*)$, restricted to each degree to obtain $\Delta_k$, and that the domain can be written equivalently as
\begin{equation}
\label{Laplacian_D_type}
\mathcal{D}(\Delta_k) = \{ v \in \mathcal{D}(D) : D v \in \mathcal{D}(D) \}.
\end{equation}
The null space is isomorphic to the cohomology for Fredholm complexes.

\begin{proposition}
\label{Kodaira_decomposition}
For any Hilbert complex $\mathcal{P}=(H, P)$ we have the \textbf{\textit{weak Kodaira decomposition}}
\begin{center}
    $H_k=\widehat{\mathcal{H}}^k(H, P) \bigoplus \overline{ran(P_{k-1})} \bigoplus \overline{ran(P_k^*)}$
\end{center}
\end{proposition}
This is Lemma 2.1 of \cite{bru1992hilbert}. 

\begin{proposition}
\label{Fredholm_is_closed_range}
If the cohomology of a Hilbert complex $\mathcal{P}=(H_*, P_*)$ is finite dimensional then, for all $k$, $ran(P_{k-1})$ is closed and therefore $\mathcal{H}^k(\mathcal{P}) \cong \widehat{\mathcal{H}}^k(\mathcal{P})$. 
\end{proposition}
This is corollary 2.5 of \cite{bru1992hilbert}. The next result justifies the use of the term \textit{Fredholm complex}.
\begin{proposition}
A Hilbert complex $(H_k, P_k)$, $k=0,...,n$ is a Fredholm complex if and only if, for each $k$, the Laplace-type operator $\Delta_k$ with the domain defined in \eqref{Laplacian_P_type} is a Fredholm operator.
\end{proposition}
This is Lemma 1 on page 203 of \cite{schulze1986elliptic}. 
Due to these results, we can identify the space of \textit{harmonic elements}, or the elements of the Hilbert space which are in the null space of the Laplace-type operator, with the cohomology of the complex in the corresponding degree. We shall use the same terminology for non-Fredholm complexes which we study as well. 

For Fredholm complexes, the null space of the Laplacian is isomorphic to the cohomology of the complex since the operator has closed range.

\begin{proposition}
\label{Kernel_equals_cohomology}
A Hilbert complex $\mathcal{P}=(H, P)$, is a Fredholm complex if and only if its adjoint complex, $(\mathcal{P}^*)$ is Fredholm.
If it is Fredholm, then 
\begin{equation}
    \mathcal{H}^k(\mathcal{P}) \cong \mathcal{H}^{n-k}(\mathcal{P}^*).
\end{equation}
\end{proposition}

In particular, for operators with closed range, the reduced cohomology groups are the same as the cohomology groups and are isomorphic to the null space of the Laplace-type operator, in which case the decomposition in Proposition \ref{Kodaira_decomposition} is called the \textbf{\textit{(strong) Kodaira decomposition}}.

\begin{remark}
\label{remark_only_normal_pseudomanifolds}
    Since the set $\widehat{X}_{n-2}$ has measure $0$ with respect to a conic metric, the $L^2$ functions on $X^{reg}$ are the same for the space and its normalization and the Hilbert complexes on $\widehat{X}$ and its normalization can be canonically identified. \textbf{From now on we will study topologically normal pseudomanifolds unless otherwise specified}. 
    In Example 7.36 of \cite{jayasinghe2023l2} we went over how the holomorphic Lefschetz fixed point theorem can be computed in the case of a non-normal pseudomanifold as well as its normalization in detail, and it is easy to work out similar correspondences for the holomorphic Morse inequalities for normal and non-normal pseudomanifolds.
\end{remark}

\subsection{Hilbert complexes on stratified pseudomanifolds}
\label{subsection_Hilbert_complexes_stratified_pseudomanifolds}

The twisted spin$^{\mathbb{C}}$ Dirac operators corresponding to twisted Dolbeault complexes are not necessarily essentially self adjoint on singular spaces.
There are two canonical domains which are the minimal domain,
\begin{equation}
    \mathcal{D}_{min}(P_X)= \{ u \in L^2(X;F) : \exists (u_n) \subseteq {C}^{\infty}_c(\mathring{X};F) \text{ s.t. }
    u_n \rightarrow u \text{ and } (P_Xu_n) \text{ is } L^2-\text{Cauchy} \},
\end{equation}
where $P_X=\overline{\partial}_F$ acting on $L^2(X;F)$ where $F=\Lambda^{\cdot}(^{w}T^*X^{1,0}) \otimes E$ where $E$ is a Hermitian bundle, and the maximal domain,
\begin{equation*}
    \mathcal{D}_{max}(P_X)= \{ u \in L^2(X;F) : (P_Xu) \in L^2(X;F) \},
\end{equation*}
wherein $P_Xu$ is computed distributionally.
For wedge Dirac type operators, these domains satisfy the inclusions
\begin{equation}
\label{equation_domain_inclusions}
    \rho_XH^1_e(X;F) \subseteq \mathcal{D}_{min}(D_X) \subseteq \mathcal{D}_{max}(D_X) \subseteq H^1_e(X;F)
\end{equation}
where
\begin{equation*}
    H^1_e(X;F) = \{ u \in L^2(X;F) : Vu \in L^2(X;F) \text{ for all } V \in  \mathcal{C}^{\infty}(X;\prescript{e}{}TX) \}
\end{equation*}
is the edge Sobolev space introduced in \cite{Mazzeo_Edge_Elliptic_1}. In \cite{jayasinghe2023l2} we studied the VAPS domain, following \cite{Albin_2017_index}, and we generalize it in this article.
Recall that we use the notation
\begin{center}
    $\rho_X = \prod_{H \in \mathcal{M}_1(X)} \rho_H$
\end{center}
for a total boundary defining function. A \textbf{multiweight} for $X$ is a map
\begin{center}
    $\mathfrak{s} : \mathcal{M}_1(X) \rightarrow \mathbb{R} \cup \{ \infty \} $
\end{center}
and we denote the corresponding product of boundary defining functions by 
\begin{center}
    $\rho_X^{\mathfrak{s}} = \prod_{H \in \mathcal{M}_1(X)} \rho_H^{\mathfrak{s}(H)}$
\end{center}
We write $\mathfrak{s} \leq \mathfrak{s'}$ if $\mathfrak{s}(H) \leq \mathfrak{s'}(H)$ for all $H \in \mathcal{M}_1(X)$.

\begin{definition}[Algebraic  domain of exponent $\alpha$]
\label{Definition_algebraic_domain}

We define \textbf{decay rates} $\alpha$ to be multiweights $\alpha$ which are real numbers in $[0,1]$ that are constant on each collection of boundary hypersurfaces $H$ of $X$ which are connected.

Given the Dolbeault operator $P=\overline{\partial}_E$ acting on $L^2\Omega^{0,\cdot}(X;E)$ for a pseudomanifold $\widehat{X}$, which is associated to the formal Dirac operator $D=P+P^*$ acting on sections in $L^2(X;F)$ where $F=\Lambda^{\cdot}(^{w}T^*X^{1,0}) \otimes E$ as introduced above, and given a decay rate $\alpha$ we define the domain $\mathcal{D}_{\alpha}(P)$ by 
\begin{multline}
\label{decay_domains_for_complex}
\mathcal{D}_{\alpha}(P) = \{ u \in L^2(X;F) : \exists (u_n) \subset \rho^{\alpha}_X L^2(X;F) \cap \mathcal{D}_{max} (P)\\ 
\text{such that } u_n \rightarrow u \; \text{and} \;  (P u_n) \; \text{is} \; L^2- \text{Cauchy} \}
\end{multline}
which is the graph closure of $\rho^{\alpha}_X L^2(X;F) \cap \mathcal{D}_{max} (P_X)$, and refer to it as \textit{\textbf{the algebraic domain of exponent $\alpha$}}.
\end{definition}

The \textit{\textbf{VAPS domain}} for the Dolbeault complex that we studied in \cite{jayasinghe2023l2} is the same as the \textit{\textbf{algebraic domain of exponent $\alpha=1/2$}}. Following the definition for the domain of the Dirac operator in \eqref{Domain_Dirac_first}, we have the VAPS domain for $D=\sqrt{2}(\overline{\partial}_E+\overline{\partial}_E^*)$. We showed in \cite[\S 3.3]{jayasinghe2023l2} that the domain for the Dirac operator $D$ in \eqref{Domain_Dirac_first} matches the VAPS domain for the Dirac operator constructed in \cite{Albin_2017_index} without referring to the Dolbeault complex.
We present the following definitions.
\begin{definition}
\label{Witt_assumption}
The operator $(D_X, \mathcal{D}_{1/2}(D_X))$ is said to satisfy the \textbf{\textit{geometric Witt condition}} if 
\begin{equation}
    Y \in \mathcal{S}(X), y \in Y \implies \text{Spec}(D_{Z_y}) \cap (-\frac{1}{2},\frac{1}{2}) = \emptyset
\end{equation}
If instead, we only require
\begin{equation}
    Y \in \mathcal{S}(X), y \in Y \implies \text{Spec}(D_{Z_y}) \cap \{ 0 \} = \emptyset
\end{equation}
then we say that $(D_X)$ satisfies the \textit{\textbf{Witt condition}}.
\end{definition}

It is known that Dirac complexes which satisfy the geometric Witt condition are essentially self adjoint, and we refer the reader to \cite{Albin_2017_index} for a detailed discussion. If only the Witt condition is satisfied, then we can pick domains $\mathcal{D}_{\alpha}(P)$ including the minimal, maximal and VAPS domains for $P=\overline{\partial}_E$ (see, e.g., \cite{jayasinghe2023l2}).

It is easy to check that given $\mathcal{P}_{\alpha}(X):=(L^2\Omega^{0,\cdot}(X;E), \mathcal{D}_{\alpha}(P), P)$, when the Witt condition is satisfied, the complex has the adjoint complex $\mathcal{P}_{1-\alpha}^*(X):=(L^2\Omega^{0,n-\cdot}(X;E), \mathcal{D}_{1-\alpha}(P^*), P^*)$.

\subsubsection{Domains and boundary conditions for complexes on fundamental neighbourhoods}
\label{subsubsubsection_Neumann_boundary_condition}

In \cite{jayasinghe2023l2} we studied local domains and local complexes for the de Rham and Dolbeault complexes at isolated fixed points corresponding to the global VAPS domain on stratified pseudomanifolds with wedge metrics.
Here we extend this to the case of certain domains for complexes satisfying the Witt condition.

\begin{remark}[Convention]
\label{Remark_convention_product_type_metrics_only_for_local_complexes}
    We will only define local domains on truncated tangent cones of isolated fixed points corresponding to Hamiltonian K\"ahler actions. In particular the metrics on the neighbourhoods are \textit{\textbf{product type metrics}} obtained by \textit{freezing coefficients} on the set $\beta_X^{-1}(a') \subset X$ where $a'$ is the fixed point on $\widehat{X}$. 
    
    The main reason for this is the ease in defining product complexes for products of cones with product type metrics. It is easy to extend the notions of local complexes we define here to non-product type metrics when the weights $\gamma_j$ at the given fixed point are either all positive or all negative since it is clear to see that the definitions for local domains in \eqref{equation_local_complex_first_factor} and \eqref{equation_local_complex_second_factor} below can be extended to fundamental neighbourhoods which are purely attracting and expanding, respectively.

    When the truncated tangent cone of a critical point has both attracting and expanding factors, the cohomology of the expanding factor is dual to the cohomology if it were attracting.   
\end{remark}

Given a twisted Dolbeault complex $\mathcal{P}(X)=(L^2(X;F),\mathcal{D}_{\min}(P), P)$ where $P=\overline{\partial}_E$ on a smooth Riemannian manifold $X$, given a fundamental neighbourhood $U_a \subset X$ of a point $a \in X$, we study local complexes on the neighbourhood. We will denote the operator restricted to $U$ by $P_U$, denoting it by $P$ when it is understood from context that we are studying the operator of the complex on $U$.

We define the complex $\mathcal{P}_N(U)$ to be  
$(L^2(U;F),\mathcal{D}_{max}(P_U), P_U)$ where $P_U$ is the restriction of the operator $P$ to sections in $L^2(U;F)$, and
\begin{equation}
    \mathcal{D}_{max}(P_U)=\{ u \in L^2(U;F) : P u \in L^2(U;F) \}
\end{equation}
where $P u$ is defined in the distributional sense.
This fixes the domain for operators in the \textit{\textbf{adjoint Hilbert complex of $\mathcal{P}_N(U)$}} which we \textit{\textbf{denote by $(\mathcal{P}_N(U))^*$}}.
This induces a domain for the Dirac-type operator $D=P+P^*$.
We refer the reader to \cite[\S 5.2.3]{jayasinghe2023l2} for a more detailed exposition of the choices of domains in the smooth setting.

We now consider the case when $\widehat{U} \subset \widehat{X}$ is singular. Given the complex $\mathcal{P}_{\alpha}(X)=(L^2(X;F), \mathcal{D}_{\alpha}(P), P)$, we define the complex $\mathcal{P}_{\alpha,N}(U)=(L^2(U;F),\mathcal{D}_{\alpha,N}(P_{U}), P_{U})$ with the domain
\begin{equation}
\label{equation_local_complex_first_factor}
    \mathcal{D}_{\alpha, N}(P_U):= \text{graph closure of } \{ \mathcal{D}_{\max}(P_{U}) \cap \rho_X^{\alpha}L^2(U;F) \}
\end{equation}
where $\rho_X$ is the boundary defining function of the original stratification, as discussed in Subsection \ref{subsection_local_structures}.
Similarly, we can define the complex $\mathcal{Q}_{1-\alpha,N}(U)=(L^2(U;F),\mathcal{D}_{1-\alpha,N}(P^*_{U}), P^*_{U})$ with the domain
\begin{equation}
\label{equation_local_complex_second_factor}
    \mathcal{D}_{1-\alpha, N}(P^*_U):= \text{graph closure of } \{ \mathcal{D}_{\max}(P^*_{U}) \cap \rho_X^{1-\alpha}L^2(U;F) \},
\end{equation}
the adjoint complex of which is 
$\mathcal{Q}^*_{1-\alpha,N}(U)=(L^2(U;F),(\mathcal{D}_{1-\alpha}(P^*_{U}))^*, P_{U})$.

In the case where there is a metric product neighbourhood $U_s \times U_u$, it is clear that we can define the product complex, and we do this for neighbourhoods of isolated fixed points on the tangent cone as follows.

\begin{definition}   
\label{definition_local_complex}
Consider a \textbf{\textit{global complex}} $\mathcal{P}_{\alpha}(X)=(L^2(X;F), \mathcal{D}_{\alpha}(P), P)$ on a stratified pseudomanifold $\widehat{X}$ equipped with a wedge K\"ahler structure, K\"ahler circle action and stratified K\"ahler Hamiltonian Morse function $h$ where the K\"ahler circle action gives rise to a geometric endomorphism $T_{\theta}$ which has only isolated fixed points for generic values of $\theta$. Let $a$ be an isolated fixed point of $T_{\theta}$ with a fundamental neighbourhood of the resolved tangent cone $U_a=U_{a,s}\times U_{a,u}$ where we assume the metric respects this splitting as in as in Definition \ref{Definition_Kahler_Morse_function}. Then we define the \textbf{\textit{local complex of $\mathcal{P}_{\alpha}(X)$ at $U_a$}}, denoted $\mathcal{P}_{\alpha,B}(U_a)=(L^2(U_a;F),D_{\alpha,B}(P_{U_a}), P_{U_a})$ to be the product complex
\begin{equation}
    \mathcal{P}_{\alpha,N}(U_{a,s})\times \mathcal{Q}_{1-\alpha,N}(U_{a,u})
\end{equation}
where the two factors are as defined above. 
We refer to this as simply the \textbf{\textit{local complex}} when the global complex is clear by context.
\end{definition}

\begin{remark}[Domains defined on tangent cones with the refined stratification]
    As discussed in Subsection \ref{subsection_local_structures}, when considering an isolated fixed point with a fundamental neighbourhood $U_a =\mathbb{D}^k \times C(L)$, we will use the fact that $U_a$ is homeomorphic to $C(Z)$ where $Z$ is the resolution of the join $\mathbb{S}^{k-1} \star \widehat{L}$, equipped with a wedge metric over the link $Z$.
    Restricted to such a chart, the above discussion shows that the choices of domains are completely determined by the metric and the boundary defining functions corresponding to the \textbf{original} stratification on $U_a$ as opposed to the refined stratification introduced in Definition \ref{Definition_refined_stratification}. The boundary defining functions of the original stratification lift to continuous functions on the manifold with corners corresponding to the resolution of the fundamental neighbourhood with the refined stratification.    
    The multi-functions $\rho(H)^{\alpha(H)}$ are defined using the boundary defining functions corresponding to the original stratification, and we can use these in defining the domains.
\end{remark}

\subsubsection{Equivariant Hilbert complexes}
\label{subsection_equivariant_Hilbert_complexes}

Here we study the subcomplexes of Hilbert complexes obtained by restricting to eigensections of $\sqrt{-1}L_V$ with a fixed eigenvalue $\mu$, where $V$ is a Hamiltonian vector field which is Killing, and $L$ denotes the Lie derivative.
Given a twisted Dolbeault complex $\mathcal{P}=(H=L^2\Omega^{0,\cdot}(X;E),\mathcal{D}(P),P)$, where $P=\overline{\partial}_E$ where $X$ has a K\"ahler Hamiltonian Morse function corresponding to the Hamiltonian vector field $V$, we have that there are equivariant Hilbert subspaces $L^2_{\mu}\Omega^{0,q}(X;E)$ such that there is an orthonormal decomposition
\begin{equation}
\label{equation_decomposition_equivariant_eigs}
    L^2\Omega^{0,q}(X;E)=\oplus_{\mu \in I} L^2_{\mu}\Omega^{0,q}(X;E).
\end{equation}
where the indexing set $I$ corresponds to the eigenvalues $\mu$ of $\sqrt{-1}L_V$, and they can be indexed by integers up to shifts due the periodicity of the circle action (or since the Pontryagin dual of $S^1$ is $\mathbb{Z}$), where the shift accounts for various non-trivial twists $E$ (see Subsection \ref{subsubsection_spin_Morse_inequalities} for examples in the case of fractional twists of canonical bundles).

The existence of such decompositions is well known when $X$ is smooth as was utilized in \cite{wu1998equivariant}. In the singular setting, we can see that a decomposition as in equation \eqref{equation_decomposition_equivariant_eigs} exists for the complexes we focus on by the following argument.

Here $V$ is a wedge vector field since it generates an isometry for the wedge metric. 
Since it generates a K\"ahler isometry, $\sqrt{-1}L_V$ commutes with $P,P^*$ on $\mathcal{D}(D^2)$, and thus with $D$.
Since $D$ preserves eigenspaces of $D^2$ (with mixed degrees $q$) we see that $\sqrt{-1}L_V$ and $D^2$ are simultaneously diagonalizable, and we have a filtration of the eigenspaces of $D^2=\Delta$ given by $\mu$. If the eigenspaces of $D^2$ give an orthonormal decomposition of $ L^2\Omega^{0,q}(X;E)$, the decomposition in equation \eqref{equation_decomposition_equivariant_eigs} follows, and we will show that the complexes we mainly focus on in this article satisfy such decompositions in the proof of Proposition \ref{Proposition_equivariant_Fredholm_complexes}. In fact such decompositions follow from the Peter-Weyl theorem since the $S^1$ action induces a unitary action on the eigenspaces of the Laplace-type operator $D^2$.

\begin{definition}[Equivariant Hilbert subcomplex of eigenvalue $\mu$]
\label{Definition_equivariant_Hilbert_subcomplex}
Let $\mathcal{P}=(H,\mathcal{D}(P),P)$ be a twisted Dolbeault complex and $\sqrt{-1}L_V$ be as in the above discussion. 
Given an eigenvalue $\mu$ of $\sqrt{-1}L_V$, we define the \textit{\textbf{equivariant subcomplex of eigenvalue $\mu$}} to be $\mathcal{P}^{\mu}=(H_{\mu},\mathcal{D}^{\mu}(P),P)$ where $H_{\mu}$ is defined as in the discussion above and $\mathcal{D}^{\mu}(P):=\{s \in \mathcal{D}(P) \cap H_{\mu} \}$.
\end{definition}

As explained in the introduction, if the equivariant complexes $\mathcal{P}^{\mu}_{\alpha}$ are Fredholm for all $\mu$, then the complex 
$\mathcal{P}_{\alpha}$ is said to be \textit{\textbf{transversally Fredholm}}.

We observe that this definition applies to local complexes on local neighbourhoods $U_a$ of isolated fixed points defined by restriction as earlier, where we used the product type metric and the corresponding complex (see Remark \ref{Remark_convention_product_type_metrics_only_for_local_complexes}).

Consider the case of a K\"ahler Hamiltonian circle action generated by a global vector field $V$, and the Lie derivative on sections of the bundle $E$ which we denote by $L_V$.
We have the action at an isolated fixed point $a$, generated by the Hamiltonian vector field on the tangent cone $\widetilde{V}_a$. We know that $\sqrt{-1}L_{\widetilde{V}_a}$ commutes with the model operators on the local domains, analogously to the global case we discussed above.
Then given $\mathcal{P}_{\alpha,B}(U_a)$, we have the equivariant local complexes $\mathcal{P}^{\mu}_{\alpha,B}(U_a)$.

\begin{proposition}
\label{Proposition_equivariant_Fredholm_complexes}
Let $X$ be the resolution of a stratified pseudomanifold $\widehat{X}$ of dimension $2n$ with a K\"ahler wedge metric and a stratified K\"ahler Hamiltonian Morse function $h$ corresponding to an isometric $S^1$ action generated infinitesimally by a vector field $V$. Let $E$ be a Hermitian vector bundle on $X$, to which the action lifts, and consider the twisted Dolbeault complex $\mathcal{P}_{\alpha}(X)=(L^2\Omega^{0,\cdot}(X;E), \mathcal{D}_{\alpha}(P), P=\overline{\partial}_E)$. Let $\mu$ be an eigenvalue of $\sqrt{-1}L_V$, and let $U_a$ be a fundamental neighbourhood of a zero of $V$. Then 
\begin{enumerate}
    \item if $\alpha=1$ of $\alpha=1/2, \mathcal{P}_{\alpha}(X)$ is Fredholm and $\mathcal{P}^{\mu}_{\alpha,B}(U_a)$ is transversally Fredholm.
    \item if $\widehat{X}$ has only isolated singularities, then $\mathcal{P}_{\alpha}(X)$ is Fredholm and $\mathcal{P}^{\mu}_{\alpha,B}(U_a)$ is transversally Fredholm for any choice of domain $\mathcal{D}(P)$.
\end{enumerate}
Moreover the spectrum is discrete or these cases.
\end{proposition}

\begin{proof}
It is well known that the global Dolbeault complexes are Fredholm and have discrete spectrum for any domain for the case when $\widehat{X}$ has only isolated conic singularities (see, e.g. \cite{lesch1999differential}). 
For $\alpha=1/2$ in the case of general singularities, the global Dolbeault complexes are Fredholm by the results in \cite{Albin_2017_index}. For $\alpha=1$, the global complex is Fredholm since the complex and the cohomology can be defined independent of the metric (see Remark \ref{Remark_minimal_domain_equivalences}). In these cases we can find an orthonormal basis of eigensections and an orthonormal decomposition of $L^2\Omega^{0,q}(X;E)$.

Proposition 5.33 of \cite{jayasinghe2023l2} shows that local Dolbeault complexes are not Fredholm, but have a direct sum decomposition where the summands are finite dimensional joint eigenspaces of both the Laplace-type operator on the link $Z$ and the Laplace-type operator on $C(Z)$. Thus there is a representation of the circle action on the Hilbert spaces $L^2\Omega^{0,q}(U_a;E)$ (since the $S^1$ action commutes with the Laplace type operator on $C(Z)$) and the Peter-Weyl theorem implies the decomposition
\begin{equation}
\label{equation_decomposition_equivariant_eigs_local}
    L^2\Omega^{0,q}(U_a;E)=\oplus_{\mu \in I} L^2_{\mu}\Omega^{0,q}(U_a;E).
\end{equation}
In Proposition 5.33 of \cite{jayasinghe2023l2} (in particular see Remark 5.36 of that article), we showed that the local cohomology group on a cone $C(Z)$ with product type metric has a countable basis, as long as the Laplace-type operator on $Z$ is Fredholm, which it is in the setting of this proposition.

Moreover in this case the null space of the Laplace-type operator on such a cone $C(Z)$ is finite dimensional when restricted to the eigenspaces $\mu$ of the operator $\sqrt{-1}L_{\widetilde{V}_a}$ where $\widetilde{V}_a$ is the Hamiltonian vector field on the tangent cone. This can be seen as follows in the case of the VAPS domain for general singulariities.

By our assumption on actions being K\"ahler structure preserving and Killing, and that the fixed points are isolated, we observe that $\widetilde{V}_a$ restricts to a vector field on the link, the flow of which preserves the metric on the link. Thus the operator $\sqrt{-1}L_{\widetilde{V}_a}$ commutes with the Laplace type operator on the link $Z$.
We showed in Remark 5.36 of \cite{jayasinghe2023l2} that the null space of the Laplace-type operator on $C_x(Z)$ has a basis given by harmonic sections of the form $x^{a_{i,j}}s_{i}$ where $a_{i,j}$ is some real constant and $s_i$ are eigensections of the Laplace-type operator on $Z$. Moreover for each such $s_i$, there are finitely many $a_{i,j}$ such that $x^{a_{i,j}}s_{i}$ are harmonic.

Thus we see that
\begin{equation}
\label{equation_blah_2_3}
    \sqrt{-1}L_{\widetilde{V}_a}(x^{a_i}s_{i})=x^{a_i}\sqrt{-1}L_{\widetilde{V}_a}(s_{i})=\mu x^{a_i}s_i
\end{equation}
and there is a decomposition 
\begin{equation}
    L^2\Omega^{0,q}(Z;E)=\oplus_{\mu \in I} L^2_{\mu}\Omega^{0,q}(Z;E).
\end{equation}
given by the eigensections of the Laplace-type operator on $Z$.
Thus for each eigenvalue $\mu$, there are only finitely many linearly independent eigensections $s_{i}$ for which the equation \eqref{equation_blah_2_3} holds, and we can find a finite basis for the elements of the null space for the Laplace-type operator on $C_x(Z)$ that has eigenvalue $\mu$ with respect to $\sqrt{-1}L_{\widetilde{V}_a}$.

Similar arguments work for the case of isolated singularities as as well as for $\alpha=1$ where the Laplace-type operator on the link $Z$ is Fredholm (since $Z$ is a smooth manifold in the isolated case and by the considerations in Remark \ref{Remark_minimal_domain_equivalences} for $\alpha=1$), proving the two numbered statements.

The discreteness of the spectrum for the local complexes also follows from an explicit description of the spectrum in terms of Sturm-Liouville operators similar to that in Proposition 5.33 of \cite{jayasinghe2023l2}. For the global complex on a Witt space, the discreteness of the spectrum and Weyl laws were establishded in \cite{Albin_2017_index} for totally geodesic wedge metrics, and for the de Rham complex with the minimal/ maximal domains in \cite{Jesus2018Wittens} where a globalization result is used after the discreteness of the spectrum is established for the neighbourhoods in Proposition 14.2 of that article. Since the minimal domain of the Dolbeault complex (case of $\alpha=1$) corresponds to the same Laplacian as that for the minimal domain for the de Rham complex (upto a constant), the discreteness of the spectrum follows for that case.
\end{proof}

While we focus on the complexes in the above Proposition, we write most of the proofs of results in an abstraction that generalizes to other choices of domains once an analog of the Proposition above is proven for them.

\subsection{Geometric endomorphisms}
\label{subsection_geometric_endo_no_deform}

In this section we briefly review the notion of a geometric endomorphism. 
\begin{definition}
\label{geometric_endo}
An \textbf{\textit{endomorphism}} $T$  \textbf{\textit{of a Hilbert complex $\mathcal{P}=(H,\mathcal{D}(P),P)$}} is given by an n-tuple of maps $T=(T_0, T_1,...,T_n)$, where $T_k:H_k \rightarrow H_k$ are bounded maps of Hilbert spaces, that satisfy the following properties.
\begin{enumerate}
    \item $T_k(\mathcal{D}(P_k)) \subseteq \mathcal{D}(P_k)$
    \item $P_k \circ T_k= T_{k+1} \circ P_k$ on $\mathcal{D}(P_k)$
\end{enumerate}
Each endomorphism $T_k$ has an adjoint $T_k^*$. If each $T_k^*$ preserves the domain $\mathcal{D}(P_k^*)$, then we call $T^*=(T^*_0, T^*_1,...,T^*_n)$ the \textit{\textbf{adjoint endomorphism}} of the dual complex.
\end{definition}

The commutation condition (condition 2) ensures that the endomorphisms will induce a map on the kernel of the Dirac-type operators that preserves the grading.

\begin{definition}
An endomorphism $T^{\mathcal{P}}$ of a Hilbert complex of the form $\mathcal{P}=(L^2(X;F),\mathcal{D}(P),P)$ on a stratified pseudomanifold $\widehat{X}$ is a \textbf{\textit{geometric endomorphism}} if there is a smooth self map $f:X \rightarrow X$ and smooth bundle morphisms $\varphi: f^*F \rightarrow F$ such that 
\begin{equation}
    T^{\mathcal{P}}_f S=\varphi (f^* S)
\end{equation}
for sections $S \in L^2(X;F)$.
\end{definition}

The notation we use is standard 
but hides the fact that the definition of the geometric endomorphism involves $F, f^*F$.
In the setting of K\"ahler circle actions studied in this article we have a family of self maps $f_{\theta}$ labelled by $\theta \in S^1$, and we will use the notation $T^{\mathcal{P}}_{\theta}:=T^{\mathcal{P}}_{f_{\theta}}$, or simply $T_{\theta}$ when the complex $\mathcal{P}$ is clear by context. 

\begin{definition}
\label{definition_geometric_endomorphism_proper}
Given an isometric K\"ahler Hamiltonian circle action generated by a K\"ahler Hamiltonian wedge vector field $V$, which generates a family of geometric endomorphisms $T_{\theta}$ on a Dolbeault complex $\mathcal{P}=(L^2(X;F),\mathcal{D}(P),P)$, and if the operator $cl(V)$ preserves the domain of the Dirac-type operator $\mathcal{D}(P+P^*)$ given by \eqref{Domain_Dirac_first}, then we say that each $T_\theta$ is a geometric endomorphism of the complex.
\end{definition}

We show that the condition is satisfied by geometric endomorphisms for the domains studied in Proposition \ref{Proposition_equivariant_Fredholm_complexes}.

\begin{proposition}
In the same setting as Proposition \ref{Proposition_equivariant_Fredholm_complexes} for the global twisted Dolbeault complex $\mathcal{P}_{\alpha}(X)=(L^2\Omega^{0,\cdot}(X;E), \mathcal{D}_{\alpha}(P), P=\overline{\partial}_E)$, the endomorphisms $T_{\theta}$ are geometric endomorphisms of the complex.
\end{proposition}

\begin{proof}
In \cite[\S 4]{jayasinghe2023l2}, we showed that certain self-maps, including the group actions studied in this paper
induce endomorphisms on the Dolbeault complex with the VAPS domain. It is easy to see that the proof of \cite[Proposition 4.9]{jayasinghe2023l2}  generalizes to the case of the $\mathcal{D}_{\alpha}(P)$ domains that we study in this article.

Next we show that $(V^{1,0})^* \wedge$ and $\iota_{V^{0,1}}$ preserve the domains $\mathcal{D}_{\alpha}(P)$ and $\mathcal{D}_{\alpha}(P^*)$ respectively for any $\alpha$, which implies that $cl(V)$ preserves the domains $\mathcal{D}_{\alpha}(D)=\mathcal{D}_{\alpha}(P) \cap \mathcal{D}_{1-\alpha}(P^*)$.

Recall from the definition of the algebraic domain of exponent $\alpha$ for $P_X$ in \eqref{decay_domains_for_complex}, that it is the graph closure of $\rho^{\alpha}_X L^2(X;F) \cap \mathcal{D}_{max} (P_X)$ where $F=\Lambda^*((^wT^*X)^{0,1}) \otimes E$.

Since $dh$ is a smooth wedge form that is $L^{\infty}$ bounded, the wedge product $(V^{1,0})^* \wedge s$ is in $\rho^{\alpha}_X L^2(X;F)$ for $s \in \rho^{\alpha}_X L^2(X;F)$. Moreover since $(V^{1,0})^*=\overline{\partial} h$ (see equation \ref{remark_deformed_operators_identities_44}) it is $\overline{\partial}$ closed. Thus the Leibniz rule shows that $(V^{1,0})^* \wedge s \in \mathcal{D}_{max} (P_X)$ for $s \in \mathcal{D}_{max} (P_X)$.

We now consider the graph closure of $\mathcal{D}_{\max}(P) \cap \rho^{\alpha}_X L^2\Omega(X;E)$. If there is a sequence $\{u_n\}$ where $\{Pu_n\}$ is $L^2$-Cauchy, then so is $\{P_\varepsilon u_n\}$ where $P_\varepsilon=(P+\varepsilon \sqrt{-2} (V^{1,0})^* \wedge)$ since $(V^{1,0})^* \wedge$ is a bounded operator on $L^2(X;F)$, showing that $\mathcal{D}_{\alpha}(P)$ is preserved by $(V^{1,0})^*$. 

A variation of this argument replacing $P$ with $P^*$ and $(V^{1,0})^* \wedge$ with $\iota_{V^{0,1}}$ shows that $\iota_{V^{0,1}}$ preserves the domain $\mathcal{D}_{\alpha}(P^*)$ for any $P$. This proves the result.
\end{proof}

We observe that $f_{\theta}^{-1}=f_{-\theta}$. 
We denote the endomorphism induced on the adjoint complex by $f_{\theta}^{-1}$ as $T^{\star}_{\theta}$, which is $(T^{\mathcal{P}}_{\theta})^*$ and we shall study this in more detail in the next subsection.

\subsection{Hodge star and Serre duality }
\label{subsection_Serre_duality}

For the de Rham complexes the adjoint complex can be defined via the Hodge star operator which introduces a duality. In the case of the Dolbeault complex, the Hodge star operator intertwines a complex with its Serre dual complex which is different from the adjoint complex.

We begin by studying the de Rham complex briefly. 
Given $\mathcal{A}_{\alpha}(X)=(L^2\Omega^k,(\mathcal{D}_{\alpha})(d),d)$, 
it is easy to see that the adjoint complex is $(\mathcal{A}_{\alpha})^*(X)=(L^2\Omega^{n-k},\mathcal{D}_{1-\alpha}(\delta),\delta)$ where $\delta:=\star^{-1} d \star$, where $\star$ is the Hodge star operator which takes sections in $L^2\Omega^k$ to $L^2\Omega^{n-k}$. Moreover the domain for the Dirac type operator $D=d+\delta$ for the de Rham complex with domain $\mathcal{D}_{\alpha}(d)$ is given by $\mathcal{D}_{\alpha}(d) \cap \mathcal{D}_{1-\alpha}(\delta)$, which is also the domain for the Dirac type operator for the complex $(\mathcal{A}_{\alpha})^*(X)$. From this it is easy to see that the $\star$ operator intertwines the complexes. The corresponding duality is called Poincar\'e duality.

The duality corresponding to the isomorphism of complexes for twisted Dolbeault complexes given by the Hodge star operator is commonly referred to as Serre duality.
Let $E$ be a holomorphic vector bundle over a resolved complex pseudomanifold $X^{2n}$, equipped with a  Hermitian metric and a compatible connection. 
The Hermitian metric gives a conjugate-linear isomorphism $E \cong E^{\ast}$ between $E$ and its dual bundle. This induces a duality of forms valued in $E$ and $E^*$ and when $F=\Lambda^{p,0}X \otimes E$ (where $\Lambda^{p,q}X:= \Lambda^p (\prescript{w}{}{T^*X)^{1,0}} \otimes \Lambda^q (\prescript{w}{}{T^*X)^{0,1}}$)
\begin{equation}
    {\star}_E: \Omega^{p,q}(X;E) \rightarrow \Omega^{n-p,n-q}(X;E^*).
\end{equation}
The canonical bundle on the regular part $\widehat{X}^{reg}$ of the pseudomanifold $\widehat{X}$ with a complex structure is $K_X=\Lambda^{n,0}X$, which is a Hermitian line bundle. 
This is the dualizing sheaf in the smooth setting and $\Lambda^{n,n-q}X \otimes E$ is isomorphic to $\Lambda^{0,n-q}X \otimes K_X \otimes E$. For general singular spaces the algebraic dualizing sheaf becomes more complicated (see Part III, Section 7 of \cite{Hartshornebook}).
For smooth sections that are compactly supported on $\widehat{X}^{reg}$, we can write the adjoint operator of $\overline{\partial}_E$ as
\begin{equation}
\label{equation_dual_of_del_bar_serre}
    \overline{\partial}_F^*= (-1)^q \star_F^{-1} \circ \overline{\partial}_{K_X \otimes F^{\star}} \circ \star_F
\end{equation}
and the Laplace type operator is $\Delta_F=\overline{\partial}_F \overline{\partial}_F^*+\overline{\partial}_F^* \overline{\partial}_F$. 

Consider the complex $\mathcal{P}_{\alpha}(X)=(L^2\Omega^{0,\cdot}(X;F),\mathcal{D}_{\alpha}(\overline{\partial}_F),\overline{\partial}_F)$ where $F=\Lambda^{p,0}X \otimes E$. Then the adjoint complex is
\begin{equation}
    (\mathcal{P}_{\alpha})^*(X)=(L^2\Omega^{0,n-\cdot}(X;F),\mathcal{D}_{1-\alpha}(\overline{\partial}^*_F),\overline{\partial}^*_F).
\end{equation}
The Serre dual of the adjoint complex $(\mathcal{P}_{\alpha})^*(X)$ is defined to be
\begin{equation}
    ((\mathcal{P}_{\alpha})^*)_{SD}(X)=(L^2\Omega^{0,\cdot}(X;F^* \otimes K),\mathcal{D}_{1-\alpha}(\overline{\partial}_{F^* \otimes K}),\overline{\partial}_{F^* \otimes K})
\end{equation}
where $F^* \otimes K=\Lambda^{n-p,0} \otimes E^*$.
The isomorphism of these two complexes can be easily deduced from the identity \eqref{equation_dual_of_del_bar_serre}.
The Serre dual of the complex $\mathcal{P}_{\alpha}(X)$ is defined to be 
\begin{equation}
    (\mathcal{P}_{\alpha})_{SD}(X)=(L^2\Omega^{0,n-\cdot}(X;F^*\otimes K),\mathcal{D}_{\alpha}(\overline{\partial}^*_{F^* \otimes K}),\overline{\partial}^*_{F^* \otimes K})
\end{equation}
where we note that $L^2\Omega^{n,n-q}(X;F^*)=L^2\Omega^{0,n-q}(X;F^*\otimes K)$,
and it is easy to see that \textit{\textbf{Serre duality commutes with adjoints}}, that is the adjoint of the Serre dual of a complex is the Serre dual of the adjoint of a complex.

One can similarly define the Serre dual of local complexes. That is, given a local complex $\mathcal{P}_{\alpha,B}(U_a)=(L^2\Omega^{0,\cdot}(U_a;E),\mathcal{D}_{\alpha,B}(P_{U_a}),P)$ where $P=\overline{\partial}_F$ of $\mathcal{P}_{\alpha}$ above corresponding to a fixed point of a K\"ahler circle action corresponding to a geometric endomorphism $T_{\theta}$, we have the  Serre dual 
\begin{equation}
    (\mathcal{P}_{\alpha,B})_{SD}(U_a)=(L^2\Omega^{0,n-\cdot}(U_a;F^* \otimes K),\mathcal{D}_{\alpha,-B}(Q_{U_a}),Q)
\end{equation}
where $Q=\overline{\partial}^*_{F^* \otimes K}$.

\begin{remark}
\label{remark_domains_and_geometric_endomorphisms}
Since the product complex in Definition \ref{definition_local_complex} depends on the decomposition of the tangent cone into the stable and unstable cones, it depends on the choice of geometric endomorphisms $T_{\theta}$.
Given a circle action, here we denote by $-B$ the choice of domains (and therefore complexes) corresponding to the geometric endomorphism $T_{-\theta}:=T^{(\mathcal{P}_{\alpha})_{SD}(U_a)}_{-\theta}$ for $(\mathcal{P}_{\alpha})_{SD}(U_a)$. 
We denote the geometric endomorphisms on $(\mathcal{P}_{\alpha})_{SD}(U_a)$ and $(\mathcal{P}_{\alpha})(U_a)$ by $T_{\theta}$, with some abuse of notation.
\end{remark}

The choices of domains as clarified in the remark above can be explained as follows.
By the definition of adjoints, we see that for sections $u,v \in L^2(U_a;F)$, the geometric endomorphisms on $\mathcal{P}(U_a)=(L^2(U_a;F),\mathcal{D}_{\alpha,B}(P_{U_a}),P_{U_a})$ for $P=\overline{\partial}_F$ and the adjoint complex satisfy
\begin{equation}
    \langle T^{\mathcal{P}}_{\theta}u,v \rangle_{L^2(U_a;F)}=\langle u, (T^{\mathcal{P}}_{\theta})^*v \rangle_{L^2(U_a;F)}.
\end{equation}
We see that for sections $u=\sum_i u_i e_i, v=\sum_i v_i e_i$ of $L^2(U_a;F)$, where $\{e_i\}$ form a local orthonormal frame for $F$, 
\begin{equation}
\int_{U_a} \langle f_{\theta}^*u, v \rangle_{E} =\sum_{i,j} \int_{U_a} f_{\theta}^*u_j \wedge \overline{\star} v_j h^E_{ij},
\end{equation}
where $h^E_{ij}=h^E(e_i,e_j)$ is the Hermitian metric on $E$. Since we demand that the group action lifts to one that preserves the Hermitian metric on $E$, we have that 
\begin{equation}
\label{equation_compute_adjoint_endo_1}
\int_{U_a} f_{\theta}^*u_i \wedge \overline{\star} v_j h^E_{ij}=\int_{U_a} (f_{-\theta})^*(f_{\theta}^*u_i \wedge \overline{\star} v_j) h^E_{ij}=\int_{U_a} u_i \wedge (f_{-\theta})^*(\overline{\star} v_i) h^E_{ij}
\end{equation}
where $f_{\theta}^*$ indicates the pullback on forms. Thus we have
\begin{equation}
\label{equation_adjoint_endo_111}
    \langle T^{\mathcal{P}}_{\theta}u,v \rangle_{L^2({U_a};F)}=\langle u, (T^{\mathcal{P}}_{\theta})^*v \rangle_{L^2({U_a};F)}=\langle (\overline{\star} u),T^{\mathcal{P}_{SD}}_{-\theta} (\overline{\star}v) \rangle_{L^2({U_a};F^* \otimes K)}
\end{equation}
where
\begin{equation}
\label{equation_adjoint_endo_11}
    T_{-\theta}^{\mathcal{P}_{SD}}=(\overline{\star})^{-1} \circ  (T^{\mathcal{P}}_{\theta})^* \circ \overline{\star}
\end{equation}
where we use the fact that $\overline{\star}$ commutes with $(f_{-\theta})^*$ and the geometric endomorphisms we study for K\"ahler isometries, where $T_{\theta}^{SD}$ is the geometric endomorphism on the Serre dual complex for the self maps $f_{\theta}$. Thus the stable cone for $T^{\mathcal{P}}_{\theta}$ at a fixed point $a$ is the unstable cone for $T^{\mathcal{P}}_{-\theta}$ and $T_{-\theta}^{\mathcal{P}_{SD}}$.

This also shows why the choice of domain when restricting to a fundamental neighbourhood of a fixed point that was clarified in Remark \ref{remark_domains_and_geometric_endomorphisms} is the correct one if we want to relate traces of geometric endomorphisms on the cohomology of the complex and the cohomology of the Serre dual complex. Again we will simply denote these geometric endomorphisms as $T_{\theta}, T_{-\theta}$ when the complexes are clear by context.
We will study geometric endomorphisms for Witten deformed complexes in Subsection \ref{subsubsection_geometric_endo_witten_deformed}.

In \cite{jayasinghe2023l2} we mostly worked with the case of $\alpha=1-\alpha=0.5$ where much of this is simpler.
We refer the reader to the explicitly worked out example in \cite[\S 7.1.2]{jayasinghe2023l2} for an illustration of the minimal and maximal domains for the example of the \textit{cusp curve} where the Dolbeault-Dirac operator is not essentially self-adjoint.

\subsection{Polynomial Lefschetz supertraces}

After constructing the holomorphic Witten instanton complex for the K\"ahler Hamiltonian actions that we study the Morse inequalities follow from the application of results proven in \cite{jayasinghe2023l2} for abstract Hilbert complexes which we review here.

\begin{definition}
\label{L2_Lefschetz}
Let $X$ be a pseudomanifold with an Fredholm complex $\mathcal{P}=(H,\mathcal{D}(P), P)$ with an endomorphism $T$, we define the associated \textit{\textbf{Lefschetz polynomial}} to be
\begin{equation} 
L(\mathcal{P},T)(b):= \sum_{k=0}^n b^k tr(T_k|_{\mathcal{H}^k(\mathcal{P})}) \in \mathbb{C}[b]
\end{equation}
and the associated \textit{\textbf{Lefschetz number}} to be
\begin{equation} 
\begin{split}
L(\mathcal{P},T) &:= L(\mathcal{P},T)(-1) \\
 & = Tr(T|_{\mathcal{H}^{+}({\mathcal{D_P}})})- Tr(T|_{\mathcal{H}^{-}({\mathcal{D_P}})})
\end{split}
\end{equation}
where $T_k|_{\mathcal{H}^k(\mathcal{P})} :\mathcal{H}^k(\mathcal{P}) \rightarrow \mathcal{H}^k(\mathcal{P})$ is the map induced by the endomorphism $T$ on the $k$-th degree cohomology group $\mathcal{H}^k(\mathcal{P})$ and $T|_{\mathcal{H}^{\pm}({\mathcal{D_P}})}$ is the map induced on the cohomology of the associated Dirac complex $\mathcal{D_P}$.
\end{definition}

\begin{definition}
\label{definition_polynomial_Lefschetz_supertrace}
Let $\mathcal{P}=(H,\mathcal{D}(P),P)$ be a wedge elliptic complex where the associated Laplace-type operators in each degree have discrete spectrum and trace class heat kernels. Let $T$ be an endomorphism of the complex. For all $t \in \mathbb{R}^+$, we define the \textbf{\textit{polynomial Lefschetz heat supertrace}} as
\begin{equation}
    \mathcal{L}(\mathcal{P},T)(b,t)=\mathcal{L}(\mathcal{P}(X),T)(b,t):=\sum_{k=0}^n  Tr( b^k T_k e^{-t \Delta_k})
\end{equation}
and we call $\mathcal{L}(\mathcal{P},T)(-1,t)$ the \textit{\textbf{Lefschetz heat supertrace}} associated to the complex. Here, we use the notation $\mathcal{L}(\mathcal{P}(X),T)(b,t)$ when the complex $\mathcal{P}=\mathcal{P}(X)$ is associated to a pseudomanifold $X$.
\end{definition}

The following is Theorem 3.14 of \cite{jayasinghe2023l2}.

\begin{theorem}
\label{Lefschetz_supertrace}
Let $\mathcal{P}=(H,\mathcal{D}(P), P)$ be a Hilbert complex where the associated Laplace-type operators in each degree have discrete spectrum and trace class heat kernels. Let $T$ be an endomorphism of the complex. For all $t \in \mathbb{R}^+$
\begin{equation} 
\label{equation_with_the_b}
    \mathcal{L}(\mathcal{P},T)(b,t)=L(\mathcal{P},T)(b)+ (1+b) \sum_{k=0}^{n-1} b^k S_k(t)
\end{equation}
where
\begin{equation} 
\label{equation_error_in_Lefschetz_supertrace}
    S_k(t)=\sum_{\lambda_i \in Spec(\Delta_k)} e^{-t \lambda_i}  \langle T_k v_{\lambda_{i}}, v_{\lambda_{i}} \rangle
\end{equation}
where $\{v_{\lambda_i}\}_{i \in \mathbb{N}}$ are an orthonormal basis of co-exact eigensections of $\Delta_k$. In particular
\begin{equation}
\label{Heat_formula_all_t_P}
    L(\mathcal{P},T)= \mathcal{L}(\mathcal{P},T)(-1,t),
\end{equation}
and the Lefschetz heat supertrace is independent of $t$.
\end{theorem}

In \cite{jayasinghe2023l2} we used the result above to prove the strong form of the de Rham Morse inequalities once the Witten instanton complex was constructed and it will play a similar role in proving the holomorphic Morse inequalities.
The following is Proposition 3.16 of \cite{jayasinghe2023l2}.

\begin{proposition}[Duality]
\label{proposition_Lefschetz_on_adjoint}
Let $\mathcal{P}=(H,\mathcal{D}(P), P)$ be an elliptic complex of maximal non-trivial degree $n$ and let $T$ be an endomorphism. Let $\mathcal{P^*}$ be the dual complex and let $T^*$ denote the endomorphism induced on the dual complex. Then
\begin{equation}
\label{Kalman_filter}
   b^n \mathcal{L}(\mathcal{P}^*,T^*)(b^{-1},t)=\mathcal{L}(\mathcal{P},T)(b,t).
\end{equation}
In particular, we have the equality
\begin{equation}
    L(\mathcal{P},T)= (-1)^n L(\mathcal{P^*},T^*).
\end{equation}
\end{proposition}

The following result summarizes the main dualities we use in this article, and is a generalization of Proposition 7.5 of \cite{jayasinghe2023l2} where it was proven for the VAPS domain.

\begin{proposition}
\label{Proposition_duality_complex_conjugation_for_local_Lefschetz_numbers}
Let $\widehat{X}$ be a stratified pseudomanifold of dimension $2n$ with a wedge metric and complex structure and let $E$ be a Hermitian bundle. Let $\mathcal{P}_{\alpha}(X)=(L^2\Omega^{0,\cdot}_{\alpha}(X;E),\mathcal{D}_{\alpha}(P), P)$ be a transversally Fredholm complex where $P=\overline{\partial}_E$, so that $(\mathcal{P}_{\alpha})^*(X)$ 
is the adjoint complex and $(\mathcal{P}_{\alpha})_{SD}(X)$ is the Serre dual complex, as defined in Subsection \ref{subsection_Serre_duality}.
Let $T_{\theta}$ be a geometric endomorphism on $\mathcal{P}(X)$ corresponding to a K\"ahler circle action with isolated fixed points, including one at $a$ with a fundamental neighbourhood $U_a \subset X$, where we denote the geometric endomorphisms on each complex by $T_{\theta}$ by abuse of notation. Then we have that
\begin{equation}
    \mathcal{L}(\mathcal{P}^{\mu}_{\alpha}(X),T_{\theta})(b,t)=b^n \mathcal{L}((\mathcal{P}^{\mu}_{\alpha})^*(X),T_{\theta}^*)(b^{-1},t)=b^n\mathcal{L}((\mathcal{P}^{\mu}_{\alpha})_{SD}(X),T_{-\theta})(b^{-1},t).
\end{equation}
Similarly for local complexes as defined in Subsection \ref{subsection_Serre_duality}, we have
\begin{equation}
    \mathcal{L}(\mathcal{P}^{\mu}_{\alpha,B}(U_a),T_{\theta})(b,t)=b^n \mathcal{L}((\mathcal{P}^{\mu}_{\alpha,B})^*(U_a),T_{\theta}^*)(b^{-1},t)=b^n\mathcal{L}((\mathcal{P}^{\mu}_{\alpha, B})_{SD}(U_a),T_{-\theta})(b^{-1},t)
\end{equation}
where each choice of domain $B$ corresponds to the geometric endomorphism appearing in $\mathcal{L}$ (see Remark \ref{remark_domains_and_geometric_endomorphisms}).
\end{proposition}

\begin{proof}
The proofs are similar for both the global and local complexes. The first equality in each follows from Proposition \ref{proposition_Lefschetz_on_adjoint}. 

The second equality follows from the discussion in Subsection \ref{subsection_Serre_duality} which shows the isomorphism between adjoint complexes and the Serre dual complexes, together with the computations in equations \eqref{equation_compute_adjoint_endo_1}, \eqref{equation_adjoint_endo_111}, \eqref{equation_adjoint_endo_11} which clarify the relationship between the adjoint endomorphism and the geometric endomorphism on the Serre dual complex.

\end{proof}

\begin{remark}[Convergence of the generating series]
\label{remark_convergence_of_generating_series}

We defined the traces of endomorphisms $T_{s,\theta}$ in equation \eqref{equation_renormalized_trace_definition} where the endomorphisms do not necessarily arise as geometric endomorphisms on $X$. We showed in Remark \ref{remark_local_extention_c_star} that the action can always be extended locally to be a $\mathbb{C}^*$ action on the tangent cone

We showed how to define renormalized traces for geometric endomorphisms on the tangent cone when the series do not converge in \cite[\S 5.2.5]{jayasinghe2023l2}, where we showed that for $\mathbb{C}^*$ actions, the renormalization is not needed for $s=|\lambda|<1$ where $\lambda$ is the $\mathbb{C}^*$ action and that the power series converges.    
\end{remark}

\section{Witten deformation for the Dolbeault complex}

In this section we introduce the Witten deformed local and global Dolbeault complexes.
Given a stratified pseudomanifold $\widehat{X}$ equipped with a wedge metric 
and a stratified K\"ahler Hamiltonian Morse function $h$ on the resolution $X$, and given a Hermitian bundle $E$ on $X$ equipped with a compatible connection we can construct a twisted Dolbeault complex $\mathcal{P}_{\alpha}(X)=(H=L^2\Omega^{0,\cdot}(X;E),\mathcal{D}_{\alpha}(P),P=\overline{\partial})$, where \textbf{we denote $\overline{\partial}_E$ by $\overline{\partial}$ to simplify the notation}.
We define the Witten deformed 
operators
\begin{equation}
\label{remark_deformed_operators_identities_44}
P_\varepsilon=\bar{\partial}_{\varepsilon}:=e^{-\varepsilon h} \bar{\partial} e^{\varepsilon h}=\bar{\partial}+\sqrt{-2} \varepsilon (V^{1,0})^* \wedge, \quad \bar{\partial}_{\varepsilon}^{*}=e^{\varepsilon h} \bar{\partial}^{*} e^{-\varepsilon h}=\bar{\partial}^{*}-\sqrt{-2} \varepsilon \iota_{V^{0,1}}
\end{equation}
and their sum defines the corresponding deformation of the Spin$^\mathbb{C}$ Dirac operator up to a constant factor of $\sqrt{2}$. 
In the first subsection we study identities for the operators, which we will use to prove estimates for localizing eigensections, as well as proving the \textit{\textbf{spectral gap}} result in Proposition \ref{Proposition_model_spectral_gap_modified_general} for Witten deformation on tangent cones in the second subsection.

\subsection{Deformed Hilbert complexes}
\label{subsection_deformed_Hilbert_complexes}

Given a resolved stratified pseudomanifold $X$ equipped with a wedge K\"ahler metric 
and a stratified K\"ahler Hamiltonian Morse function $h$, and given a Hermitian bundle $E$ with a connection to which the Hamiltonian action lifts as a linear action, we have the Dolbeault complexes $\mathcal{P}_{\alpha}(X)=(H^{\cdot}=L^2\Omega^{0,\cdot}(X;E),\mathcal{D}_{\alpha}(P),P)$ where $P=\overline{\partial})$. 

We show that there is an equivalence of domains for the Witten deformed and undeformed operators on the equivariant complexes.
Multiplication by $e^{-\varepsilon h}$ is a homeomorphism $H_{\mu}^q=L^2_{\mu}\Omega^{0,q}(X;E) \rightarrow H_{\mu}^q$ that takes a section in the domain of any closed extension $\mathcal{D}^{\mu}_{\alpha}(P)$ to one in $e^{-\varepsilon h} \mathcal{D}^{\mu}_{\alpha}(P)=\mathcal{D}^{\mu}_{\alpha}(P_\varepsilon)$
and conjugates $P$ and $P_\varepsilon$.
We define the \textit{\textbf{Witten deformed complex}} $\mathcal{P}_{\alpha,\varepsilon}=(H^q, \mathcal{D}_{\alpha}(P_{\varepsilon}), P_{\varepsilon})$, and equivariant complexes are notated $\mathcal{P}^{\mu}_{\alpha,\varepsilon}$ since $P_{\varepsilon},P^*_{\varepsilon}$ commute with $\sqrt{-1}L_V$. Multiplication by $e^{-\varepsilon h}$ induces an isomorphism between the complexes $\mathcal{P}^{\mu}_{\alpha,\varepsilon}(X)$ and $\mathcal{P}^{\mu}_{\alpha}(X)=\mathcal{P}^{\mu}_{\alpha,0}(X)$ for $P=\overline{\partial}$.
This can be summarized using the commutative diagram
\[\begin{tikzcd}
	{...} && {H_{\mu}^q} && {H_{\mu}^{q+1}} && {...} \\
	\\
	{...} && {H_{\mu}^q} && {H_{\mu}^{q+1}} && {...}
	\arrow["P", from=1-1, to=1-3]
	\arrow["P", from=1-3, to=1-5]
	\arrow["P", from=1-5, to=1-7]
	\arrow["{P_\varepsilon}", from=3-3, to=3-5]
	\arrow["{P_\varepsilon}", from=3-1, to=3-3]
	\arrow["{P_\varepsilon}", from=3-5, to=3-7]
	\arrow["{e^{-\varepsilon h}}", from=1-3, to=3-3]
	\arrow["{e^{-\varepsilon h}}", from=1-5, to=3-5]
\end{tikzcd}\]
where the domains $\mathcal{D}^{\mu}_{\alpha}(P)$ and $\mathcal{D}^{\mu}_{\alpha}(P_{\varepsilon})$ can be canonically identified by the unitary conjugation as described above. For $\overline{\partial}^*$ the isomorphism is given by the inverse multiplication (that is by $e^{+\varepsilon h}$ rather than $e^{-\varepsilon h}$).

Given an isolated fixed point $a$ with a fundamental neighbourhood of the tangent cone $U_a$, we have the local complex 
$\mathcal{P}^{\mu}_{\alpha,B}(U_a)=(L^2_{\mu}(U_a;F),\mathcal{D}^{\mu}_{\alpha,B}(P_{U_a}), P)$ as given by Definition \ref{definition_local_complex}. Our discussion above extends to local complexes and deformed local complexes, allowing us to define $\mathcal{P}^{\mu}_{\alpha,B,\varepsilon}(U_a)=(L^2_{\mu}(U_a;E),\mathcal{D}^{\mu}_{\alpha,B}((P_{U_a})_\varepsilon), P_{\varepsilon})$. There is a commutative diagram for the local deformed and undeformed complexes as in the global case.

Given an equivariant subcomplex of the Dolbeault complex $\mathcal{P}_{\alpha}^{\mu}(P)$ and the Witten deformed subcomplex 
$\mathcal{P}_{\alpha,\varepsilon}^{\mu}(P)$ for an eigenvalue $\mu$ of $\sqrt{-1}L_V$, then the domains satisfy $e^{-\varepsilon h} \mathcal{D}_{\alpha}^{\mu}(P)=\mathcal{D}_{\alpha}^{\mu}(P_\varepsilon)$, and $e^{+\varepsilon h} \mathcal{D}_{\alpha}^{\mu}(P^*)=\mathcal{D}_{\alpha}^{\mu}(P^*_\varepsilon)$ giving an isomorphism $\mathcal{D}_{\alpha}^{\mu}(D) \cong \mathcal{D}_{\alpha}^{\mu}(D_\varepsilon)$.   
Restricted to any neighbourhood of a fixed point $U_a$, we have $\mathcal{D}_{\alpha,B}^{\mu}(P_U)\cong \mathcal{D}_{\alpha,B}^{\mu}((P_U)_\varepsilon)$, and $\mathcal{D}_{\alpha,B}^{\mu}(D_U) \cong \mathcal{D}_{\alpha,B}^{\mu}((D_U)_\varepsilon)$.
The Serre duality that we discussed in Subsection \ref{subsection_Serre_duality} extends to the deformed complexes.

\subsubsection{Geometric endomorphisms for Witten deformed complexes.}
\label{subsubsection_geometric_endo_witten_deformed}

In Subsection 6.5.1 of \cite{jayasinghe2023l2} we introduced Witten deformed geometric endomorphisms corresponding to geometric endomorphisms, for the case of the de Rham complex. Here we discuss the case for the Dolbeault complex when the geometric endomorphism $T_{\theta}$ corresponds to a K\"ahler isometric circle action.

Given a geometric endomorphism $T_{\theta}$ of $\mathcal{P}$, we define \textit{\textbf{Witten deformed geometric endomorphism}} on $\mathcal{P}_{\varepsilon}$ on $X$ to be $T_{{\theta},\varepsilon}=e^{-\varepsilon h} T_{\theta} e^{\varepsilon h}$, and the corresponding adjoint endomorphism $T^*_{{\theta},\varepsilon}=e^{\varepsilon h} T_{\theta}^* e^{-\varepsilon h}$, where $T_{\theta}=T_{{\theta},0}$ is the geometric endomorphism on the un deformed complex $\mathcal{P}$. A key observation here is that the geometric endomorphisms $T_{\theta}$ commute with multiplication by $h$, and therefore by $e^{-\varepsilon h}$, and so the geometric endomorphism for the Witten deformed and undeformed complexes are identical on sections of the Hilbert spaces. This is because pulling back by the family of self-maps corresponding to the geometric endomorphisms $T_{\theta}$ preserves the corresponding K\"ahler Hamiltonian $h$.

Thus our discussion in Subsection \ref{subsection_geometric_endo_no_deform} for the adjoint endomorphism applies even to the Witten deformed case.

\subsection{Bochner identities, localization and model operators}

Here we follow the notation introduced in Subsection \ref{subsubsection_infinitesimal_actions_lie} following that of \cite[\S 3]{wu1998equivariant}, and refer the reader to that article for more details.
Recall that $\overline{L_V}=\{\iota_V,\nabla^E\}$ is the Lie derivative of $E$ valued anti-holomorphic forms along $V$, where the fibers of $E$ over different points on the orbits of points are related by the lifted circle action, and that we denote by $\sqrt{-1}L_V$ the infinitesimal generator of the $S^1$ action on $L^2(X;F)=L^2\Omega^{0,\cdot}(X;E)$.

\begin{proposition}[Operator identities]
\label{Proposition_operator_identities_Gollum}
In the setting above we have the identities
\begin{equation}
\label{equation_operator_identities_1}
    D=\sqrt{2}\left(\bar{\partial}+\bar{\partial}^{*}\right),
\end{equation}
\begin{equation}
\label{equation_operator_identities_2}
    D_{\varepsilon}:= D + \sqrt{-1} \varepsilon cl(V) =\sqrt{2}(\bar{\partial}_{\varepsilon}+\bar{\partial}_{\varepsilon}^{*}),
\end{equation}
and
\begin{equation}
\label{equation_operator_identities_3}
\Delta_{\varepsilon}:=(D+\sqrt{-1} \varepsilon cl(V))^{2}=D^{2}+\varepsilon^2|V|^{2} + \varepsilon K
\end{equation}
where 
\begin{equation}
\label{equation_operator_identities_4}
    K:=\frac{1}{2} \sqrt{-1} \sum_{i=1}^{2 n} cl\left(e_{i}\right) cl\left(\nabla_{e_{i}} V\right)+\left.\sqrt{-1} \operatorname{tr} \nabla \cdot V\right|_{^wT^{0,1} X} -2\sqrt{-1} (\overline{L_V})
\end{equation}
where $\nabla( \cdot )$ 
is the divergence operator on vector fields on $X^{reg}$ which extends uniquely to an operator on wedge vector fields.
\end{proposition}

\begin{proof}
These identities are proven in the smooth setting in Lemma 3.1 of \cite{wu1998equivariant} by straightforward computations which we simply replicate in the singular setting (adjusting for the difference in our sign convention as described in Remark \ref{Remark_sign_convention_Hamiltonians}), where one uses the definitions of the Dirac operators in terms of the wedge Clifford algebra.
Equation \eqref{equation_operator_identities_1} follows from equation \eqref{equation_d_bar_spin_c_correspondence_1}, and equation \eqref{equation_operator_identities_2} follows from 
\begin{equation}
\bar{\partial}_{\varepsilon}=\bar{\partial}+\frac{1}{2 \sqrt{2}} \sum_{i=1}^{2 n} cl\left(e_{i}-\sqrt{-1} J e_{i}\right) (-\nabla^E_{e_i}h), \quad \bar{\partial}_{\varepsilon}^{*}=\bar{\partial}^{*}-\frac{1}{2 \sqrt{2}} \sum_{i=1}^{2 n} cl\left(e_{i}+\sqrt{-1} J e_{i}\right) (-\nabla^E_{e_i}h),
\end{equation}
and the fact that and that $J \operatorname{grad} h=V$.
To show equation \eqref{equation_operator_identities_3}, first note that
\begin{equation}
\label{Baila_gahapang_1}
\left(D+\sqrt{-1} \varepsilon cl(V)\right)^{2}=\left(D\right)^{2}+\sqrt{-1} \varepsilon \left\{D, cl(V)\right\}+ \varepsilon^2|V|^{2}
\end{equation}
and that
\begin{equation}
\label{Baila_gahapang_2}
\left\{D, cl(V)\right\}=\sum_{i=1}^{2 n}\left(\left\{cl\left(e_{i}\right), cl(V)\right\} \nabla^E_{e_{i}}+cl\left(e_{i}\right) cl\left(\nabla^E_{e_{i}} V\right)\right)=-2 \nabla^E_{V}+\sum_{i=1}^{2 n} cl\left(e_{i}\right) cl\left(\nabla^E_{e_{i}} V\right) 
\end{equation}
from which the identity in expression \eqref{equation_operator_identities_4} can be obtained by arguments given in the proof of Lemma 3.1 of \cite{wu1998equivariant} on $X^{reg}$, where we note that the multiplier operator notated by $r_V$ in that article is given by $\overline{L_V}-L_V$. 

A more detailed coordinate computation of this identity for smooth manifolds is given in Proposition 2.6 \cite{mathai1997equivariant} (see also equation 2.23 of that article) in different notation, and avoiding the Clifford algebra. This identity extends from $X^{reg}$ to $X$ since the individual terms extend.
\end{proof}

Understanding the \textit{Bochner type} identity given in equation \ref{equation_operator_identities_3} is crucial for localization. The biggest difference between the de Rham case that we studied in \cite{jayasinghe2023l2} vs the Dolbeault complex is that there are first order terms in $D^2_{\varepsilon}-D^2$ for $D=\sqrt{2} (\overline{\partial}+\overline{\partial}^*)$ as witnessed by the Bochner-type identity above.
In the introduction of \cite{mathai1997equivariant},
it is shown that if we denote the Witten deformed Laplacian for the de Rham Laplacian as $\Delta_{\varepsilon,dR}$, then
\begin{equation}
\label{equation_de_Rham_Dolbeault_difference}
    D^2_\varepsilon= \frac{1}{2}\Delta_{\varepsilon,dR} - \sqrt{-1} \varepsilon L_V
\end{equation}
in our notation.
In fact this observation was made in equation (15) of \cite{witten1984holomorphic}. 
We observe that $|V|_{^wTX}^2=|dh|_{^wT^*X}^2$ by the Hamiltonian condition, similar to the smooth case. We refer to the introduction and section 2 of \cite{mathai1997equivariant} for detailed explanations of these facts in the smooth setting. The same formal computations go through in our singular setting as well.

The key observation is that when restricted to the equivariant Hilbert spaces $H_{\mu}$ for fixed eigenvalues $\mu$ of $\sqrt{-1}L_V$, the first order operator $\sqrt{-1}L_V$ is a bounded operator. 
We can use this to prove the proposition below for equivariant subcomplexes given by Definition \ref{Definition_equivariant_Hilbert_subcomplex}.

\begin{proposition}
\label{Proposition_stratified_Morse_Laplace_structure}
Let $X$ be the resolution of a stratified pseudomanifold of dimension $2n$ with a K\"ahler wedge metric and a stratified K\"ahler Hamiltonian Morse function $h$ corresponding to an isometric $S^1$ action generated infinitesimally by a wedge vector field $V$, where the metric is locally conformally totally geodesic at fundamental neighbourhoods of critical points of $h$. Let $E$ be a Hermitian vector bundle on $X$, to which the action lifts, yielding a geometric endomorphism $T^{\mathcal{P}_{\alpha}(X)}_{\theta}=T_{\theta}$ on the Dolbeault complex $\mathcal{P}_{\alpha,\varepsilon}(X)$ introduced above, which we assume is 
Fredholm, and where local complexes induced at isolated critical points are transversally Fredholm. 
For any $\varepsilon \in \mathbb{R}$, given $s \in \mathcal{D}^{\mu}_{\alpha}(D_{\varepsilon}^2)$, we have
\begin{equation}
\label{Zhangs_Bochner_formula_1}
    ||D_\varepsilon s||^2_{L^2(X;F)}=\langle [D^2+\varepsilon K  +\varepsilon^2 |dh|^2] s,s \rangle_{L^2(X;F)}.
\end{equation}
where $F=\Lambda^*(\prescript{w}{}{(T^*X)^{0,1}} \otimes E)$ and $K$ is the bounded operator on $H_{\mu}$ given by \eqref{equation_operator_identities_4}.
\end{proposition}

\begin{proof}
We begin by expanding the left hand side of equation \eqref{Zhangs_Bochner_formula_1} to get 
\begin{equation}
\begin{split}
||D_\varepsilon s||_{L^2(X;F)}^2 & = \langle D_\varepsilon s, D_\varepsilon s \rangle_{L^2(X;F)} \\
 & = \langle Ds, Ds \rangle_{L^2(X;F)} + 2\varepsilon \langle Ds, \sqrt{-1}cl(V)s \rangle_{L^2(X;F)} + \varepsilon^2 \langle \sqrt{-1}cl(V)s, \sqrt{-1}cl(V)s \rangle_{L^2(X;F)}\\
& = \langle [D^2+\varepsilon^2 |V|^2]s, s \rangle_{L^2(X;F)}  + \varepsilon [\langle Ds, \sqrt{-1}cl(V)s \rangle_{L^2(X;F)} + \langle \sqrt{-1}cl(V)s, Ds \rangle_{L^2(X;F)} ]\\
& + \int_{\partial X} \langle i cl(d\rho_X)s,D s \rangle_F dvol_{\partial X}
\end{split}
\end{equation}
where we have used the Green-Stokes formula for $D$ and that $cl(d\rho_X)=\sigma_1(D)(d\rho_X)$ to get the boundary term.
We have that 
\begin{multline}
\label{equation_with_boundary_term_Dirac_deformed}
    \langle Ds, \sqrt{-1}cl(V)s \rangle_{L^2(X;F)}+ \langle \sqrt{-1}cl(V)s, Ds \rangle_{L^2(X;F)} \\
    =\langle s, [D \sqrt{-1}cl(V) + \sqrt{-1}cl(V) D]s \rangle_{L^2(X;F)}+ \int_{\partial X} \langle i cl(d\rho_X)s,\sqrt{-1}cl(V)s \rangle_F dvol_{\partial X}
\end{multline}
and thus the boundary contribution on $\partial X$ is 
\begin{equation}
    \int_{\partial X} \langle i cl(d\rho_X)s, (D + \varepsilon \sqrt{-1}cl(V))s \rangle_F dvol_{\partial X}=    \int_{\partial X} \langle i \sigma_1(D_\varepsilon)(d\rho_X) s, s' \rangle_F dvol_{\partial X}=0
\end{equation}
since $s'=D_\varepsilon s \in \mathcal{D}_\alpha(D)$ for $s \in \mathcal{D}_\alpha(D_\varepsilon^2)$ and the boundary integral vanishes any sections $s,s'$ in any self adjoint domain $\mathcal{D}(D_\varepsilon)$.

We see that $K:=[D \sqrt{-1}cl(V) + \sqrt{-1}cl(V) D]=\sqrt{-1} \left\{D,  cl(V)\right\}$ is the operator given by \eqref{equation_operator_identities_4}, where all the terms on the right hand side of that expression are bounded operators on on $H_{\mu}$ for any fixed $\mu$. This proves the theorem.
\end{proof}

\begin{remark}
\label{Remark_simplifies_essentially_self_adjoint_1}
Proposition 6.22 of \cite{jayasinghe2023l2} can be considered to be an analog of this proposition in the de Rham case, and in Remark 6.23 of that article we discuss how in the non-essentially self-adjoint case there can be complicated in general. In the setting of this article the above result shows that there are no such complications. 
\end{remark}

\begin{proposition}
\label{Propostion_growth_estimate_witten_deformed}
In the same setting as Proposition \ref{Proposition_stratified_Morse_Laplace_structure}, there exist constants $C>0$, $\varepsilon_0>0$ such that for any section $s \in \mathcal{D}_{\alpha}^{\mu}(D_\varepsilon)$ with $\text{supp}(s) \subset (X \setminus \cup_{a \in crit(h)} U_{a})$ and $\varepsilon \geq \varepsilon_0$, one has 
\begin{equation}
    ||D_{\varepsilon}s||_{L^2(X;F)} \geq C \sqrt{\varepsilon}||s||_{L^2(X;F)}.
\end{equation}
\end{proposition}

\begin{proof}
Let $C_1$ be the minimum value of $||dh||$ on $X \setminus \cup_{a \in crit(h)} U_{a}$.
Using the Bochner type formula \eqref{Zhangs_Bochner_formula_1},  it is easy to see that there exists a finite constant $K_1$ such that, for all $s \in \mathcal{D}_{\alpha}^{\mu}(D_\varepsilon^2)$,
\begin{equation}
||D_{\varepsilon} s||_{L^2(X;F)}^2  = \langle D_{\varepsilon} s, D_{\varepsilon} s \rangle_{L^2(X;F)} \geq (\varepsilon^2 C_1^2-\varepsilon |K_1|)||s||^2_{L^2(X;F)}.
\end{equation}
For any $C>0$ we may, by taking $\varepsilon$ sufficiently large, bound the right hand side by $\varepsilon C^2||s||^2_{L^2(X;F)}$. Finally as the inequality holds for all $s \in \mathcal{D}^{\mu}_{\alpha}(D_\varepsilon^2)$ it holds by density on the closure of $\mathcal{D}^{\mu}_{\alpha}(D_\varepsilon^2)$ in the graph norm of $\mathcal{D}^{\mu}_{\alpha}(D_\varepsilon)$.

\end{proof}

The harmonic sections of the model operator on the tangent cone of a fixed point can be understood using computations with separation of variables. We studied this for the undeformed complex with boundary conditions in \cite[\S 5.2.4]{jayasinghe2023l2}, and explained how a similar computation works for the Witten deformed Laplace type operator in the de Rham case in Proposition 6.26 of that article, which is a Fredholm complex. 
We prove the analog for the Dolbeault complexes we study in Proposition \ref{Proposition_model_spectral_gap_modified_general} below. 
We begin with a definition to capture the settings for which we need these results.

\begin{definition}
\label{definition_tangent_cone_truncations}
Given the infinite tangent cone 
$Z^+= [0,\infty) \times Z$
of a critical point $a$ of a stratified K\"ahler Hamiltonian Morse function $h$ on $X$, we can find a product decomposition $Z^+=V_s \times V_u$, and we can extend the Morse function from the truncated tangent cone $U_a$ to this space as $r_s^2-r_u^2$ using the stable and unstable radii introduced in Definition \ref{definition_stable_unstable_radii}.
For any $t >0$, we define $U^{t}_{a,s}:=\{r_s \leq t\} \subseteq V_s$, $U^t_{a,u}:=\{r_u \leq t\} \subseteq V_u$, and $U^t_{a}:=U^{t}_{a,s} \times U^{t}_{a,u}$.
In particular $U_a=U^1_a$. Given a complex $\mathcal{P}_{\alpha}(X)$, we define the complex $\mathcal{P}_{\alpha,B}(U^t_{a})$ similarly to how we defined it for the case where $t=1$, and all the constructions we did for $t=1$ can be replicated for general $t$.
We denote the infinite tangent cone $U^{\infty}_{a}$.
\end{definition}

This is an analog of Definition 6.25 of \cite{jayasinghe2023l2}, where we studied the de Rham complex. There we did not care about weights in the notion of stratified Morse functions since the metric could be perturbed (see Remark 6.21 of \cite{jayasinghe2023l2}).

We observe that while the local complexes $\mathcal{P}_{\alpha,B,\varepsilon}(U^t_{a})$ are only defined for $t \in [0,\infty)$ for $\varepsilon \geq 0$, when we restrict to $\varepsilon>0$ they are defined for $t=\infty$ as well, and the cohomology groups of all these complexes are isomorphic as discussed above. In the case of $t=\infty$, the key is that factors of $e^{-\varepsilon (r^2_n)}$ appear in the local cohomology (where $r^2_n=r_s^2+r_u^2$, the normalized radius in Definition \ref{definition_stable_unstable_radii}), giving sufficient decay at infinity for the harmonic forms to be $L^2$ bounded on the infinite tangent cone.

\begin{proposition}[Model spectral gap]
\label{Proposition_model_spectral_gap_modified_general}
In the setting of this subsection, consider the complex $\mathcal{P}^{\mu}_{\alpha,B,\varepsilon}(U^{t\sqrt{\varepsilon}}_{a})$ for some fixed $t \in [0,\infty]$, where we assume the complex is transversally Fredholm for $\varepsilon=0$. Let $\Delta_{a,\varepsilon}=(D^a_{\varepsilon})^2$ be the Laplace-type operator on $U_a$ for $\varepsilon>0$. 
If the positive eigenvalues of the Laplace type operator can be written as $\{\lambda^2_i\}_{i \in \mathbb{N}}$ for $\Delta_{a,1}$, then the positive eigenvalues for $\Delta_{a,\varepsilon}$ are $\{\varepsilon \lambda^2_i\}_{i \in \mathbb{N}}$. 
\end{proposition}

In the case of $t=\infty$, we note that $\mathcal{P}^{\mu}_{\alpha,B,\varepsilon}(U^{t\sqrt{\varepsilon}}_{a})=\mathcal{P}^{\mu}_{\alpha,B,\varepsilon}(U^{t}_{a})$. 

\begin{proof}
This statement follows similarly to the case of the de Rham case that we proved in Proposition 6.26 of \cite{jayasinghe2023l2}, after restricting to the equivariant Hilbert complexes, since the Laplace type operators are related by equation \eqref{equation_de_Rham_Dolbeault_difference}.
The spectral theory for $\varepsilon=0$, $t < \infty$ is covered in Proposition 5.33 of \cite{jayasinghe2023l2}, without restricting to equivariant sub complexes. 
\end{proof}

\section{Main technical results}
\label{Section_proof_of_main_result}

In the first subsection we study how the Witten deformed Dolbeault-Dirac operators on fundamental neighbourhoods of critical points of K\"ahler Hamiltonian Morse functions where the wedge metric is locally conformally totally geodesic, is \textit{well approximated} by the Dirac operator on the tangent cone at the critical point as $\varepsilon$ goes to $\infty$. This is used to prove key estimates for the global Dirac operator as $\varepsilon$ limits to $\infty$, and construct the holomorphic Witten instanton complex in the next subsection, the main result of this article.
The Morse inequalities are drawn as corollaries of this in the third subsection.
\textbf{While we study wedge metrics that are locally conformally totally geodesic, we will prove some of the technical estimates for locally conformally asymptotically $\delta_1$ wedge metrics which are asymptotically $\delta$ wedge K\"ahler metrics on spaces $X$,  for some $\delta_1 \geq 1$ and $\delta>0$ tracking some of the effects of these chosen constants.}

\subsection{Polynomial expansion of operators}

Following Definition \ref{definition_tangent_cone_truncations}, let $U^t_a=[0,t)_x \times Z$ be a fundamental neighbourhood of a critical point $a$ of $h$. We denote $U_a:=U^t_{a}$ for $t=1$. Freezing the coefficients of the metric tensor of the link $Z$ at the singularity, we can restrict the metric induced on the tangent cone $\mathfrak{T}_aX$ to the cone $U_a$, equipping it with a product type wedge metric $g_{a,pt}$, and it is clear that this can be identified with the truncated tangent cone at $a$. 
We study two spin$^\mathbb{C}$ Dirac operators on $U_a$, the first being \textbf{$D^X$ restricted to $U_a$ which we will continue to denote as $D^X$}. The second is the \textbf{model operator $D^{a}$}, which is the composition of the connection and the Clifford action for $U_a$ with the product type metric.
Due to our assumptions on the metrics, we have that 
\begin{equation}
\label{equation_operator_model_error}
    D^X=D^a+(x)^{\delta}W
\end{equation}
for asymptotically $\delta$ K\"ahler wedge metrics where $W$ is a first order wedge differential operator, as discussed in Subsection \ref{subsection_spin_c_Dirac}. However our local conditions imply that up to certain identifications of the complexes, we can write $D^X=D^a+(x)^{\delta_1}W$ for $W$ a first order wedge differential operator where $\delta_1 \geq 1$ (as opposed to $\delta>0$).

The product type wedge metric gives rise to the volume form $dvol_{a,pt}$ on $U_a$, and we have the associated space of $L^2$ bounded sections $L^2_a:=L^2(U_{a};F)$, 
the sections of which can be identified with those in the space $L^2_X(U_a;F)$ defined for the restriction of the K\"ahler metric $g_w$ on $X$ to $U_a$ since they are quasi-isometric.
We \textbf{define} the stratified smooth function $k(x,z)$ on $U_a$ by
\begin{equation}
\label{equation_defining_k}
    dvol_X=k(x,z)dvol_{pt}
\end{equation}
which also relates the inner products $\langle \cdot, \cdot \rangle_{L^2_{a}}$ and $\langle \cdot, \cdot \rangle_{L^2_X(U_a;F)}$. 
\begin{remark}
    In the smooth setting $k$ is the Jacobian corresponding to the determinant of the exponential map, properties of which appear in many technical computations of the proof in the smooth setting (see for instance the proof of Proposition 3.2 \cite{mathai1997equivariant}). 
    This factor is important in computations that relate the K\"ahler geometry of the local complex to the global complex, and the Hamiltonian vector field on the truncated tangent cone to that of the fundamental neighbourhood. This continues to be the case in the singular setting, since we have an isometric embedding of $L^2(U_{a};F)$ into $L^2_X(U_a;F)$, conjugating the group actions as well as the Dirac operators as we shall see below.
\end{remark}

Our assumptions on the asymptotics of the metric in subsection \ref{subsection_LCTGMetrics}
show that $k(0,z)=1$, and that $C_1 \leq k \leq C_2$ for some positive constants $C_1, C_2$. 
While $D^{a}$ is formally self-adjoint with respect to the $\langle \cdot, \cdot \rangle_{L^2_{a}}$ inner product on $U_a$, $D^X$ in general is not. However we observe that $k^{1/2} D^X k^{-1/2}$ is formally self adjoint with respect to $L^2_{a}$. 

For a metric $g_w$ on a fundamental neighbourhood that is locally asymptotically $\delta_1=1/2$ wedge we can write $g_w=g_{w,pt}+s$ where 
\begin{equation}
\label{equation_structure of the metric_middle}
    s \in x^{2\delta_1} \mathcal{C}^{\infty} (U_a; S^2(^{w}T^*X)))
\end{equation}
restricted to $U_a$ similar to equation \eqref{equation_structure of the metric} for the global metric.
We can write the volume forms in local coordinates as $\sqrt{det(g_w)}=dVol_x$ and $\sqrt{det(g_{w,pt})}=dVol_{pt}$. Since $(g_w)^{-1}(g_{w,pt}+s)=Id+(g_w)^{-1}(s)$ we can expand the formula for the determinant of $Id+(g_w)^{-1}(s)$ and take the square root to see that $k=1+\mathcal{O}(x^{\delta})$.

We consider the operators obtained by freezing coefficients of the operators $D^X$ and $k^{1/2} D^X k^{-1/2}$ at $\beta^{-1}(a)$ (where $\beta$ is the blow down map from $U_a$ to $\widehat{U_a}$).
The Leibniz rule yields
\begin{equation}
\label{equation_first_in_operator_expansion_prescale}
\begin{aligned}
    k^{1/2} D^X k^{-1/2}=D^X -\frac{1}{2 k}  [D^X, k]  
\end{aligned}
\end{equation}
and since $D^X$ has terms with $\nabla^F_{\partial_x}$, the second term is a zeroth order operator of order $\mathcal{O}(x^{\delta_1-1})$. This shows that the operator obtained by freezing coefficients at $\beta^{-1}(a)$ coincide for $\delta_1>1$ with $D^a$ (for both $D^X$ and $k^{1/2} D^X k^{-1/2}$), and coincide up to a bounded zero-th order operator for $\delta_1=1$.

We need to compare certain deformed operators as $\varepsilon$ varies. 
To that end we introduce the following notation.
For $\varepsilon \geqslant 1$, let $\mathrm{Q}_{\varepsilon}$ be an operator
\begin{equation}
\label{equation_mapping_error_estimate_2}
\mathrm{Q}_{\varepsilon}=\sum_{j=1}^{2n} b_{j}(\varepsilon, x, z)  {\nabla}^F_{e_{j}}+c(\varepsilon, x, z)
\end{equation}
where $b_{j}(\varepsilon, x, z), c(\varepsilon, x, z)$ are endomorphisms of $F$ which depend continuously on $(x, z)$, and the $e_j$ vectors form an orthonormal frame for the wedge tangent bundle. Assume that 
\begin{equation}
\label{equation_mapping_error_estimate_1}
    \mathcal{D}_{\alpha}(\mathrm{Q}_{\varepsilon}) \subseteq \mathcal{D}_{\alpha}.
\end{equation}
Assume there exist constants $C>0$ and $m_1>0, m_2 \geq 0$ such that for any $\varepsilon \geq 1$, $(x,z) \in U^{t\sqrt{\varepsilon}}_a$ for $x>0$ on $U^{t\varepsilon}_a$,
\begin{equation}
\begin{aligned}
& |b_{j}(\varepsilon, x, z)| \leq C|x|^{m_1} ; 1 \leq j \leq \dim(X), \\
& |c(\varepsilon, x, z)| \leq C|x|^{m_2} 
\end{aligned}
\end{equation}
We will then use the notation
\begin{equation}
\label{equation_mapping_error_estimate_3}
\mathrm{Q}_{\varepsilon}=\mathcal{O} (|x|^{m_1} \partial+ |x|^{m_2})
\end{equation}
following the notation introduced in equation (8.55) of \cite{bismut1991complex}.
We define the following scaling operator.
\begin{definition}[Scaling operator]
    Let $\varepsilon>0$. If $u \in L^2(U^{t}_a;F)$, let $S_{\varepsilon}u \in L^2(U^{t\sqrt{\varepsilon}}_a;F)$ be given by
    \begin{equation}
        S_{\varepsilon}u(x,z):=u(x/{\sqrt{\varepsilon}},z)
    \end{equation}
    where $(x,z) \in U^{t\sqrt{\varepsilon}}_{a}$ which we identify with a neighbourhood of the resolved tangent cone $\mathfrak{T}_aX$, and observe that $F$ can be trivialized along rays with fixed values $z$ to define the rescaling.
\end{definition}

\begin{lemma}
\label{Lemma_operator_estimates}
In the setting of the discussion above for an asymptotically $\delta$ wedge metric on $X$ that restricts to a locally conformally asymptotically $\delta_1$ wedge K\"ahler metric on a fundamental neighbourhood $U_a$ with $\delta_1 \geq 1$, we have that as  $\varepsilon \rightarrow \infty$
\begin{equation}
\label{equation_operator_estimates_1}
    S_\varepsilon k^{1/2} D^X k^{-1/2}S_{\varepsilon}^{-1} = \sqrt{\varepsilon} D^a + \varepsilon^{(1-\delta_1)/2} Err'
\end{equation}
where $Err'=\mathcal{O}(|x|^{\delta_1} \partial+|x|^{\delta_1-1})$. 
Moreover, we have that for the Hamiltonian vector field $V$, in the notation of Definition \ref{Definition_Kahler_Morse_function}
\begin{equation}
\label{equation_operator_estimates_2}
    S_\varepsilon k^{1/2} \varepsilon cl(V) k^{-1/2}S_{\varepsilon}^{-1} = \varepsilon^{1/2} cl(\widetilde{V}) + \varepsilon^{(1 - \delta)/2}\mathcal{O}(|x|^{1+\delta}).
\end{equation}
Then there exists $Err=\mathcal{O}(|x|^{\delta_1} \partial + |x|^{\delta_1-1} + |x|^{1+\delta})$ and $f=(1-\delta)/2$ such that
\begin{equation}
\label{equation_operator_estimates_3}
    S_\varepsilon k^{1/2} D_{\varepsilon}^X k^{-1/2}S_{\varepsilon}^{-1} =\sqrt{\varepsilon} D^a_{1} + \varepsilon^{f}Err
\end{equation}
where $\sqrt{\varepsilon} D^a_{1}= D^a_{\varepsilon}$.
\end{lemma}

\begin{proof}
Throughout the proof we will use the facts that the conjugation by scaling sends the multiplier operator $x$ to $\varepsilon^{-1/2}x$, and the operator $\nabla^F_W$ acting on smooth sections of $F$ on $U_a$ to $\varepsilon^{1/2} \nabla^F_{W}$ for $W$ for a frame for $^wT^*X$ on $U_a$, chosen in the neighbourhood similar to the frame used in \eqref{temp_lah_di_dah}.
In particular this shows that $\sqrt{\varepsilon} D^a_{1}= D^a_{\varepsilon}=S_{\varepsilon} D^a S^{-1}_{\varepsilon}$ on smooth sections, which also holds for sections in $L^2(U_a;F)$ by Proposition \ref{Proposition_model_spectral_gap_modified_general}.

We first observe that if the metric is locally product type, then it is isometric to the resolved truncated tangent cone and the operators $D^X$ and $D^a$ can be identified and $Err'=0$ for all $\varepsilon$.
If the metric is locally conformally product type (see Subsection \ref{definition_locally_conformally_pt_wedge}), then the operators $D^X$ and $D^a$ on $U_a$, the local complexes and their eigensections can be identified, and thus again $Err'=0$ and we get \eqref{equation_operator_estimates_1}.

If the metric is locally conformally asymptotically $\delta_1$ for $\delta_1 \geq 1$, 
we can show that $Err'=\mathcal{O}(|x|^{\delta_1} \partial + |x|^{\delta_1-1})$ as follows.
Conjugating the left hand side of equation \eqref{equation_first_in_operator_expansion_prescale} by the scaling operator yields
\begin{equation}
\label{equation_first_in_operator_expansion}
\begin{aligned}
    S_\varepsilon k^{1/2} D^X k^{-1/2}S_{\varepsilon}^{-1}u(x,z)=S_\varepsilon ((D^X u)(x{\sqrt{\varepsilon}},z))-\frac{1}{2 k(x/\sqrt{\varepsilon},z)} \Big( [D^X, k] \Big) \Big( x/{\sqrt{\varepsilon}},z \Big) u(x,z)
\end{aligned}
\end{equation}
and the discussion below equation \eqref{equation_first_in_operator_expansion_prescale} shows that the second term is of order $\varepsilon^{(1-\delta_1)/2}\mathcal{O}(x^{\delta_1-1})$. 
Moreover a similar argument shows that $k^{1/2}(x^{\delta_1} W)k^{-1/2}=\mathcal{O}(x^{\delta_1}\partial +x^{2\delta_1-1})$ which after conjugation by $S_\varepsilon$ is of order $\varepsilon^{(1-2\delta_1)/2}\mathcal{O}(x^{\delta_1-1})$. 
Recall that the truncated tangent cone is simply the fundamental neighbourhood $U_a$ with the product-type metric obtained by freezing coefficients on the link $\beta^{-1}(a)$, and that we can write $D^X=D^a+x^{\delta}W$ where $D^a$ is the model operator on the tangent cone and $W$ a first order wedge operator, where $D^a$ and $D^X$ are both operators on the space of smooth sections of $U_a$ as a resolution of a Thom-Mather stratified pseudomanifold with boundary. Thus the first term on the right hand side of \eqref{equation_first_in_operator_expansion} is $S_{\varepsilon} D^a S^{-1}_{\varepsilon}$ up to an operator in $\varepsilon^{(1-2\delta_1)/2}\mathcal{O}(x^{\delta_1-1})$.

The definition of the Hamiltonian vector field on the tangent cone (see Definition \ref{Definition_Kahler_Morse_function}) shows that $V-\widetilde{V}=x^{1+\delta}W_1$ for a smooth wedge vector field $W_1$. The operator $cl(V)= cl(\widetilde{V}) + x^{1+ \delta} cl(W_1)$ is a zeroth order operator and since $\widetilde{V}$ vanishes to first order in $x$, equation \eqref{equation_operator_estimates_2} follows. 
Equation \eqref{equation_operator_estimates_3} follows from the preceding statements and equation \eqref{equation_operator_identities_2}. 
\end{proof}

\begin{remark}
\label{Remark_general_special_considerations_poly_exp}
For locally conformally asymptotically $\delta_1$ wedge metrics that are asymptotically $\delta$ wedge metrics, the observations in Remark \ref{Remark_special_conformal_functions} 
can be used to show that if the conformal factor is radial, that we can take $Err=\mathcal{O}(|x|^{\delta_1} \partial + |x|^{\delta_1-1} + |x|^{1+\delta_1})$ in \eqref{equation_operator_estimates_3}, and we can take $f=(1-\delta_1)/2$ 
satisfying $f \leq 0$ for $\delta_1 \geq 0$.

We observe that since $D^X=D^a+x^{\delta}W$ for a first order wedge differential operator $W$, the difference between $k^{1/2}D^Xk^{-1/2}$ and $k^{1/2}D^ak^{-1/2}$ is given by $k^{1/2}(x^{\delta} W)k^{-1/2}=\mathcal{O}(x^{2\delta-1})$. This suggests that using the null space of $k^{1/2}D^ak^{-1/2}$ after freezing coefficients, as opposed to that of $D^a$ could be used as better candidates for the local cohomology of the deformed complex.
While we do not investigate these approaches in this initial article, these considerations were part of the motivation for developing the estimation scheme in this section that culminates with the construction of the holomorphic Witten instanton complex.    
\end{remark}

\begin{remark}
\label{Remark_operator_estimates_smooth}
In the smooth setting the analog of this Proposition follows from Theorem 8.18 in \cite{bismut1991complex}. There the result is proven for non-isolated fixed point sets of group actions. The K\"ahler property of a smooth metric on a smooth manifold can be used to show metrics are asymptotically $\delta=1$ wedge metrics as opposed to $\delta=1/2$ which is the case for general smooth metrics, and having Taylor expansions of the metric in coordinate charts where power series expansions are in integers is crucially used in the proof of this result (see Proposition 3.14 of \cite{voisin2002hodge}), no longer the case for general wedge metrics as seen in subsection \ref{subsection_LCTGMetrics} for the cusp curve.

Theorem 8.18 in \cite{bismut1991complex} is more technical in a different way since one has to keep track of the tangential and normal directions of the fixed point sets and scale the operators appropriately. The authors of \cite{wu1998equivariant} note in the proof of Proposition 3.3 of that article, that the proof simplifies considerably when the fixed point sets are zeros of Killing vector fields, which is certainly the case for any isolated fixed point where things are even simpler. Hence our result can be viewed as a generalization of the localization result in \cite{bismut1991complex} for isolated fixed points on stratified spaces with asymptotically $\delta$ wedge metrics.

\end{remark}

\subsection{The holomorphic Witten instanton complex}

In this section we prove estimates and results which hold for complexes with algebraic domains of exponent $\alpha$ for any fixed $\alpha \in [0,1]$, and thus will not indicate the $\alpha$ explicitly in the notation for the complexes.
We define $\mathcal{P}^{\mu}_{B}(U_{a}^{t_0})=(H_{\mu},\mathcal{D}^{\mu}_B(P),P)$ where $H_{\mu}=L^2_{\mu}\Omega^{0,q}(U_{a}^{t_0};E)$ with the product type metric (see Remark \ref{Remark_convention_product_type_metrics_only_for_local_complexes}).

Given $t_0 \in (0, \infty]$, choose $\chi_{a,t_0}: \mathbb{R}_s \rightarrow[0,t_0]$ to be a function in $C^{\infty}(\mathbb{R})$ such that $\chi_{a,t_0}(s)=1$ when $s<t_0/2$, and $\chi_{a,t_0}(s)=0$ when $s>3t_0/4$. At each critical point $a$, we define $t:=r_n^2=\sum |\gamma_j| r_j^2$ where the functions $r_j$ are the functions in Definition \ref{definition_stratified_Morse_function} corresponding to each critical point $a$. We observe that the functions $\chi_{a,t_0}(t)$ are invariant with respect to the operator $\sqrt{-1}L_{\widetilde{V}}$ on each $U_a$, since the functions $t$ are.
Given a harmonic form $\omega_{a} \in \mathcal{H}(\mathcal{P}^{\mu}_{B}(U_{a}^{t_0}))$, we define $\omega_{a,\varepsilon}:=\omega_a e^{-t \varepsilon} \in \mathcal{H}(\mathcal{P}^{\mu}_{B,\varepsilon}(U_{a}^{t_0}))$ which are in the null space of $D^a_\varepsilon$ (see Subsection \ref{subsection_deformed_Hilbert_complexes}). We define 
\begin{equation}
\label{equation_modify_forms_cutoff}
    \alpha_{a, t_0, \varepsilon} :=\left\|\chi_{a,t_0}(t) \omega_{a,\varepsilon}\right\|_{L^{2}(U_{a,z})}, \hspace{5mm} \eta_{a, t_0, \varepsilon}:=\frac{\chi_{a,t_0}(t) \omega_{a,\varepsilon}}{\alpha_{a, t_0, \varepsilon}},
\end{equation}
and
\begin{equation}
\label{definition_perturbed_basis}
    \mathcal{W'}(\mathcal{P}^{\mu}_{B,\varepsilon}(U_{a}^{t_0})):=\Bigg\{ \eta_{a, t_0, \varepsilon}=\frac{\chi_{a,t_0}(t) \omega_{a,\varepsilon}}{\alpha_{a, t_0, \varepsilon}} : \omega_{a,\varepsilon} \in \mathcal{H}(\mathcal{P}^{\mu}_{B,\varepsilon}(U_{a}^{t_0})) \Bigg\}
\end{equation}
where the forms $\eta_{a, t_0, \varepsilon}$ each have unit $L^2$ norm and are supported on $U_{a}^{t_0}$. In the case where $t_0=1$ we drop the subscript $t_0=1$ and denote the  corresponding forms $\alpha_{a, \varepsilon}, \eta_{a, \varepsilon}$, and the cutoff function $\chi_a$.

\begin{remark}
    It suffices to prove most estimates in neighbourhoods $U_a^{t_0}$ for some $t_0$ that is smaller than $1/|\mu|$ for eigenvalues $\mu \neq 0$ for the operator $\sqrt{-1}L_V$ (so as to control the $\varepsilon\sqrt{-1}L_V$ term appearing in $\Delta_{\varepsilon}$). In the smooth setting this is handled in \cite{bismut1991complex} by taking smaller cutoffs (see equation (9.2) of \cite{bismut1991complex}). Here we will prove the relevant statements for $t_0=1$ instead of arbitrary finite truncated tangent cones to keep the notation manageable. The arguments are similar for the general case.
\end{remark}

We can extend the forms defined above from the fundamental neighbourhood $U_{a}$ to $X$ by $0$ away from their supports.
We have the Witten deformed Dirac type operator $D_{\varepsilon}$, whose square is the Witten deformed Laplace type operator $\Delta_{\varepsilon}$.

Let $\widetilde{G_{\varepsilon,\mu}}$ be the vector space generated by the set $\left\{\mathcal{W'}(\mathcal{P}^{\mu}_{B,\varepsilon}(U_{a})) : a \in Cr(h)\right\}$ corresponding to all the harmonic sections $\omega_{a,\varepsilon}$ as above; $\widetilde{G_{\varepsilon,\mu}}$ is a subspace of $L^2_{\mu}\Omega^{0,\cdot}(U_a;E)$ since each $\eta_{a, \varepsilon}$ has finite length and compact support. 

We define $G_{\varepsilon,\mu}=\{ k^{-1/2}s | s \in \widetilde{G_{\varepsilon,\mu}} \}$, the elements of which are elements of $L^2_{\mu}\Omega^{0,\cdot}(X;E)$ as can be seen from the definition of $k$ in \eqref{equation_defining_k}. Moreover these sections are equivariant with respect to the K\"ahler Hamiltonian vector field corresponding to the K\"ahler metric on the neighbourhood (as opposed to that on the tangent cone).

Since $G_{\varepsilon,\mu}$ is finite dimensional (in particular closed), there exists an orthogonal splitting
\begin{equation}
    L^2\Omega^{0,\cdot}_{\mu}(X;E)= G_{\varepsilon,\mu} \oplus G_{\varepsilon,\mu}^{\perp}.
\end{equation}
where $G_{\varepsilon,\mu}^{\perp}$ is the orthogonal complement of $G_{\varepsilon,\mu}$ in $L^2\Omega^{0,\cdot}_{\mu}(X;E)$. Denote by $\Pi_{\varepsilon,\mu}, \Pi_{\varepsilon,\mu}^{\perp}$ the orthogonal projection maps from $L^2\Omega^{0,\cdot}_{\mu}(X;E)$ to $G_{\varepsilon,\mu}, G_{\varepsilon,\mu}^{\perp}$ respectively.
We split the deformed Witten operator $D^X$ by the projections as follows.
\begin{equation}
\label{equation_fourfold_operators}
D_{\varepsilon, \mu, 1} =\Pi_{\varepsilon,\mu} D^X_{\varepsilon} \Pi_{\varepsilon,\mu}, \hspace{3mm}
D_{\varepsilon, \mu, 2} =\Pi_{\varepsilon,\mu} D^X_{\varepsilon} \Pi_{\varepsilon,\mu}^{\perp}, \hspace{3mm} D_{\varepsilon, \mu, 3} =\Pi_{\varepsilon,\mu}^{\perp} D^X_{\varepsilon} \Pi_{\varepsilon,\mu}, \hspace{3mm} D_{\varepsilon, \mu, 4} =\Pi_{\varepsilon,\mu}^{\perp} D^X_{\varepsilon} \Pi_{\varepsilon,\mu}^{\perp}.
\end{equation}
In the following proposition and its proof, the norms and inner products for the $L^2$ forms with the product type metrics on $U_a$ are denoted $L^2_a$ (or $L^2_a(U_a;F)$), while those for the metric on $X$ are denoted by $L^2_X$ (or $L^2_X(X;F)$).

\begin{proposition}
\label{proposition_Zhangs_Morse_inequalities_intermediate_estimates}
In the setting described above, for locally conformally totally geodesic wedge metrics that are asymptotically $\delta$ wedge metrics for some $\delta>0$, with $f=(1-\delta)/2$, $(f<1/2)$ 
we have the following estimates.
\begin{enumerate}
    \item There exist constants $\varepsilon_{0}>0, I_0>0$ such that for any $\varepsilon>\varepsilon_{0}$ and for any $s \in \mathcal{D}^{\mu}(D_{\varepsilon})$, $$||D_{\varepsilon,\mu, 1}s||_{L^2_X} \leq I_0 \varepsilon^f \|s\|_{L^2_X}.$$ 
    \item There exist constants $\varepsilon_{1}>0, I_1>0$ such that for any $s \in G_{\varepsilon,\mu}^{\perp} \cap \mathcal{D}^{\mu}(D_{\varepsilon}), s^{\prime} \in G_{\varepsilon,\mu}$, and $\varepsilon>\varepsilon_{1}$,
$$
\begin{aligned}
\left\|D_{\varepsilon,\mu, 2} s\right\|_{L^2_X} & \leq    I_1 \varepsilon^f {\|s\|_{L^2_X}}\\
\left\|D_{\varepsilon,\mu, 3} s^{\prime}\right\|_{L^2_X} & \leq  I_1 \varepsilon^f \left\|s^{\prime}\right\|_{L^2_X}
\end{aligned}
$$
    
    \item  There exist constants $\varepsilon_{2}>0, I_2>0$ and $C>0$ such that for any $s \in G_{\varepsilon,\mu}^{\perp} \cap \mathcal{D}^{\mu}(D_{\varepsilon})$ and $\varepsilon>\varepsilon_{2}$,
$$
\left\|D_{\varepsilon,\mu} s\right\|_{L^2_X} \geq I_2 \sqrt{\varepsilon}\|s\|_{L^2_X}
$$
\end{enumerate}
\end{proposition}

\begin{remark}
This is the analog of Proposition 6.29 of \cite{jayasinghe2023l2}, where we worked out the details for the case of the de Rham operator in detail. There we followed the proof strategy of Zhang in Chapters 4 and 5 of \cite{Zhanglectures} in the smooth setting, described in the outline part of the proof in \cite{jayasinghe2023l2}.
There are two new ingredients in this proof for the case of the Dolbeault complex, the restriction to equivariant Hilbert complexes and the use of the expansion of the operator in Lemma \ref{Lemma_operator_estimates}.
\end{remark}

\begin{proof}
We follow the same strategic steps as in the proof of Proposition 6.29 of \cite{jayasinghe2023l2}, even though the technicalities are slightly more involved. We pick an orthonormal basis $\widehat{W_{\varepsilon,\mu, a}}$ for the vector space generated by the forms in $G_{\varepsilon,\mu}$.

\textbf{\textit{proof of 1:}} For any $s \in L^2\Omega^{0,q}_{\mu}(X;E)=H_{\mu}$ the projection $\Pi_{\varepsilon,\mu} s$ can be written
\begin{equation}
\label{Projection_first_witten_deform}    
\Pi_{\varepsilon,\mu} s=\sum_{a \in Cr(h)} \sum_{\eta \in \widehat{W_{\varepsilon,\mu, a}}} \left\langle \eta, s\right\rangle_{L^2_X} \eta
\end{equation}
where each $\eta$ can be written as a linear combination of forms $k^{-1/2}\eta_{a,\varepsilon}$ with $\eta_{a,\varepsilon}$ as defined in equation \eqref{equation_modify_forms_cutoff}. Since we want to estimate the first of the four operators in the decomposition given in \eqref{equation_fourfold_operators}, we show that
\begin{equation}
    ||\Pi_{\varepsilon,\mu} D^X_\varepsilon k^{-1/2} \eta_{a, \varepsilon}||_{L^2_X} \leq C \varepsilon^f
\end{equation}
for some $C>0$ and $\varepsilon$ large enough using the following argument.

Since $D^a_{\varepsilon}=D^a+ \varepsilon cl(V)$, observe that for a smooth function $v$ we have 
\begin{equation}
\label{equation_Leibniz_for_Witten_deformation}
    D^a_{\varepsilon}( v \omega_{a,\varepsilon}) = D^a(v \omega_{a,\varepsilon})+ \varepsilon cl(V^a)( v \omega_{a,\varepsilon}) = cl(dv) \omega_{a,\varepsilon} + v \big ( (D^a \omega_{a,\varepsilon}) + \varepsilon cl(V^a)(\omega_{a,\varepsilon}) \big ) = cl(dv) \omega_{a,\varepsilon}
\end{equation}
since $D^a_{\varepsilon} \omega_{a,\varepsilon}=0$. Since $\text{supp}( d\chi_{a}(t)) \subseteq \{1/2 \leq t \leq 3/4\}$ (recall that $\chi_{a}=\chi_{a,t_0=1}$), 
each $D^a_{\varepsilon}\eta_{a,\varepsilon}$ is compactly supported in $\{1/2 \leq t \leq 3/4\}$. Then for $\varepsilon$ large enough
\begin{equation} 
\label{inequality_useful_for_Morse_proof_234}
\left\langle D^a_{\varepsilon} \eta_{a, \varepsilon}, \eta_{a, \varepsilon} \right\rangle_{L^2} =  \left\langle cl(d\chi_a(t)) \omega_{a, \varepsilon}/ {\alpha_{a,\varepsilon}}, \eta_{a, \varepsilon} \right\rangle_{L^2_a} \leq e^{-C_0 \varepsilon}
\end{equation}
for some large enough positive constant $C_0$ 
since $cl(d\chi_a(t)) \omega_{a, \varepsilon}/ {\alpha_{a,\varepsilon}}$ is supported away from $t \leq 1/2$ and $\omega_{a, \varepsilon}=\omega_a e^{-t\varepsilon}$.

By Lemma \ref{Lemma_operator_estimates} we have  $k^{1/2}D^X_{\varepsilon,\mu}k^{-1/2}= D^a_{\varepsilon,\mu}+ \varepsilon^f Err$ on $U_a$ where $Err$ is an operator as in the statement of that lemma. 
Since $G_{\mu,\varepsilon}$ is a finite dimensional subspace with each element in $\mathcal{D}^{\mu}(D_\varepsilon)$, we see that 
$\langle Err (\eta, \eta \rangle_{L^2_X}$
can be bounded by a constant uniform in $\varepsilon$ for large $\varepsilon$, for each $\eta \in \widehat{W_{\varepsilon,\mu, a}}$ for fixed $\mu$. Thus we have the inequality
\begin{equation} 
\label{inequality_useful_for_Morse_proof_23754}
\left\langle D^X_{\varepsilon} \eta_{a, \varepsilon}, \eta_{a, \varepsilon} \right\rangle_{L^2_X} =  \left\langle cl(d\chi_a(t)) \omega_{a, \varepsilon}/ {\alpha_{a,\varepsilon}}, \eta_{a, \varepsilon} \right\rangle_{L^2_X} \leq \widetilde{C}_0 \varepsilon^f.
\end{equation}

Since each $\eta$ has unit norm, the Cauchy Schwartz inequality shows that $|\left\langle \eta, s\right\rangle_{L^2}| \leq ||s||$.
Since the forms $\eta$ in the basis $\widehat{W_{\varepsilon,\mu,a}}$ for a given critical point $a$ are orthogonal and since the supports of the forms in $\widehat{W_{\varepsilon,\mu,a}}$ for different critical points have no intersection, using equation \eqref{Projection_first_witten_deform} we see that
\begin{equation}
    ||\Pi_{\varepsilon,\mu} D^X_{\varepsilon} \Pi_{\varepsilon,\mu}s|| \leq I_1 \varepsilon^f||s||
\end{equation}
for large enough $\varepsilon$ in order to compensate for the finite sum of terms as well as ensuring \eqref{inequality_useful_for_Morse_proof_234}. 

\textbf{\textit{proof of 2:}} 
While $D^X$ is not formally self-adjoint with respect to the $L^2_a$ inner product on each $U_a$, the operator $D^a$ is. Thus, while $D_{\varepsilon,\mu, 2}$ and $D_{\varepsilon,\mu, 3}$ are not formal adjoints, $D^a_{\varepsilon, \mu, 2} =\Pi_{\varepsilon,\mu} D^a_{\varepsilon} \Pi_{\varepsilon,\mu}^{\perp}, D^a_{\varepsilon, \mu, 3} =\Pi_{\varepsilon,\mu}^{\perp} D^a_{\varepsilon} \Pi_{\varepsilon,\mu}$ are formally adjoint. Lemma \ref{Lemma_operator_estimates} shows that 
\begin{equation}
    k^{1/2}D_{\varepsilon,\mu, 2}k^{-1/2}=D^a_{\varepsilon,\mu,2}+\varepsilon^f Err, \hspace{5mm} k^{1/2}D_{\varepsilon,\mu, 3}k^{-1/2}=D^a_{\varepsilon,\mu,3}+\varepsilon^f Err.
\end{equation}
Thus it suffices to prove the first estimate of the two, which we now do.
Since each $\eta_{a, \varepsilon}$ has support in $U_a$, so does $\eta=k^{-1/2}\eta_{a, \varepsilon}$. One deduces that for any $s \in G_{\varepsilon,\mu}^{\perp} \cap \mathcal{D}_{\mu}(D_{\varepsilon})$,
$$
\begin{aligned}
D^X_{\varepsilon,\mu 2} s & =\Pi_{\varepsilon,\mu} D^X_{\varepsilon,\mu} \Pi_{\varepsilon,\mu}^{\perp} s=\Pi_{\varepsilon,\mu} D^X_{\varepsilon,\mu} s \\
& =\sum_{a \in Cr(h)} \sum_{\eta \in \widehat{W_{\varepsilon,\mu,a}}} \left\langle\eta, D^X_{\varepsilon,\mu} s\right\rangle_{L^2_X} \eta =\sum_{a \in Cr(h)} \sum_{\eta \in \widehat{W_{\varepsilon,\mu,a}}} \left\langle D^X_{\varepsilon,\mu} \eta, s\right\rangle_{L^2_X} \eta \\
& =\sum_{a \in Cr(h)} \sum_{\eta \in \widehat{W_{\varepsilon,\mu,a}}} \eta \Bigg( \int_{U_a} \left\langle D^a_{\varepsilon,\mu} k^{1/2}\eta,  s\right\rangle_{F}  dv_{U_a}+ \varepsilon^f \langle Err(\eta),  s \rangle_{L^2_a} \Bigg).
\end{aligned}
$$

The terms $\langle Err(\eta),  s \rangle_{L^2_a}$ can be bounded as in the proof of the first numbered statement above. Since each $\eta$ can be written as a linear combination of forms $k^{-1/2}\eta_{a,\varepsilon}$ as defined in equation \eqref{definition_perturbed_basis}, the other terms can be controlled by controlling the integral over $U_a$ with integrand $\left\langle D^a_{\varepsilon,\mu} \eta_{a,\varepsilon},  s\right\rangle_{F}$. This can be expanded as
\begin{equation}
    \left\langle D^a_{\varepsilon,\mu} \frac{\chi_a(t) \omega_{a} e^{-t \varepsilon}}{{\alpha_{a, \varepsilon}}},  s\right\rangle_{L^2} =\left\langle \frac{cl(d\chi_a(t)) \omega_{a} e^{-t \varepsilon}}{{\alpha_{a, \varepsilon}}},  s\right\rangle_{L^2} 
\end{equation}
restricted to each $U_a$, where we have used the argument in \eqref{equation_Leibniz_for_Witten_deformation}. We know that $d\chi_a$ is only supported on the set $\{1/2 \leq t \leq 3/4\}$ in each $U_a$, we see that the desired inequality follows from 
arguments similar to those in the proof of Step 1.

\textbf{\textit{proof of 3:}} 

This is proven in three steps:
\begin{enumerate}
    \item  Assume $\operatorname{supp}(s) \subset \bigcup_{a \in Cr(h)} U^{1/2}_{a}$.
    \item  Assume $\operatorname{supp}(s) \subset X \backslash \bigcup_{a \in Cr(h)} U^{1/4}_{a}$. 
    \item General Case.
\end{enumerate}

\textit{Step 1:} 
Since $s \in G_{\varepsilon,\mu}^{\perp}$, and since $\chi_{a}=\chi_{a,1}$ is identically $1$ on $U_{a}^{1/2}$, we see that $\langle s, k^{-1/2}\omega_{a,\varepsilon} \rangle_{L^2_X}=0$ for forms $\omega_{a,\varepsilon} \in \mathcal{H}(\mathcal{P}^{\mu}_{B,\varepsilon}(U_a))$ at critical points $a$ of $h$. Then Proposition \ref{Proposition_model_spectral_gap_modified_general}, implies that 
\begin{equation}
    \left\|D^a_{\varepsilon} k^{1/2}s\right\|_{L^2_a}^{2} \geq {\varepsilon} \left\| k^{1/2}s\right\|_{L^2_a}^{2}
\end{equation}
since we have projected away from the $0$ eigenvalues of the model harmonic oscillator and the other eigenvalues of $\Delta_\varepsilon$ grow of order $\varepsilon$.
Then Lemma \ref{Lemma_operator_estimates} together with the fact that the norms for $L^2_a, L^2_X$ are comparable can be used to show that
\begin{equation}
    \left\|D^X_{\varepsilon} s\right\|_{L^2_X} \geq  \left\|D^a_{\varepsilon}k^{1/2}s \right\|_{L^2_X} - \varepsilon^{f} C\left\|s \right\|_{L^2_X} \geq \sqrt{\varepsilon} \left\| s\right\|_{L^2_X} - \varepsilon^{f} C\left\|s \right\|_{L^2_X}
\end{equation}
for some constant $C$. 
Since $f< 1/2$, we see that there exists large enough $\varepsilon_{1}>0, C_1>0$  such that for any $\varepsilon \geq \varepsilon_{1}$
\begin{equation}
\left\|D^X_{\varepsilon} s\right\|_{L^2_X} \geq C_{1} \varepsilon^{1/2} \|s\|_{L^2_X},
\end{equation}
proving step 1. 

\textit{Step 2:} 
Since $\operatorname{supp}(s) \subset X \backslash \bigcup_{a \in Cr(h)} U_a^{1/4}$, one can proceed as in the proof of Proposition \ref{Propostion_growth_estimate_witten_deformed} to find constants $\varepsilon_{2}>0$ and $C_{2}>0$, such that for any $\varepsilon \geq \varepsilon_{2}$,
$$
\left\|D^X_{\varepsilon} s\right\|_{L^2_X} \geq C_{2} \sqrt{\varepsilon}\|s\|_{L^2_X},
$$
proving step 2.

\textit{Step 3:} 
Let $\widetilde{\chi} \in C_{\Phi}^{\infty}(X)$ be defined such that restricted to each $U_{a}$ for critical points $a$, $\widetilde{\chi}(t)=\chi_a(t)$, and that $\left.\widetilde{\chi}\right|_{X \backslash \bigcup_{a \in Cr(h)} U_{a}}=0$.
For any  $s \in G_{\varepsilon,\mu}^{\perp} \cap \mathcal{D}^{\mu}(D_{\varepsilon})$ we see that $\widetilde{\chi} s \in G_{\varepsilon,\mu}^{\perp} \cap \mathcal{D}^{\mu}(D_{\varepsilon})$.
Then the results of steps 1 and 2 shows that there exists $C_{9}>0$ such that for any $\varepsilon \geq \varepsilon_{0}+\varepsilon_{1}+\varepsilon_{2}$,
$$
\begin{gathered}
\left\|D^X_{\varepsilon,\mu} s\right\| \geq \frac{1}{2}\left(\left\|(1-\widetilde{\chi}) D^X_{\varepsilon,\mu} s\right\|+\left\|\widetilde{\chi} D^X_{\varepsilon,\mu} s\right\|\right)
=\frac{1}{2}\left(\left\|D^X_{\varepsilon,\mu}((1-\widetilde{\chi}) s)+[D^X, \widetilde{\chi}] s\right\|+\left\|D^X_{\varepsilon,\mu}(\widetilde{\chi} s)+[\widetilde{\chi}, D^X] s\right\|\right) \\
\geq \sqrt{\varepsilon}\left(C_{2}\|(1-\widetilde{\chi}) s\|+C_1\|\widetilde{\chi} s\|_{0}\right)-C_{3}\|s\| 
\geq C_{10} \sqrt{\varepsilon}\|s\|_{0}-C_{3}\|s\|,
\end{gathered}
$$
where the norms are for $L^2_X$ for some $C_3$ and $C_{10}=\min \left\{C_1, C_{2} / 2\right\}$, 
which completes the proof of the proposition.
\end{proof}

For any $c>0$, denote by {$G_{\varepsilon,\mu, c}$} the direct sum of the eigenspaces of $D^X_{\varepsilon}$ with eigenvalues lying in $[-c, c]$ which are also eigenspaces for $\sqrt{-1}L_V$ for the eigenvalue $\mu$, which is a finite dimensional subspace of $L^2_{\mu}\Omega^{0,q}(X;E)$. Let $\Pi_{\varepsilon,\mu,c}$ be the orthogonal projection from $L^2\Omega^{0,q}(X;E)$ to $G_{\varepsilon,\mu,c}$. The following is an analog of Lemma 6.31 of \cite{jayasinghe2023l2}, itself a generalization of Lemma 5.8 of \cite{Zhanglectures}. We define $f^+$ to be a real constant that satisfies $f < f^+ < 1/2$.

\begin{lemma}
\label{Lemma_inequality_spectral_for_Witten_deformation}
There exist $C_1>0$, $\varepsilon_3>0$ such that for any $\varepsilon>\varepsilon_3$, and any $\sigma \in G_{\varepsilon,\mu}$, 
\begin{equation}
    \left\|\Pi_{\varepsilon,\mu,\kappa} \sigma-\sigma\right\| \leq {C_1}{\varepsilon}^f\|\sigma\|
\end{equation}
where $\kappa=\varepsilon^{f^+}$.
\end{lemma}

\begin{proof}
Let $\delta=\{\lambda \in \mathbf{C}:|\lambda|=\kappa \}$ be the counter-clockwise oriented circle. By Proposition \ref{proposition_Zhangs_Morse_inequalities_intermediate_estimates}, one deduces that for any $\lambda \in \delta, \varepsilon \geq \varepsilon_{0}+\varepsilon_{1}+\varepsilon_{2}$ and $s \in \mathcal{D^\mu}(D^X_\varepsilon)$, there exists positive constants $I_1,I_2, I_3$ such that
\begin{equation}   
\begin{gathered}
\left\|\left(\lambda-D^X_{\varepsilon}\right) s\right\| \geq \frac{1}{2}\left\|\lambda \Pi_{\varepsilon,\mu} s -D^X_{\varepsilon,\mu, 1}\Pi_{\varepsilon,\mu}s   -D^X_{\varepsilon,\mu, 2} \Pi_{\varepsilon,\mu} s\right\| +\frac{1}{2}\left\|\lambda \Pi_{\varepsilon,\mu}^{\perp} s-D^X_{\varepsilon,\mu, 3} \Pi_{\varepsilon,\mu}^{\perp} s-D^X_{\varepsilon,\mu, 4} \Pi_{\varepsilon,\mu}^{\perp} s\right\| \\
\geq \frac{1}{2}\left(\left(\kappa-({I_1}+{I_2})\varepsilon^f\right)\left\|\Pi_{\varepsilon,\mu} s\right\|+\left({I_3} \sqrt{\varepsilon}-\kappa-{I_2}\varepsilon^f  \right)\left\|\Pi_{\varepsilon}^{\perp} s\right\|\right) .
\end{gathered}
\end{equation}
This shows that there exist $\varepsilon_{4}>\varepsilon_{0}+\varepsilon_{1}+\varepsilon_{2}$ and $C_{2}>0$ such that for any $\varepsilon \geq \varepsilon_{4}$ and $s \in \mathcal{D^\mu}((D^X_\varepsilon)^2)$,
\begin{equation}  
\label{equation_(5.27)_Zhang}
\left\|\left(\lambda-D^X_{\varepsilon}\right) s\right\| \geq C_{2}\|s\|.
\end{equation}
Thus, for any $\varepsilon  \geq \varepsilon _{4}$ and $\lambda \in \delta$,
\begin{equation}  
\lambda-D^X_{\varepsilon }: \mathcal{D^\mu}(D^X_\varepsilon) \rightarrow L^2\Omega^{0,q}_{\mu}(X;E)
\end{equation}
is invertible and the resolvent $\left(\lambda-D^X_{\varepsilon}\right)^{-1}$ is well-defined.
By the spectral theorem 
one has
\begin{equation}
\label{equation_spectral_projector_countour_integral}
\Pi_{\varepsilon,\mu,\kappa} \sigma-\sigma=\frac{1}{2 \pi \sqrt{-1}} \int_{\delta}\left(\left(\lambda-D^X_{\varepsilon}\right)^{-1}-\lambda^{-1}\right) \sigma d \lambda.
\end{equation}
Now one verifies directly
that for any $\sigma \in G_{\varepsilon,\mu}$
\begin{equation}
\left(\left(\lambda-D^X_{\varepsilon}\right)^{-1}-\lambda^{-1}\right) \sigma=\lambda^{-1}\left(\lambda-D^X_{\varepsilon}\right)^{-1} (D^X_{\varepsilon,\mu, 1}+D^X_{\varepsilon,\mu, 3}) \sigma .
\end{equation}
From Proposition \ref{proposition_Zhangs_Morse_inequalities_intermediate_estimates} and \eqref{equation_(5.27)_Zhang}, one then deduces that for any $\varepsilon \geq \varepsilon_{4}$ and $\sigma \in G_{\varepsilon,\mu}$,
\begin{equation}
\label{equation_estimate_finale}
\left\|\left(\lambda-D^X_{\varepsilon}\right)^{-1} (D^X_{\varepsilon, \mu,  1}+D^X_{\varepsilon, \mu, 3}) \sigma\right\| \leq C_{2}^{-1}\left\|(D^X_{\varepsilon, \mu,  1}+D^X_{\varepsilon, \mu, 3}) \sigma\right\| \leq \frac{({I_1}+{I_3}) \varepsilon^f}{C_{2} }   \|\sigma\|
\end{equation}
From \eqref{equation_spectral_projector_countour_integral}-\eqref{equation_estimate_finale}, we get the estimate in the statement of the lemma, finishing the proof.
\end{proof}

In Remark 5.9 of \cite{Zhanglectures}, Zhang explains that one can work out an analog of the proof with real coefficients, whereas the proof above implicitly uses the fact that we work in the category of complex coefficients.
The following is a generalization of Proposition 6.32 of \cite{jayasinghe2023l2}, itself a generalization of Proposition 5.5 of \cite{Zhanglectures}.

\begin{proposition}
\label{Proposition_small_eig_estimate_and_dimension}
There exists $\varepsilon_0>0$ such that when $\varepsilon>\varepsilon_0$, the number of eigenvalues in $[0,\varepsilon^{f^+}]$ of $\Delta_{\varepsilon}^{(q)}$, the Laplace-type operator acting on forms of degree $q$ of the complex $\mathcal{P}^{\mu}_{\varepsilon}(X)$, is equal to 
\begin{equation}
    \sum_{a \in Cr(h)} \dim \mathcal{H}^q(\mathcal{P}^{\mu}_{B,\varepsilon}(U_{a})).
\end{equation}
\end{proposition}

\begin{proof}
By applying Lemma \ref{Lemma_inequality_spectral_for_Witten_deformation} to the elements of $k^{1/2}\eta_{a,\varepsilon} \in \left\{\mathcal{W'}(\mathcal{P}^{\mu}_{B,\varepsilon}(U_{a})) : a \in Cr(h)\right\}$, one sees easily that for any $\kappa=\varepsilon^{f^+}$ when $\varepsilon$ is large enough, the elements of the set $\{\Pi_{\varepsilon,\mu,\kappa} \eta_{a,\varepsilon}: \eta_{a,\varepsilon} \in  \cup_{a \in Cr(h)} \widehat{W}_{a,\varepsilon,\mu}\}$, where $\widehat{W}_{a,\varepsilon,\mu}$ is the basis chosen in the proof of Proposition \ref{proposition_Zhangs_Morse_inequalities_intermediate_estimates}, are linearly independent. Thus, there exists $\varepsilon_{5}>0$ such that when $\varepsilon \geq \varepsilon_{5}$,
\begin{equation}
    \dim G_{\varepsilon,\kappa} \geq \dim G_{\varepsilon}.
\end{equation}
If $\dim G_{\varepsilon,\mu,\kappa}>\dim G_{\varepsilon,\mu}$, then there should exist a nonzero element $s \in G_{\varepsilon,\mu,\kappa}$ such that $s$ is perpendicular to $\Pi_{\varepsilon,\mu,\kappa} G_{\varepsilon,\mu}$. That is, $\left\langle s,\Pi_{\varepsilon,\mu,\kappa} \eta_{a,\varepsilon} \right\rangle_{L^2(X;E)}=0$ for any $\eta_{a,\varepsilon}$ as above.
Thus, there exists a constant $C_{4}>0$ such that 
\begin{equation}
\left\|\Pi_{\varepsilon,\mu}^{\perp} s\right\| \geq C_{4}\|s\| .
\end{equation}
Using this and Proposition \ref{proposition_Zhangs_Morse_inequalities_intermediate_estimates} one sees that when $\varepsilon>0$ is large enough,
\begin{equation}
\begin{gathered}
C C_{4} \sqrt{\varepsilon}\|s\| \leq\left\|D^X_{\varepsilon} \Pi_{\varepsilon,\mu}^{\perp} s\right\|=\left\|D^X_{\varepsilon} s-D^X_{\varepsilon} \Pi_{\varepsilon,\mu} s\right\|
=\left\|D^X_{\varepsilon} s-D^X_{\varepsilon, 1} s - D^X_{\varepsilon, 3} s\right\|\\
\leq\left\|D^X_{\varepsilon} s\right\|+\left\|D^X_{\varepsilon, 1} s\right\|+\left\|D^X_{\varepsilon, 3} s\right\|
\leq\left\|D^X_{\varepsilon} s\right\|+{I\varepsilon^f}\|s\|
\end{gathered}
\end{equation}
for some constants $C, I$, from which one gets
\begin{equation}
    \left\|D^X_{\varepsilon} s\right\| \geq C C_{4} \sqrt{\varepsilon}\|s\|-I{\varepsilon}^f\|s\|.
\end{equation}
Since $f < 1/2$, when $\varepsilon>0$ is large enough this contradicts the assumption that $s$ is a nonzero element in $G_{\varepsilon,\mu,\kappa}$.
Thus, one has
\begin{equation}
    \dim G_{\varepsilon,\mu,\kappa}=\dim G_{\varepsilon,\mu}=\sum_{a \in Cr(h)} \dim \mathcal{H}^q(\mathcal{P}^{\mu}_{B,\varepsilon}(U_{a}))
\end{equation}
proving the result.
\end{proof}

The following is Theorem \ref{theorem_small_eig_complex_intro} and describes what is known as the holomorphic Witten instanton complex, or the small eigenvalue complex of the Witten deformed Laplacian, encoding information about the local cohomology at the critical points.

\begin{theorem}[holomorphic Witten instanton complex]
\label{theorem_small_eig_complex}
For any integer $0 \leq q \leq n$, let $
\mathrm{F}_{\varepsilon, \mu, q}^{[0, \kappa]} \subset L^2_{\mu}\Omega^{0,q}(X;E)$ denote the vector space generated by the eigenspaces of $\Delta_{\varepsilon,q}$ associated with eigenvalues in $[0, \kappa]$. 
For $f=(1-\delta)/2$ 
let $\kappa=\varepsilon^{f^+}$ for $f^+$ satisfying $f<f^+<1/2$. Then there exists $\varepsilon_0>0$ such that when $\varepsilon>\varepsilon_0$ $\mathrm{F}_{\varepsilon, \mu, q}^{[0, \kappa]}$ has the same dimension as 
\begin{equation}
    \sum_{a \in Cr(h)} \dim \mathcal{H}^q(\mathcal{P}^{\mu}_{B,\varepsilon}(U_{a})),
\end{equation}
and form a finite dimensional subcomplex of $\mathcal{P}^{\mu}_{\varepsilon}(X)$ :
\begin{equation}
    \label{small_eigenvalue_complex}
\left(\mathrm{F}_{\varepsilon,\mu, q}^{[0, \kappa]}, P_{\varepsilon}\right): 0 \longrightarrow \mathrm{F}_{\varepsilon,\mu, 0}^{[0, \kappa]} \stackrel{P_{\varepsilon}}{\longrightarrow} \mathrm{F}_{\varepsilon,\mu, 1}^{[0, \kappa]} \stackrel{P_{\varepsilon}}{\longrightarrow} \cdots \stackrel{P_{\varepsilon}}{\longrightarrow} \mathrm{F}_{\varepsilon,\mu, n}^{[0, \kappa]} \longrightarrow 0.
\end{equation}
\end{theorem}

\begin{proof}    
This follows from Proposition \ref{Proposition_small_eig_estimate_and_dimension} once one shows that the small eigenvalue eigensections form a complex.
Since
$$
P_{\varepsilon} \Delta_{\varepsilon}=\Delta_{\varepsilon} P_{\varepsilon} =P_{\varepsilon} P_{\varepsilon}^* P_{\varepsilon} \text{   and    }
P_{\varepsilon}^* \Delta_{\varepsilon}=\Delta_{\varepsilon} P_{\varepsilon}^* =P_{\varepsilon}^* P_{\varepsilon} P_{\varepsilon}^*
$$
one sees that $P_{\varepsilon}$ (resp. $P_{\varepsilon}^*$ ) maps each $\mathrm{F}_{\varepsilon,\mu, q}^{[0, \kappa]}$ to $\mathrm{F}_{\varepsilon,\mu, q+1}^{[0, \kappa]}$ (resp. $\mathrm{F}_{\varepsilon,\mu, q-1}^{[0, \kappa]}$ ). 
The Kodaira decomposition of $\mathcal{P}_{\varepsilon}(X)$ restricts to this finite dimensional complex $\left(\mathrm{F}_{\varepsilon,\mu, q}^{[0, \kappa]}, P_{\varepsilon}\right)$. 

\end{proof}

\subsection{Morse inequalities}

In this subsection we prove the holomorphic Morse inequalities using the instanton complex. Theorem \ref{Theorem_strong_Morse_Dolbeault_version_2_intro} is a restatement of the following.

\begin{theorem}[Strong form of the holomorphic Morse inequalities]
\label{Theorem_strong_Morse_Dolbeault_version_2}

In the same setting as in the previous subsection with $\mathcal{P}_{\alpha}(X)=(L^2\Omega^{0,\cdot}(X;E), \mathcal{D}_{\alpha}(P), P=\overline{\partial}_E)$ which we assume is Fredholm and where local complexes induced at isolated critical points are transversally Fredholm, we have 
\begin{equation}
\label{equation_blah_440}
    \Big( \sum_{a \in Crit(h)}  \sum_{q=0}^n b^q \mathrm{dim}(\mathcal{H}^{q}(\mathcal{P}^{\mu}_{\alpha,B}(U_a)) \Big) = \sum_{q=0}^n b^q \mathrm{dim}(\mathcal{H}^{q}(\mathcal{P}^{\mu}_{\alpha}(X))) + (1+b) \sum_{q=0}^{n-1} Q^{\mu}_q b^q
\end{equation}
where $Q^{\mu}_q$ are non-negative integers.
Equivalently we have a power series
\begin{equation}
\label{equation_blah_441}
    \Big( \sum_{a \in Crit(h)}  \sum_{q=0}^n b^q {Tr} (T_{s,\theta}|_{\mathcal{H}^{q}(\mathcal{P}_{\alpha,B}(U_a))}) \Big) = \sum_{q=0}^n b^q {Tr} (T_{s,\theta}|_{\mathcal{H}^{q}(\mathcal{P}_{\alpha}(X))}) + (1+b) \sum_{q=0}^{n-1} Q_q b^q 
\end{equation}
where the $Q_q$ are power series in the variable $\lambda=se^{i\theta}$ 
converging for $|\lambda|=s<1$. The coefficients of this series are non-negative integers.
\end{theorem}

\begin{proof}
We first prove the statement in \eqref{equation_blah_440}.
We apply equation \eqref{equation_with_the_b} of Theorem \ref{Lefschetz_supertrace} with the endomorphism $T=Id$
to the complex \eqref{small_eigenvalue_complex} for the fixed value of $\mu$. Since this is a finite dimensional (equivariant) Hilbert complex, we can take $t$ to $0$ 
to see that the left hand side is exactly the expression
\begin{equation}
    \Big( \sum_{a \in Crit(h)}  \sum_{q=0}^n b^q \dim(\mathcal{H}^{q}(\mathcal{P}^{\mu}_{\alpha,B,\varepsilon}(U_a)) \Big)
\end{equation}
and the right hand side is of the form
\begin{equation}
    \sum_{q=0}^n b^q \dim(\mathcal{H}^{q}(\mathcal{P}^{\mu}_{\alpha,\varepsilon}(X))) + (1+b) \sum_{q=0}^{n-1} Q_q b^q
\end{equation}
where $Q_q$ are non-negative integers. Indeed the $Q_q$ are the dimensions of the co-exact smaller eigenvalue eigensections of the deformed Laplace-type operator, as is clear from the statement of Theorem \ref{Lefschetz_supertrace}. Since the cohomology groups of the Witten deformed and the undeformed complexes have the same dimensions for both the local and global complexes we obtain equation \eqref{equation_blah_440}.

Equation \eqref{equation_blah_441} can be obtained by multiplying each term in equation \eqref{equation_blah_440} by $\mu$ where $\mu=\lambda^{r}$ for some power of $r$ and taking sums over all $\mu$ as defined in equation \eqref{equation_renormalized_trace_definition}.
The convergence of the power series for $s<1$ follows from results in \cite{jayasinghe2023l2} which we discussed in see Remark \ref{remark_convergence_of_generating_series}.
This proves the result.    
\end{proof}

Theorem \ref{Theorem_dual_inequalities_intro_version} is a restatement of the following.

\begin{theorem}[Dual equivariant holomorphic Morse inequalities]
\label{Theorem_dual_inequalities}

In the same setting as Theorem \ref{Theorem_strong_Morse_Dolbeault_version_2}, let $(\mathcal{P}_{\alpha})_{SD}(X)$ be the Serre dual complex
\begin{equation}
    (\mathcal{P}_{\alpha})_{SD}(X):=\mathcal{R}_{\alpha}(X)=(L^2\Omega^{n,n-\cdot}(X;E^*),\mathcal{D}_{\alpha}(\overline{\partial}^*_{E^* \otimes K}),\overline{\partial}^*_{E^* \otimes K})
\end{equation}
and let $T^{\mathcal{R}_{\alpha}(X)}_{\theta}$ be the geometric endomorphism induced on this complex corresponding to the K\"ahler action, and we denote it as $T_{\theta}$ with some abuse of notation.
Then we have that
\begin{equation}
    b^n \Big( \sum_{a \in Crit(h)}  \sum_{q=0}^n b^{-q} Tr(T_{s,-\theta}|_{\mathcal{H}^{q}(\mathcal{R}_{\alpha,B}(U_a))}) \Big) = b^n \sum_{q=0}^n b^{-q} Tr(T_{s,-\theta}|_{\mathcal{H}^{q}(\mathcal{R}_{\alpha}(X))}) + (1+b) \sum_{q=0}^{n-1} \widetilde{Q}_q b^q 
\end{equation}
where the $B$ subscript denotes a choice of domain for the local complex corresponding to the K\"ahler hamiltonian $-h$ (corresponding to $T_{-\theta}$)
where the $\widetilde{Q}_q$ are power series in the variable $\lambda=se^{i\theta}$ 
which converge for $s=|\lambda|>1$. The coefficients of this series are non-negative integers.
Moreover
\begin{equation}
\label{equation_Poincare_polynomial_same_for_dual_22}
    \sum_{q=0}^n b^q Tr  (T_{1,\theta}|_{\mathcal{H}^{q}(\mathcal{P}_{\alpha}(X))})=b^n \sum_{q=0}^n b^{-q} Tr  (T_{1,-\theta}|_{\mathcal{H}^{q}(\mathcal{R}_{\alpha}(X))}).
\end{equation}
\end{theorem}

\begin{proof}
This is a straightforward corollary of the Morse inequalities and the dualities in Proposition \ref{Proposition_duality_complex_conjugation_for_local_Lefschetz_numbers}.
The equality \eqref{equation_Poincare_polynomial_same_for_dual_22} holds due to Serre duality.
\end{proof}

\section{Applications: Examples and generalizations}

In this section we discuss some applications of the holomorphic Morse inequalities. We begin by studying some examples, including explicit verifications of conjectural holomorphic Morse inequalities corresponding to the algebraic Lefschetz-Riemann-Roch formulas of \cite{baum1979lefschetz}. We then discuss Poincar\'e Hodge polynomials, Hirzebruch $\chi_y$ invariants, signature and other equivariant invariants. 

\subsection{Examples}
\label{Subsection_examples}

In \cite[\S 7.3]{jayasinghe2023l2}, we studied the holomorphic Lefschetz fixed point theorem for many natural examples of algebraic toric varieties with wedge metrics and K\"ahler Hamiltonian actions of algebraic tori. Here we study the holomorphic Morse inequalities in some of those cases.

In the smooth setting, the Dirac type operators being essentially self-adjoint, we recover the equivariant holomorphic Morse inequalities of Witten for smooth K\"ahler manifolds actions given in equation (20) of \cite{witten1984holomorphic}.
The contribution at isolated fixed points $a$ in the smooth setting can be written as (see equation (21) of \cite{witten1984holomorphic})
\begin{equation}
\label{equation_Naki_Mayina}
    b^{n_a} E_a(\theta) \prod_{\gamma_j^a >0} \frac{1}{1-e^{i\gamma_j^a \theta}} \prod_{\gamma_j^a <0} \frac{e^{i|\gamma_j^a| \theta}}{1-e^{i|\gamma_j^a| \theta}}
\end{equation}
Here $n_a$ is the number of negative $\gamma_j^a$, and is called the \textit{Morse index} at the zero $a$ of $V$ (this is twice the Morse index for the K\"ahler moment map considered as a classical Morse function). Here $E_a(\theta)$ is the trace of the induced action of $T_{\theta}$ on the vector space $E_a$ given by the restriction of the bundle $E$ to $a$.

This formula is derived by Witten by studying what happens on polydisk fundamental neighbourhoods of isolated smooth fixed points. The $\gamma_j^a$ are precisely the weights for the K\"ahler action at the fixed point given in definition \eqref{Definition_Kahler_Morse_function}.
Each product for a fixed positive weight $\gamma_j^a$ is the trace of the endomorphism $T_{\theta}$ over the Hardy space of $L^2$ bounded holomorphic functions on the disk factor where the group action is a rotation with weight $\gamma_j^a$.
On the other hand, the factor for $\gamma_j^a<0$ is derived in the discussion leading up to equation (27) of \cite{witten1984holomorphic}, as the trace over the basis of $L^2$ integrable anti-holomorphic one forms on the disc which are in the local cohomology of the dual complex. This is similar to the computations in Example \ref{Example_spinning_sphere_intro}.

We now study the Morse inequalities corresponding to the example of the singular quadric in Subsection 7.3.3 of \cite{jayasinghe2023l2}, for the minimal domain.

\begin{example}
\label{example_quadric_Morse_inequalities}
Consider $\widehat{M}$, the zero set of the polynomial $G(W,X,Y,Z)=Z^2-XY$ in $\mathbb{CP}^3$, which has an isolated singularity at $[W:X:Y:Z]=[1:0:0:0]$. 
This has an action of the complex torus $(\mathbb{C}^*)^2$ given by 
\begin{equation}
\label{equation_group_action_for_quadric}
    (\lambda,\mu) \cdot [W:X:Y:Z]=[W:\lambda^2 X:\mu^2 Y:\lambda \mu Z].
\end{equation}
A simple computation shows that this action, for generic values of $\lambda$ and $\mu$, will fix three points. The singular point is fixed, as are $a_1=[0:1:0:0]$ and $a_2=[0:0:1:0]$.

Let us consider the inequality corresponding to $p=0$.
We can consider the case of the circle action where $\lambda=e^{i\gamma_1 \theta}$ and $\mu=e^{i\gamma_2 \theta}$ with $\gamma_1 > \gamma_2 >0$. Then the singular fixed point contributes $\frac{1+\lambda \mu}{(1-\lambda^2)(1-\mu^2)}$ to the local Lefschetz polynomial.
Here $\mu / \lambda$, $1/\mu^2$ and $1/\lambda^2$ have negative exponents (up to the factor of $i$), and therefore the contributions to the polynomial are  $\frac{b^2 (\lambda/\mu)(\lambda^2) }{(1-\lambda/\mu)(1-\lambda^2)}$ at $a_1$ and $\frac{b \mu^2}{(1-\lambda/\mu)(1-\mu^2)}$ at $a_2$. Here $b$ is the formal variable in the polynomial.
\end{example}

The contribution to the dual inequalities from the singular fixed point is 
\begin{equation}
    b^2\frac{(1+\lambda \mu)\lambda^{-2}\mu^{-2}}{(1-\lambda^{-2})(1-\mu^{-2})}=b^2\frac{(1+1/\lambda \mu)\lambda^{-1}\mu^{-1}}{(1-\lambda^{-2})(1-\mu^{-2})} 
\end{equation}
where the series expansion of the latter in powers of $\lambda^{-1}, \mu^{-1}$ is easily seen to be the Morse polynomial at the singular critical point for the Serre dual complex with Hilbert space $L^2\Omega^{n,n-q}(X;\mathbb{C})$. The following argument shows that the global cohomology vanishes for $q>0$.
The sum of the local Morse polynomials for the complexes at the critical points is given by
\begin{equation}
\label{equation_example_1_500_1}
    \frac{1+\lambda \mu}{(1-\lambda^2)(1-\mu^2)}+\frac{b \mu^2}{(1-\lambda/\mu)(1-\mu^2)}+\frac{b^2 (\lambda/\mu)(\lambda^2) }{(1-\lambda/\mu)(1-\lambda^2)},
\end{equation}
and those for the dual complex is given by
\begin{equation}
\label{equation_example_1_500_2}
    b^2\frac{(1+\lambda \mu)\lambda^{-2}\mu^{-2}}{(1-\lambda^{-2})(1-\mu^{-2})}+\frac{b \mu/\lambda}{(1-\lambda/\mu)(1-\mu^2)}+\frac{1}{(1-\lambda/\mu)(1-\lambda^2)}.
\end{equation}
The Morse polynomial in \eqref{equation_example_1_500_1} can be expanded as a series in $\lambda, \mu$ where there are no non-negative powers of $\lambda, \mu$ 
for each $b$, and it is easy to see that there is only one constant term $1$.

The dual Morse polynomial in \eqref{equation_example_1_500_2} has a series expansion in non-positive powers of $\lambda, \mu$ for each $b$, and again it is clear that the only constant term $1$.
Thus taking the common terms in each polynomial we get only $1$, which is the classical Morse polynomial discussed in the introduction.
Since there are no $(1+b)$ terms in this classical Morse polynomial, the Morse lacunary principle (see Theorem 3.39 of \cite{banyaga2004lectures}) can be invoked to show that the classical Morse polynomial is equal to the Poincar\'e polynomial.

\begin{remark}[Related results and choices of formulation]
In \cite[\S 7.3]{jayasinghe2023l2}, we studied the holomorphic Lefschetz numbers for $p=1,2$, as well as the case of the spin-Dirac complex. The detailed computations in that article can be used to easily compute the holomorphic Morse inequalities corresponding to those Lefschetz numbers.

Furthermore, we saw that the holomorphic Lefschetz numbers in this setting matched with the Baum-Fulton-Quart versions in \cite{baum1979lefschetz,Baumformula81}, and thus these holomorphic Morse inequalities are a generalization of those results as well.

For $p=1$, we saw that the local cohomology groups are not free modules over the holomorphic functions. Thus we cannot simply multiply the local Morse polynomial for the untwisted complex by $E_{a}({\theta})$ as in the smooth case. 
\end{remark}

We copy from \cite{jayasinghe2023l2} our description of the Hamiltonian Morse functions corresponding to the $(\mathbb{C}^*)^2$ action on $\widehat{M}$.
Consider the $(\mathbb{C}^*)^3$ toric action on $\mathbb{CP}^3$ corresponding to the $(S^1)^3$ action given by $(\theta_1,\theta_2,\theta_3)\cdot [w:x:y:z] \rightarrow [w:e^{i\theta_1}x:e^{i\theta_2}y:e^{i\theta_3}z]$. This corresponds to the symplectic moment map 
\begin{equation}
\label{Moment_map_for_quadric}
    [w:x:y:z] \rightarrow [H_1,H_2,H_3]=\Bigg[\frac{|x|^2}{|\rho|^2},\frac{|y|^2}{|\rho|^2},\frac{|z|^2}{|\rho|^2} \Bigg]
\end{equation}
where $|\rho|^2=|w|^2+|x|^2+|y|^2+|z|^2$. Consider the Hamiltonian function $F_1=H_3^2-H_2H_1$, and let us write $\lambda^2=e^{i\theta_1}$, $\mu^2=e^{i\theta_2}$ and $\alpha=e^{i\theta_3}/(\lambda\mu)=e^{i(\theta_3-(\theta_1+\theta_2)/2)}=e^{i\theta_4}$. We see that $F_1^{-1}(0)=\{|z|^2-|xy|=0\}$, which is the space given by the points satisfying $(z^2 \alpha^2 -xy)=0$, for all $\alpha,w,x,y,z$ with $|\alpha|=1$ satisfying $z^2\alpha^2-xy=0$. Quotienting the vanishing set of $z^2-xy=0$ by the circle action corresponding to $e^{i\theta_4}$, we get the symplectic reduction of $F_1^{-1}(0)$, which is precisely the space $\widehat{M}$ and the K\"ahler action we studied is the restriction of the residual effective action on the reduced space.

\subsection{Conjectural Morse inequalities for other cohomology theories}
\label{subsubsection_conjectural_inequalities}

In \cite[\S 7.3.6]{jayasinghe2023l2} we showed that for non-normal toric varities, the Baum-Fulton-Quart Lefschetz-Riemann-Roch formulas do not match those in $L^2$ cohomology. 
The holomorphic Morse inequalities we proved in this article are for $L^2$ cohomology. 
In \cite{witten1984holomorphic}, he explains how to go from the Lefschetz-Riemann-Roch theorem on smooth complex manifolds to conjectural Morse inequalities, which we know hold on K\"ahler manifolds. We can follow the same idea to come up with conjectural Morse inequalities for the \cite{baum1979lefschetz} complexes and Hilbert complexes in \cite{LottHilbertconplex,RuppenthalSerre2018}.

We first discuss features of the Morse inequalities we study here, in more abstract language.
Given a $\mathbb{Z}$ graded co-chain complex on a space $X$ and a \textit{Morse function} $h$, a \text{conjectural abstract Morse inequality} is roughly of the form 
\begin{multline*}
\sum_{a \in crit(h)} \text{local (equivariant) Morse polynomials}(b) - \text{(equivariant) Poincar\'e polynomial}(b)\\ 
= (1+b) \text{Error polynomial}(b)
\end{multline*}
where the Poincar\'e polynomial is $\sum_k b^k \dim {H}^k(X)$ where ${H}^k(X)$ is the cohomology group of the complex. The equivariant Poincar\'e polynomial is $\sum_k b^k Tr T_{\theta} {H}^k(X)$ where $T_\theta$ is the endomorphism on the complex corresponding to a \textit{compatible} circle action. 
The local equivariant Morse polynomial at a critical point $a$ of $h$ is the equivariant Poincar\'e polynomial for the local cohomology group, $\sum_k b^k Tr T_{\theta} {H}^k(U_a)$, where ${H}^k(U_a)$ are local cohomology groups which depend on the Morse function at the critical point.

There are two main features of a useful Morse inequality.
\begin{enumerate}
    \item The formula for $b=-1$ must give the corresponding Lefschetz fixed point theorem.
    \item The error terms must have a \textit{positivity property}.
\end{enumerate}
The second feature is what allows the Morse inequality to be more useful than simply the Lefschetz fixed point theorem, allowing for the use of tools like the lacunary principle. 

\textit{\textbf{
Our conjecture is that this setup works for the complexes studied in \cite{baum1979lefschetz} and \cite{LottHilbertconplex,RuppenthalSerre2018}, with the correct choices of local cohomology groups for the complex.
}}

It is clear that the conjectural Morse inequality is completely determined by the local Morse polynomials. Given the setup in \cite{baum1979lefschetz,Baumformula81}, the Poincar\'e polynomial is defined. What remains is to determine the local Morse polynomial when the K\"ahler action is Hamiltonian. We assume we can decompose a neighbourhood of the critical point $a$ of the singular variety into a product $U_{a,s}\times U_{a,u}$, where the factors are attracting and expanding factors as we have studied. We can simply take the local cohomology of the complex for the attracting factor $U_{a,s}$. For instance in the case of the structure sheaf for a complete intersection variety, this is simply the local ring used in \cite{baum1979lefschetz}. For the expanding factor, we need to use the local cohomology of the dual complex (for instance for the structure sheaf, this is just the complex for the canonical bundle), with a new shifted grading defined by $[\cdot] \mapsto n-[\cdot]$ for a variety of complex dimension $n$.
The K\"unneth formula then gives us a cohomology group for $U_a$, and the equivariant Poincar\'e polynomial (series) for this group is the local Morse polynomial. We can renormalize the traces following ideas in \cite{Baumformula81} and \cite{jayasinghe2023l2} to define the equivariant local Morse polynomials.

We have verified that the holomorphic Morse inequalities hold for equivariant holomorphic sheaves studied in \cite{baum1979lefschetz,Baumformula81} in examples. We provide two here.

\begin{example}
\label{example_singular_calabi_yau}
Consider the example of $Z^4-X^3Y=0$ in $\mathbb{CP}^3$ with coordinates $[W:X:Y:Z]$, which admits the algebraic torus action $(\lambda,\mu)\cdot [W:X:Y:Z]=[W:\lambda^4X:\mu^4Y:\lambda^3\mu Z]$. We compute the equivariant index for the untwisted Dolbeault complex $\mathcal{P}=(\Omega^{0,q}(X), \overline{\partial})$.
There are three fixed points at $a_1=[1:0:0:0]$, $a_2=[0:1:0:0]$ and $a_3=[0:0:1:0]$, where $a_2$ is a smooth fixed point while $a_1$ and $a_3$ are singular. The singular set is $[W:0:Y:0]$.
\end{example}

This space can be obtained as the compactification of the affine cone over the singular algebraic curve $Z^4-X^3Y=0$ in $\mathbb{CP}^2$. Thus the point $a_1$ is a stratum of depth $2$ within the singular set. We studied this example in some detail in \cite[\S 7.3.6]{jayasinghe2023l2}. In particular we showed that the Baum-Fulton-Quart Lefschetz Riemann-Roch number in \cite[\S 3.3]{baum1979lefschetz}, obtained by summing up the local Lefschetz numbers for $a_1, a_2$ and $a_3$ in order is
\begin{equation}
\label{equation_Singular_Calabi_Yau_p=0}
    \frac{1+(\lambda^3\mu)+(\lambda^3\mu)^2+(\lambda^3\mu)^3}{(1-\lambda^4)(1-\mu^4)}+ \frac{1}{(1-1/\lambda^4)(1-\mu/\lambda)}
    +\frac{1+(\lambda^3/\mu^3)
    +(\lambda^3/\mu^3)^2+(\lambda^3/\mu^3)^3}{(1-1/\mu^4)(1-\lambda^4/\mu^4)}= 1+\frac{\lambda^5}{\mu}.
\end{equation}
We explained there that the second term on the right hand side was the trace of the geometric endomorphism on a global section in degree $2$ of the complex.
We now consider the holomorphic Morse inequalities for the circle action corresponding to $s=1$ where $\lambda=se^{i\gamma_1 \theta}$, $\mu=se^{i\gamma_2 \theta}$ with $\gamma_1>0, \gamma_2>0, \gamma_2-\gamma_1 >0$. Here we shall write it as Laurent series are convergent for $s<1$, as defined in \eqref{equation_renormalized_trace_definition}. 
The contribution from $a_1$ is the same as that for the Lefschetz-Riemann-Roch formula, those for $a_2$ and $a_3$ are
\begin{equation}
    \frac{b \lambda^4}{(1-\lambda^4)}\frac{1}{(1-\mu/\lambda)}, \hspace{5mm} (1+(\lambda^3/\mu^3)
    +(\lambda^3/\mu^3)^2+(\lambda^3/\mu^3)^3)\frac{b \mu^4}{(1-\mu^4)}\frac{b \mu^4/\lambda^4}{(1-\mu^4/\lambda^4)}.
\end{equation}
It is easy to see that the only terms in the series with $e^{ic\theta}$ for $c \leq 0$ are given by $1$ from the series expansion of $a_1$, and $b^2 (\lambda^3/\mu^3)^3 \mu^4 \mu^4/\lambda^4=b^2\lambda^5/\mu$ from the series corresponding to $a_3$.

Similarly if we take $\gamma_1<0, \gamma_2<0, \gamma_2-\gamma_1<0$ and write the contributions in terms of formal sums where the Laurent series are convergent for $s>1$, 
we see that the contribution from $a_3$ is the same as that for the Lefschetz-Riemann-Roch formula, and those for $a_1$ and $a_2$ are 
\begin{equation}
    (1+(\lambda^3\mu)+(\lambda^3\mu)^2+(\lambda^3\mu)^3) \frac{b/\lambda^4}{(1-1/\lambda^4)} \frac{b /\mu^4}{(1-1/\mu^4)}, \hspace{5mm} \frac{1}{(1-\lambda^{-4})}\frac{b\lambda/\mu}{(1-\lambda/\mu)}.
\end{equation}
The only terms in the series with $e^{ic\theta}$ for $c > 0$ are given by $1$ from the series expansion of $a_3$, and $b^2 (\lambda^3\mu)^3 \mu^{-4}/\lambda^{-4}=b^2\lambda^5/\mu$ from the series corresponding to $a_1$.

Thus the classical Morse polynomial is given by $1+b^2 \lambda^5/\mu$, which by the lacunary principle is equal to the Poincar\'e polynomial which we can verify explicitly by computing traces over the global cohomology groups as in \cite[\S 7.3.6]{jayasinghe2023l2} and we see that the conjectural holomorphic Morse inequalities hold for this example.

We explained in \cite[\S 7.3.6]{jayasinghe2023l2} that the section in degree $2$ of the complex is the complex conjugate of the trivializing section of the canonical sheaf. Indeed the Lefschetz-Riemann-Roch formula for the complex twisted by the square root of the canonical sheaf gives the same formula as equation \eqref{equation_Singular_Calabi_Yau_p=0} up to an overall factor of $\sqrt{\mu/ \lambda^5}$ multiplying all terms.
Indeed this shows that there may exist global cohomology for the complex in the Baum-Fulton-MacPherson-Quart sense, even when there are K\"ahler Hamiltonian $S^1$ actions on the space. We contrast this with the rigidity result in the smooth setting in Subsection \ref{subsection_Rigidity_version_1}.
In the smooth case where K3 surfaces are the only K\"ahler manifolds of real dimension 4 with vanishing first Chern class, there are no K\"ahler circle actions.

Let us consider the simpler example of the cusp curve which we studied in Example 7.33 of \cite{jayasinghe2023l2}. This is a space with a non-normal singularity. We compute the $L^2$ and Baum-Fulton-MacPherson holomorphic Lefschetz numbers for a toric action in that article. In \cite{LottHilbertconplex} Lott constructs a Hilbert complex based on a complex in \cite[Corollary 1.2]{RuppenthalSerre2018} which has the same Todd class as the Baum-Fulton-MacPherson Todd class. Thus the local and global holomorphic Lefschetz numbers for this analytic complex would be the same as those given in \cite{baum1979lefschetz}. However the conjectural holomorphic Morse inequalities for the analytic complex on spaces is in general different from the conjectural ones for the Baum-Fulton-MacPherson theory. 

\begin{example}[Cusp singularity in $\mathbb{CP}^2$]
\label{example_cusp_singularity}
Consider the cusp curve given by $ZY^2-X^3=0$ in $\mathbb{CP}^2$, where we have the $\mathbb{C}^*$ action $(\lambda)\cdot [X:Y:Z]=[\lambda^2X:\lambda^3Y:Z]$. We consider the associated family of geometric endomorphisms on the Dolbeault complex for the trivial bundle. The action has one smooth fixed point at $[0:1:0]$. If we consider the case where $|\lambda|<1$, then the local Morse polynomial at the smooth fixed point is given by $b \sum_{k=1}^{\infty} \lambda^{k}$. 
The other fixed point is at the singularity $[0:0:1]$.
\end{example}

We observed in \cite{jayasinghe2023l2} that the local ring at the singular point is the module with \textit{basis} ${1,t^2,t^3,t^4,...}$ where $t=Y/X$. Thus the local Morse polynomial corresponding to \cite{baum1979lefschetz} is given by $1+\lambda^2 (\sum_{k=0}^{\infty}\lambda^k)=1/(1-\lambda) - \lambda$.

The complex given in equation (1.2) of \cite{LottHilbertconplex} for the cusp curve we study here, has local cohomology at the singularity given by the holomorphic functions in degree $0$, which is the Hilbert space completion of $\mathbb{C}[t]$, and has local cohomology in degree $1$ generated by $t$ of Lott, where $t$ is the element in the skyscraper sheaf given by $\mathcal{O}_s/ \mathcal{O}_X$, where $\mathcal{O}_s$ is the sheaf of germs of weakly holomorphic functions while $\mathcal{O}_X$ is the structure sheaf). Thus the local Morse polynomial is given by $\sum_{k=0}^{\infty}\lambda^k +b \lambda =1/(1-\lambda) +b \lambda$.

Now it is easy to observe that the classical Morse polynomials for both complexes are equal, given by $1+b\lambda$. We explained in \cite{jayasinghe2023l2} how there is a section in the degree $1$ global singular cohomology group for these complexes (of \cite{LottHilbertconplex} and of \cite{baum1979lefschetz}), which is not $L^2$ bounded and thus not appearing in the $L^2$ cohomology. \textit{This example shows that the local Morse inequalities for these complexes can be different, even while the classical Morse polynomials are the same.}

\begin{remark}
\label{Remark_considerations_integrals_cup_products}
While there is a cup product for singular cohomology on singular spaces, the failure of Poincar\'e duality also means one has to be more careful when trying to define dual inequalities. Aside from this, the fact that certain classes correspond to sections that are not integrable means that the cup products may not be straightforwadly interpretted as integrals, important when it comes to physical applications similar to those we study in Subsection \ref{subsection_Nut_charge_signature} (see Remark \ref{Remark_formulas_NUTs_other_cohom_theory}). 
\end{remark}

As discussed in Subsection \ref{subsection_background_related_results}, there are counterexamples to the holomorphic Morse inequalities when the action is not associated to a Bialynicki-Birula decomposition (see Remark 4.2 \cite{wu1996equivariant}), whereas the theorem has been proven when such a decomposition exists on smooth complex manifolds in \cite{wu2003instanton}.

\subsection{Poincar\'e Hodge polynomials}
\label{subsection_Poincare_Hodge_polynomial}

In \cite{wu1996equivariant}, the equivariant Poincar\'e Hodge polynomials are studied in the smooth case, which are sometimes also called $Ell$ functions (see, e.g., \cite{Libgober_2018}). These admit natural extensions to Hilbert complexes associated to singular spaces that we study in this article.
We will denote the Dolbeault complexes $^p\mathcal{P}_{\alpha}=(L^2\Omega^{p,\cdot}(X;E),\mathcal{D}_{\alpha}(P),P)$ where $P_E=\overline{\partial}_{\Lambda^{p,0}\otimes E}$, and we define the corresponding \textit{\textbf{Poincar\'e Hodge polynomial}} to be 
\begin{equation}
    \chi_{y,b}(\mathcal{P}_{\alpha}):=\sum_{p=0}^n y^p \sum_{q=0}^n b^q \dim {\mathcal{H}^{q}(^p\mathcal{P}_{\alpha})}
\end{equation}
where we abuse notation to denote the Dolbeault complexes for all $p$ with fixed $\alpha$ in the left hand side of the above equality.
Given a circle action as we study in this article, we define the \textbf{\textit{equivariant Poincar\'e Hodge polynomial}} for the Dolbeault complex to be 
\begin{equation}
\label{equivariant_Poincare_Hodge_definition}
    \chi_{y,b}(\mathcal{P}_{\alpha},T_{\theta}):=\sum_{p=0}^n y^p \sum_{q=0}^n b^q TrT_{\theta}|_{\mathcal{H}^{q}(^p\mathcal{P}_{\alpha})}
\end{equation}
for global complexes as well as local complexes $^p\mathcal{P}_{\alpha}(U)=(L^2\Omega^{p,\cdot}(U;E),\mathcal{D}_{\alpha}(P_U),P_U)$.
These satisfy the following duality
\begin{proposition}
\label{Proposition_duality_Poincare_Hodge}
Let $X$ be a resolved stratified pseudomanifold with a K\"ahler wedge metric, and let $T_{\theta}$ be a geometric endomorphism corresponding to an isometric K\"ahler Hamiltonian circle action on 
$^p\mathcal{P}_{\alpha}=(L^2\Omega^{p,\cdot}(X;E),\mathcal{D}_{\alpha}(P_E),P_E)$ where $P_E=\overline{\partial}_{\Lambda^{p,0}\otimes E}$.
We denote the Serre dual complexes as $(^{p}\mathcal{R}_{\alpha})^*=(L^2\Omega^{p,n-\cdot}(X;E^*),\mathcal{D}_{\alpha}(P^*_{E^*}),P^*_{E^*})$ where $P_{E^*}=\overline{\partial}_{\Lambda^{p,0}\otimes E^*}$. We have $(^p\mathcal{P}_{\alpha})_{SD}=(^{n-p}\mathcal{R}_{\alpha})^*$.
Then
\begin{equation}
\label{equation_duality_Poincare_Hodge}
   \chi_{y,b}(\mathcal{P}_{\alpha},T_{\theta})=\chi_{y^{-1},b^{-n}}(\mathcal{R}_{\alpha},T_{-\theta})(by)^n
\end{equation}
and for fixed points $a$ of $T_{\theta}$,
\begin{equation}
\label{equation_duality_Poincare_Hodge_local}
   \chi_{y,b}(\mathcal{P}_{\alpha,B}(U_a),T_{\theta})=\chi_{y^{-1},b^{-1}}(\mathcal{R}_{\alpha,B}(U_a),T_{-\theta})(by)^n.
\end{equation}
where the choice of domain $B$ corresponds to the geometric endomorphism on the fundamental neighbourhoods $U_a$ (see Remark \ref{remark_domains_and_geometric_endomorphisms}), which are duals for $T_{\theta}$ and $T_{-\theta}$.
\end{proposition}

\begin{proof}
The proof follows from Proposition \ref{Proposition_duality_complex_conjugation_for_local_Lefschetz_numbers} and following the definitions of the invariants on both sides, noticing that there are sums over all $p$.
\end{proof}

It is immediate from Theorem \ref{Theorem_strong_Morse_Dolbeault_version_2_intro} that there are Morse inequalities for the Poincar\'e Hodge polynomials. For $y=b$, and $\mathcal{D}_{\min}(\overline{\partial})=\mathcal{D}_{\alpha=1}(\overline{\partial})$ we recover the \textbf{de Rham Morse inequalities} given in Theorem 1.5 of \cite{jayasinghe2023l2} (which hold more generally than for the K\"ahler Hamiltonian group actions that we study in this article) and when $y=-1=b$, we recover the Lefschetz fixed point theorem for the de Rham setting in that article.

\begin{remark}
\label{Remark_minimal_domain_equivalences}
Given a twisted Dolbeault complex $^p\mathcal{P}_{\alpha}=(L^2\Omega^{p,\cdot}(X;E),\mathcal{D}_{\alpha}(P_E),P_E)$ where $P_E=\overline{\partial}_{\Lambda^{p,0}\otimes E}$,  
we can express the cohomology groups as the quotients $(ker P_E/ im P_E)$. We note that $P_E$ is simply the projection of $d_E$ onto the $\Omega^{p,\cdot}(E)$ forms which can be define using the complex structure, and can be defined independent of the metric.
Thus when $\alpha=1$, the cohomology groups of the Dolbeault complex are completely determined by the minimal domain for $P_E$. 

In the case where $E$ is a flat bundle, we have the twisted de Rham complex   
\begin{equation}
    \mathcal{Q}=(L^2\Omega^k(X;E)),\mathcal{D}_{\min}(d_E),d_E),
\end{equation}
and the cohomology of this complex inherits the Hodge bi-grading given by the K\"ahler structure and we denote the cohomology groups in bi-degree $(p,q)$ as $\mathcal{H}^{p,q}(\mathcal{Q})$.

The cohomology of the Dolbeault complex with the minimal domain does not vary for different Hermitian wedge metrics, and we can find a Hermitian wedge metric for which the Dolbeault-Dirac operator satisfies the geometric Witt condition.
Since the Laplace-type operators for the de Rham and Dolbeault complexes on a K\"ahler space differ by a constant factor, this shows that the \textit{global} cohomology groups satisfy $\mathcal{H}^q(^p\mathcal{P}_{\alpha=1}) = \mathcal{H}^{p,q}(\mathcal{Q})$ for the global complex. In particular, since the de Rham complex is globally Fredholm, so is the Dolbeault complex.
\end{remark}

We define the \textit{\textbf{equivariant Hirzebruch $\chi_y$ invariant}} for a given Dolbeault complex as in the discussion above by
\begin{equation}
    \chi_{y}(\mathcal{P}_{\alpha},T_{\theta}):=\chi_{y,-1}(\mathcal{P}_{\alpha},T_{\theta}).
\end{equation}
While the $\chi_{y,b}$ polynomials can be studied in the K\"ahler setting, the Hirzebruch $\chi_y$ invariants are of more general interest on spaces with complex, and even almost complex structures.
In \cite[\S 7.2.7]{jayasinghe2023l2}, we studied the case when $b=-1$, for arbitrary $y$ where we restricted ourselves to the setting where the VAPS domain corresponded to the minimal domain for $P$ since the main results were proven for that domain. 
\textit{\textbf{In order to simplify the exposition we will now restrict ourselves to the case $E=\mathbb{C}$}} for the rest of this section.
Versions of the following results hold for twisted complexes as well, and twisted signatures are of wide interest in the study of elliptic genera and the Witten genus (see, e.g., \cite{hirzebruch1992manifolds}.

While it is possible to study versions of the Poincar\'e Hodge polynomials and related invariants for general domains, the case of the minimal domain for $\overline{\partial}$ is the case of most interest due to the correspondence of the cohomology of the Dolbeault and de Rham complexes highlighted in Remark \ref{Remark_minimal_domain_equivalences}.

It is a simple corollary of Proposition \ref{Proposition_duality_Poincare_Hodge} that the $\chi_{y=1}$ invariants vanish unless $n$ is even. This generalizes a well known property of the Hirzebruch signature invariant in the smooth setting. The following is a more general property of the equivariant $\chi_1$ invariants.

\begin{remark}[Self duality of $\chi_y$ invariant under Serre duality]
\label{remark_self_dual_chi_y}
In the setting of Proposition \ref{Proposition_duality_Poincare_Hodge}, for $E=\mathbb{C}^n$, we have that the $\chi_y$ invariant is \textit{\textbf{self dual under Serre duality}}, in the sense that $\chi_{y}(\mathcal{P}_{\alpha},T_{\theta})=\chi_{y^{-1}}(\mathcal{P}_{\alpha},T_{-\theta}) (-y^{n})$ and $\chi_{y}(\mathcal{P}_{\alpha,B}(U_a),T_{\theta})=\chi_{y^{-1}}(\mathcal{P}_{\alpha,B}(U_a),T_{-\theta})(-y)^n$ in the notation of Proposition \ref{Proposition_duality_Poincare_Hodge}.
The cases of $y=\pm 1$ motivate the name \textit{self duality} under Serre duality for this property.
\end{remark}

We state the following results, referring to \cite[\S 7.2.7]{jayasinghe2023l2} for more details where it was studied in the case where the cohomology of the VAPS domain matches that of the minimal domain.

\begin{theorem}[Lefschetz Hodge index theorem]
\label{Theorem_L2_Lefschetz_Hodge_index}
In the same setting as that of Proposition \ref{Proposition_duality_Poincare_Hodge}, where $E=\mathbb{C}$, the local equivariant signature at the fixed point $a$ is given by $\chi_{1}(\mathcal{P}_{1,B}(U_a),T_{\theta})$, and the global equivariant signature by $\chi_{1}(\mathcal{P}_{1,B}(X),T_{\theta})$.
\end{theorem}

\begin{remark}[Some connections to instanton computations]
\label{Remark_instanton_computations}
    In \cite[7.2.8]{jayasinghe2023l2}, we systematically explore how the $\chi_y$ invariant can be used to compute equivariant indices of self-dual and anti-self-dual complexes, with applications to computations that arise in physics. There we studied how certain relationships of instanton partition functions arising in equivariant Yang-Mills theories such as in \cite{pestun2012localization,festuccia2020transversally,festuccia2020twisting} are related to those of $\chi_y$ invariants.
\end{remark}

\section{Applications: Rigidity and more formulas}

The phenomenon that the index of operators restricted to equivariant Hilbert complexes vanishes for all $\mu \neq 0$ is usually called \textit{\textbf{rigidity}}. The equivariant indices of the signature and spin Dirac operators are self-dual under Serre duality as we explained in the previous section, and are rigid in the smooth setting when there are group actions (see, e.g. \cite{bott1989rigidity}).

\begin{remark}  
\label{Remark_Self_duality_symmetries}
Self duality under Serre duality shows that \textbf{if} one has a global Lefschetz number that has an expansion as $\sum_{j} m_j \lambda^{j}$, where $\lambda=se^{i\theta}$, then $m_j= m_{-j}$, and the global Lefschetz number is symmetric under the transformation $\lambda \mapsto \lambda^{-1}$.
\end{remark}

\subsection{Rigidity of $L^2$ de Rham cohomology}
\label{subsection_Rigidity_version_1}

The de Rham complex and its cohomology has many deformation invariance properties which is useful in seeing the rigidity of equivariant indices which can be expressed as linear combinations of traces of geometric endomorphisms over de Rham cohomology groups. The includes the equivariant Euler characteristic, signature, self-dual and anti-self-dual complexes.

We begin with a proposition which gives a correspondence between the local de Rham cohomology and the cohomology of the equivariant subcomplex with $\mu=0$ with respect to a local action.
Here we briefly summarize the local de Rham cohomology for fundamental neighbourhoods of critical points on Witt spaces, appearing in Theorem 1.5 of \cite{jayasinghe2023l2} referring to subsections 5.2.3 and 6.5.1 of that article for more details.

Given a fundamental neighbourhood $U_a$ of a critical point of a stratified Morse function as in Definition \ref{Definition_Kahler_Morse_function}, we can define de Rham complexes on the attracting factor $U_{a,s}$ to be $(L^2\Omega^{k}(U_{a,s}), \mathcal{D}_N(d_\varepsilon), d_\varepsilon)$ where $d_{U_{a,s},\varepsilon}$ is the Witten deformed de Rham operator acting on sections in $U_a$ and $\mathcal{D}_N(d_\varepsilon)$ is the maximal domain for the de Rham operator (this is defined differently for general operators in \cite{jayasinghe2023l2} but is equal to the maximal domain due to the Witt condition which implies that the Dirac operator corresponding to the de Rham operator is essentially self-adjoint). On the expanding factor we have $(L^2\Omega^{m-k}(U_{a,u}), \mathcal{D}_N(d^*_\varepsilon), d^*_\varepsilon)$ where $m$ is the dimension of $U_{a,u}$. We consider the product complex on the neighbourhood $U_a$, denoting it as $\mathcal{Q}_B(U)=(L^2\Omega_B^k(U)),\mathcal{D}_{B}(d_\varepsilon),d_\varepsilon)$.

\begin{proposition}
\label{proposition_de_rham_dolbeault_local_complex_correspondence}
Let $U_a=U$ be a tangent cone of a critical point, equipped with a local K\"ahler Hamiltonian Morse function.
Then the local de Rham cohomology  
group $\mathcal{Q}_B(U)=(L^2\Omega_B^k(U)),\mathcal{D}_{B}(d),d)$ (as discussed above for $\varepsilon=0$).
The cohomology groups of this complex inherit the Hodge bi-grading as discussed in Remark \ref{Remark_minimal_domain_equivalences}. Consider also the local Dolbeault complexes $^p\mathcal{P}_{\alpha,B}(U)=(L^2\Omega^{p,\cdot}(X;\mathbb{C}),\mathcal{D}_{\alpha}(P),P)$ where $P=\overline{\partial}_{\Lambda^{p,0}}$.
Then we have that 
\begin{equation}
    \mathcal{H}^q(^p\mathcal{P}^{\mu=0}_{\alpha,B}(U)) = \mathcal{H}^{p,q}(\mathcal{Q}_B(U))
\end{equation}
for all $\alpha \in [0,1]$ in each bi-degree $(p,q)$, where $^p\mathcal{P}^{\mu=0}_{1,B}(U)$ is the local equivariant sub-complex with eigenvalue $\mu=0$ with respect to the operator $\sqrt{-1}L_V$ where $V$ is the Hamiltonian vector field on the tangent cone.
\end{proposition}

\begin{proof}
We first prove this in the case of a cone $U=C_x(Z)$ where the Hamiltonian Morse function is $x^2$ (attracting factor) where $V$ is the Reeb vector field on the Sasaki structure that extends to the action of the cone as in Definition \ref{Definition_Kahler_Morse_function}.

The de Rham cohomology of a cone $\mathcal{H}^{p,q}(\mathcal{Q}_B(C(Z)))$ was computed in Lemma 6.16 of \cite{jayasinghe2023l2} where the harmonic representatives $\phi$ were shown to be given by the pullback under $\pi_Z: C(Z) \rightarrow Z=\{x=1\}$ of the harmonic representatives of the cohomology of $Z$, of degree $k \leq [\frac{l-1}{2}]$ where $l$ is the dimension of the link $Z$.
In particular that ${\partial_x} \phi=0$ (short for $\nabla_{\partial_x} \phi=0$).
We observed in the proof of Proposition 7.1 of \cite{jayasinghe2023l2} that these harmonic representatives, when restricted to the link $Z$ are in the basic cohomology corresponding to the Reeb foliation, and in particular $L_V \phi=0$.

Now consider elements in $\mathcal{H}^q(^p\mathcal{P}^{\mu=0}_{\alpha,B}(U))$. 
The condition that $\mu=0$ is equivalent to $L_Vu=0$.
Any element in the Dolbeault complex can be written as the sum of a \textit{tangential} form and a \textit{normal} form using the decompostion
\begin{equation}
\label{Dolbeault_form_decomposition_Morse}
    \Omega^{0,q}(C(Z))=\Omega^{0,q}(Z) \oplus (dx-\sqrt{-1}xJ(dx)) \wedge \Omega^{0,q-1} (Z)
\end{equation}
and the analyticity of the harmonic sections together with the boundary conditions corresponding to the choice of domain imply that the elements $u \in \mathcal{H}^q(^p\mathcal{P}^{\mu=0}_{\alpha,B}(U))$ are tangential on the attracting factor. In particular $\iota_Vu=0$ for $u$ in the local cohomology group. Then $L_Vu=\iota_V d u=\partial_Vu=0$.
Since such $u$ are also in the null space of the Dolbeault-Dirac operator, $(\partial_x + i\frac{1}{x}\partial_V)u=0$ which implies that $\partial_xu=0$ as well. The analysis for the harmonic sections of the Laplace-type operator in Proposition 5.33 of \cite{jayasinghe2023l2} using separation of variables shows that if $\Delta u=0$ and $\partial_x u=0$, then $\Delta_Z u=0$ where $\Delta_Z$ is the Hodge Laplacian on $Z$ (since the metric is K\"ahler). These are precisely the elements in the null space of $\mathcal{H}^{p,q}(\mathcal{Q}_B(C(Z)))$ once integrability on the cone is taken into account.
This proves the result for the case of $C(Z)$ which is an attracting factor.

The case when the critical point has an expanding fundamental neighbourhood follows from duality since the cohomology groups of the attracting and expanding factors are related by Poincar\'e duality for the de Rham complex, and Serre duality for the Dolbeault complex, both of which are realized by the Hodge star operator.
For the case of general tangent cones as in Definition \ref{Definition_Kahler_Morse_function}, the result follows from the K\"unneth formula.  
\end{proof}

\begin{remark}
The above result strengthens Proposition 7.1 of \cite{jayasinghe2023l2} where we showed that the harmonic representatives of the local de Rham cohomology are harmonic representatives of the local Dolbeault cohomology for a K\"ahler cone. 

Moreover we observe that this local rigidity can be applied to Hamiltonian's corresponding to isometric circle actions on stratified pseudomanifolds with wedge metrics and wedge symplectic structures where the Hamiltonian action is locally a K\"ahler action in a fundamental neighbourhoods of each critical point, with a local normal form in Definition \ref{Definition_Kahler_Morse_function}, and this allows one to prove the rigidity of $L^2$ cohomology for such actions using ideas in the theorem stated below.
\end{remark}

\begin{theorem}
\label{Theorem_instanton_correspondence_deRham_Dolbeault_non_intro}
Let $X$ be a resolved stratified pseudomanifold with a wedge K\"ahler metric and a K\"ahler Hamiltonian morse function $h$. Then the Witten deformed de Rham complex in Theorem 6.33 of \cite{jayasinghe2023l2} is isomorphic as a Hilbert complex to the direct sum over $p$ of the Witten deformed complexes given in Theorem \ref{theorem_small_eig_complex_intro} for the equivariant Dolbeault complexes $^p\mathcal{P}^{\mu=0}_{\alpha}=(L^2\Omega_{\mu=0}^{p,\cdot}(X;\mathbb{C}),\mathcal{D}^{\mu=0}_{\alpha}(P),P)$ where $P=\overline{\partial}_{\Lambda^{p,0}}$ for any $\alpha$. In particular the Witten deformed de Rham complex inherits the Hodge bi-grading.
\end{theorem}

\begin{remark}
In \cite{jayasinghe2023l2} we studied the Witten deformed de Rham complexes for stratified Morse functions given by Definition \ref{definition_stratified_Morse_function}, which does not seem to capture the  K\"ahler Hamiltonian Morse functions in Definition \ref{Definition_Kahler_Morse_function} at first glance.

We observed in Remark 6.21 of \cite{jayasinghe2023l2} that we can perturb a metric within a quasi-isometry class in a neighbourhood of the critical points of $h$ and arrange that it is product type in fundamental neighbourhoods of the critical points. Similarly, given a product decomposition as in equation \eqref{equation_decomposition_tangent_cone_1} of the tangent cone equipped with a local K\"ahler Hamiltonian Morse function, where the weights are not necessarily $\pm 1$, we can rescale the metrics on the factors $C_{r_j}(Z_j)$ to see that the K\"ahler Hamiltonian function is indeed a Morse function in the sense of Definition \ref{definition_stratified_Morse_function} with respect to a deformed K\"ahler metric on the tangent cone. This shows how we have a de Rham Witten instanton complex for K\"ahler Hamiltonian Morse functions.    
\end{remark}
    
\begin{proof}
The direct sum over $p$ of the Witten deformed Dolbeault subcomplexes in the statement of the theorem have graded vector spaces in each bi-degree $(p,q)$ which are isomorphic to $\oplus_{a \in crit(h)}  \mathcal{H}^q(^p\mathcal{P}^{\mu=0}_{\alpha,B}(U_a))$. Recall that the Witten deformed complexes and the undeformed complexes are quasi-isomorphic at the local and global levels (see Subsection \ref{subsection_deformed_Hilbert_complexes}). There is a similar isomorphism between the Witten deformed de Rham complexes and the (undeformed) de Rham complexes (see \cite[\S 6.5]{jayasinghe2023l2} at the local and global levels. Since we assume that $X$ has a K\"ahler structure it is easy to see that the de Rham operator $d={\partial}+\overline{\partial}$ respects the K\"ahler structure and the Hodge bi-grading. Moreover since the Morse function is K\"ahler Hamiltonian, the Witten deformed de Rham operator respects the K\"ahler structure and the Hodge bigrading as well. Moreover we observe that the Witten deformed de Rham operator is the conjugation of $d={\partial}+\overline{\partial}$ by $e^{\varepsilon h}$, the operator for the Witten instanton subcomplexes (direct sums for the Dolbeault case) which respects the Hodge decomposition.
The local cohomology groups for the undeformed complexes are isomorphic by Proposition \ref{proposition_de_rham_dolbeault_local_complex_correspondence} and this proves the result. 
\end{proof}

This result allows us to prove rigidity, as well as formulate Morse inequalities for many invariants in terms of local de Rham cohomology groups. The following is Theorem \ref{theorem_rigidity_Poincare_Hodge}, restated to aid the reader.

\begin{theorem}
\label{theorem_rigidity_Poincare_Hodge_non_intro}
In the same setting as Theorem \ref{Theorem_instanton_correspondence_deRham_Dolbeault_non_intro}, for $E=\mathbb{C}$, $\alpha=1$ the global equivariant Poincar\'e Hodge polynomials are rigid. 
Moreover we have that 
\begin{equation}
\label{equation_Rigidity_deRham_1}
    \chi_{y,b}(\mathcal{P}^{\mu=0}_{\alpha}(X),T_{\theta})=\chi_{y,b}(\mathcal{Q}(X)):=\sum_{p,q=0}^n y^p  b^q \dim {\mathcal{H}^{p,q}(\mathcal{Q}(X))}
\end{equation}
and for fundamental neighbourhoods $U_a$ of critical points,
\begin{equation}
\label{equation_Rigidity_deRham_2}
    \chi_{y,b}(\mathcal{P}^{\mu=0}_{\alpha,B}(U_a),T_{\theta})=\chi_{y,b}(\mathcal{Q}_B(U_a),T_{\theta})):=\sum_{p,q=0}^n y^p b^q \dim {\mathcal{H}^{p,q}(\mathcal{Q}_B(U_a))}
\end{equation}
for any $\alpha$,
and these satisfy {\textit{\textbf{de Rham type Poincar\'e Hodge inequalities}}} 
\begin{equation}
\label{equation_Rigidity_deRham_3}
    \sum_{a \in crit(h)} \chi_{y,b}(\mathcal{Q}_B(U_a),T_{\theta}))=\sum_{a \in crit(h)} \chi_{y,b}(\mathcal{Q}_B(U_a),Id))=\chi_{y,b}(\mathcal{Q}(X))+(1+b) \sum_{q=0}^{n-1} R_q b^q
\end{equation}
where $R_q$ are non-negative integers.
\end{theorem}

\begin{proof}
The rigidity for $E=\mathbb{C}$, $\alpha=1$ follows from Theorem \ref{Theorem_instanton_correspondence_deRham_Dolbeault_non_intro}, the definition of the global Poincar\'e Hodge polynomial, and the fact that for $\alpha=1$ the global cohomology for the Dolbeault and de Rham complexes are isomorphic as fleshed out in Remark \ref{Remark_minimal_domain_equivalences}.

Equations \eqref{equation_Rigidity_deRham_1} and \eqref{equation_Rigidity_deRham_2} follow from the definition of the Poincar\'e Hodge polynomials and Theorem \ref{Theorem_instanton_correspondence_deRham_Dolbeault_non_intro}.

The de Rham type Poincar\'e Hodge inequalities follow from the holomorphic Morse inequalities, restricted to $\mu=0$.
\end{proof}

\begin{corollary}
In the setting of Theorem \ref{theorem_rigidity_Poincare_Hodge_non_intro}, given the de Rham type Poincar\'e Hodge inequalities, we recover the de Rham Morse inequalities in \cite{jayasinghe2023l2} for $y=-1$. We get \textit{\textbf{de Rham type signature Morse inequalities}} for $y=+1$. In both cases we get formulas for the indices (Euler characteristic and Signature) when we set $b=-1$.
\end{corollary}

\begin{proof}
These follow from the fact that $\chi_{-1}$ and $\chi_1$ give the Euler characteristic and the signature invariant, as discussed in Subsection \ref{subsection_Poincare_Hodge_polynomial} and Theorem \ref{theorem_rigidity_Poincare_Hodge_non_intro}.  
\end{proof}

We refer to section 7.2.9 of \cite{jayasinghe2023l2} for conditions where the Morse function is perfect. Then the error polynomials vanish.

\begin{remark}[Generalizations]
\label{Remark_orientation_signature_Hamiltonian}
The above corollary generalizes equation (42) of \cite{witten1982supersymmetry} in our singular K\"ahler setting. 
We observe that for general actions there needs to be a sign accounting for the co-orientation of the normal fibers to the critical points (see also \cite[\S 2]{Donghooncircleactions_2018}) which is accounted for by the local Hodge $(p,q)$ degree in our setting. 
\end{remark}

\begin{remark}[Related versions and History]
\label{Remark_signature_Witten_Zhang_proofs}
The ideas used in the proof of Theorem \ref{theorem_rigidity_Poincare_Hodge_non_intro} can be used to construct de Rham type Morse inequalities for some other complexes that can be expressed using $L^2$ hodge numbers $h^{p,q}$, when there is a K\"ahler Hamiltonian Morse function, including the self-dual and anti-self dual complexes, where the indices are given by $1/2(\chi_{-1} \pm \chi_1)$ (see \cite[\S 7.2.8]{jayasinghe2023l2} for more details).
When there are K\"ahler isometric group actions, one can still derive the formulas corresponding to $b=-1$ for the equivariant indices.

In \cite[\S 3]{witten1982supersymmetry} Witten discusses how the formula above for the de Rham type signature Morse inequalities for $b=-1$ can be derived using a deformation when given a Killing vector field with isolated zeros, even on manifolds without K\"ahler structures. This can be done in the singular setting as well using the methods we have developed in this article.
In chapter 4 of \cite{Zhanglectures} Zhang derives the case for $y=-1, b=-1$ for arbitrary vector fields with isolated zeros, the Poincar\'e Hopf theorem, which is also sketched by Witten in \cite[\S 3]{witten1982supersymmetry}. Implicit in Proposition 4.10 of \cite{Zhanglectures} is a $\mathbb{Z}_2$ graded Witten instanton complex for the de Rham complex.

\end{remark}

\begin{example}[de Rham type Poincar\'e Hodge polynomial of a conifold]
\label{example_conifold_Poincare_poly}
We build on the example we studied in \cite[\S 7.3.4]{jayasinghe2023l2}.
Consider the quadric $Y_1 Y_4 - Y_2 Y_3 =0$ in $\mathbb{CP}^4$ with coordinates $[W:Y_1:Y_2:Y_3:Y_4]$. The zero set of this polynomial is a conifold which we shall call $X$. 
This space has a $(\mathbb{C}^*)^3$ algebraic toric action, $(\lambda, \mu, \gamma) \cdot [W: Y_1: Y_2: Y_3: Y_4] = [W: \lambda Y_1:\mu\gamma Y_2: \mu^{-1}\lambda Y_3: \gamma Y_4]$.
It is easy to verify that generically there are 5 fixed points on $X$, given by $[1:0:0:0:0]$, $[0:1:0:0:0]$, $[0:0:1:0:0]$, $[0:0:0:1:0]$, $[0:0:0:0:1]$. The last four fixed points are smooth and the first is singular.

The arguments in \cite{jayasinghe2023l2} shows that if we pick $\gamma=se^{i\theta}$, $\mu=\gamma^2$ and $\lambda=\gamma^4$, where $s<1$, the action on the chart $W=1$ corresponds to a Hamiltonian function where the singular fixed point is attracting. The fundamental neighbourhood is a cone over a link $S^2 \times S^3$, and this contributes an element in the local cohomology of bi-degree $(1,1)$, and an element in degree $0$ of bi-degree $(0,0)$. 

The smooth fixed points each have a one dimensional local cohomology group with elements in bi-degrees $(1,1)$ at $[0:0:0:0:1]$, $(2,2)$ at $[0:0:0:1:0]$ and at $[0:0:1:0:0]$, and $(3,3)$ at $[0:1:0:0:0]$.
It is easy to see that the \textbf{Poincar\'e Hodge polynomial} is $1+2by+2b^2y^2+b^3y^3$. In particular the \textbf{signature} is $0$  (given by setting $b=-1, y=+1$ for this 6 dimensional space.
\end{example}

\subsection{NUT charge and signature}
\label{subsection_Nut_charge_signature}

In \cite{gibbons1979classification}, Gibbons and Hawking study the Euler characteristic and signature, and their equivariant versions on 4 dimensional blackhole spacetimes, in particular using equivariant computations at \textit{nuts} (isolated fixed points) and \textit{bolts} (2 dimensional fixed point sets) when there are Killing vector fields to give formulas for gravitational action functionals and entropy. More generally on $d$ dimensional spaces with time-like Killing vector fields, the fixed point sets give rise to entropy (these are the obstructions to foliating the spacetimes with constant imaginary time) which include blackhole event horizons.
The $d-2$ dimensional fixed point sets are also called bolts in \cite{hunter1998action}, and their NUT charge is computed, along with contributions from other important geometric objects. Under the assumption that the spacetime has a symmetry expressed by a Killing vector field, physical quantities of interest in gravitation which are geometric invariants will respect the symmetry and equivariant localization can be used to compute many of them.

The topic is of great interest in mathematics and physics with many singular spaces arising naturally in its study. The smooth manifolds in \cite{gibbons1979classification} are obtained by compactifying them in Kruskal coordinates, with careful choices of coordinates in order to obtain a smooth manifold, and other choices of compactification lead to stratified pseudomanifolds where \textit{angles} at singularities correspond to black hole temperatures (see page 1 of \cite{de2023comments}).

More recent work in \cite{aksteiner2023gravitational} uses techniques and results in \cite{Donghooncircleactions_2018} which crucially use rigidity. Any extension of the study of gravitational instantons with conic singularities (e.g., \cite{toricgravitationalinstantonsBiquardGauduchon_2023}) will no doubt require equivariant signature theorem`s', with the correct choice of cohomology necessary to obtain equivariant computations for the corresponding choices of domains for operators involved, and for corresponding choices of integrals. Fixed point computations at nuts and bolts with more general localization computations are now studied widely on smooth manifolds and orbifolds in physics (see \cite{Equivariant_loc_Sparks_2024} where they are used in studying supergravity theories), and extending these to more singular spaces leads to many interesting questions. We discuss how this article extends other computations (such as in \cite{gupta2015supersymmetric}) related to the study of blackhole entropy in Subsection \ref{subsection_background_related_results}.

In this subsection we will show how some formulas for NUT charges and the signature given in \cite{gibbons1979classification} can be generalized to stratified pseudomanifolds with wedge K\"ahler metrics with K\"ahler Hamiltonian Morse functions, with plausible generalizations to spaces with local complex and almost complex structures near fixed points of a Killing vector field preserving the relevant structures using work in \cite[\S 7]{jayasinghe2023l2}, and beyond that using eta invariant based formulae for the equivariant signature on spaces with wedge metrics.
We study it using the Dolbeault complex with the minimal domain ($\alpha=1$). For geometric Witt spaces, the signature operator is essentially self-adjoint. 

We begin with equation (4.9) of \cite{gibbons1979classification}, which is derived for $S^1_{2\theta}$ actions on gravitational instantons where there is an imaginary time-like Killing vector field. The weights of the action at fixed points encode the surface gravity of black holes. This is simply the formula for the equivariant signature of an oriented 4 manifold with a Killing vector field $V$, with nut and bolt fixed points. For such a space with only nuts, we have that the signature can be written as
\begin{equation}
\label{equation_signature_3}
    \sum_{a \in zeroes(V)} \tau_3(p_a,q_a)=\tau
\end{equation}
where $\tau_3(p_a,q_a)=1/3(p_a q_a^{-1}+q_a p_a^{-1})$, where $p_a,q_a$ are relatively prime integers such that $\kappa_1^{-1}\kappa_2=q_a/p_a$ for weights $\kappa_1,\kappa_2$ of the local Hamiltonian corresponding to the circle action at the isolated fixed points. We use the notation $\tau_3$ since this is the third local quantity which we can define, that can be added up at fixed points. \textbf{\textit{We emphasize that this technique works only because the equivariant signature is rigid.}}

These weights are also the surface gravities (see equations (3.7), (3.8) of \cite{gibbons1979classification}). From now on we will make the simplifying assumption that the weights at the fixed points are $\kappa_1=p_a$ and $\kappa_2=q_a$.
Equation \eqref{equation_signature_3} is derived by taking the Atiyah-Bott-Lefschetz fixed point formula for the equivariant Signature, expanding it in powers of $\theta$, and using rigidity to write the global equivariant signature as the sum of the coefficients of the $\theta^0$ terms in the expansion at each fixed point. Equation (4.8) of \cite{gibbons1979classification} is derived by equating the coefficients of the $\theta^{-2}$ terms in the expansion, and for isolated fixed points yields
\begin{equation}
\label{equation_Nut_charge_vanishing}
    \sum_{a \in Zeroes(V)} N(p_a,q_a)=0 
\end{equation}
where for smooth fixed points $a$, the contribution is given by $N(p,q)=-(pq)^{-1}$.
The quantity $N(p,q)$ gives the NUT charge of the fixed points up to a constant, derived in \cite[\S 5]{gibbons1979classification}
We note that there are different conventions of overall constants chosen for the NUT charge (compare equations (5.9) and (5.10) of \cite{gibbons1979classification} with the formulas in Proposition 4.5 of \cite{aksteiner2023gravitational}), with more factors appearing when studying entropy and action to account for this.

\begin{remark}[The issue of orientation]
\label{Remark_orientation_issues}
While equation (4.8) of \cite{gibbons1979classification} is derived correctly using the Atiyah-Bott-Lefschetz fixed point theorem, the subsequent derivation of the NUT charge of nuts and bolts in section 5 of that article lead to equations (5.9) and (5.10) of that article, where the co-orientation of the bolts is not accounted for in equation (5.10). This leads to a sign error in equation (6.9) of that article, reappearing in equation (6.14).

The co-orientation is correctly accounted for in the proof of Proposition 4.5 of \cite{aksteiner2023gravitational}.
We refer the reader to \cite{Orientation_Morse_Bott_Rot_2016} and \cite[\S 4]{witten1982supersymmetry} for discussions on similar issues with picking orientations in Morse-Bott inequalities and equivariant de Rham type signature theorems (see also Remark \ref{Remark_orientation_signature_Hamiltonian}).

We emphasize the importance of these signs in physical interpretations of the computed quantities which have lead to deep philosophical questions in versions of the theory (see \cite[\S 3]{de2023comments} for details and references). 
\end{remark}

Equation (6.14) in \cite{gibbons1979classification} is followed by the remark that the entropy computation has contributions from the \textit{NUT charges} of both nuts and bolts, which cannot be attributed to individual nuts and bolts due to a translational freedom, but makes sense since the total value is invariant since the ``total sum of the NUT charges is zero for a compact manifold". However the contributions in equations (6.9), (6.14) of that article only vanish once the sign error discussed in the above remark is fixed. Moreover, the coordinate computations in both \cite{gibbons1979classification} and \cite{aksteiner2023gravitational} show that the Nut charge is actually the evaluation of a characteristic polynomial of the curvature, which vanishes globally. 

That the right hand side of equation \ref{equation_Nut_charge_vanishing} (the total NUT charge) vanishes can be deduced as a consequence of observation that the series expansion in $\theta$ of the global equivariant signature $\tau(\theta)$ cannot have powers $\theta^k$ for $k < 0$ since it is a smooth function in $\theta$, including at $\theta=0$ where $\tau(0)$ is an integer. \textit{\textbf{This is weaker than rigidity and maybe useful for deriving similar results in cohomology theories (such as for singular cohomology, using the formulas of \cite{baum1979lefschetz}) where the equivariant index is not rigid.}}
This vanishing is also related to an observation by Bott (see page 244 of \cite{bott1967vector}) of a vanishing phenomenon for a modification of the Bott-residue formula for certain equivariant polynomials in curvature, which he observes holds for the $L$ polynomials.
We refer to related formulas in Theorem 3.2, Theorem 3.3, Corollary 3.4, Lemma 3.5 of \cite{Donghooncircleactions_2018}, and results in \cite[\S 3.2]{aksteiner2023gravitational} which are derived as consequences of rigidity. 

We now extend formulas \eqref{equation_signature_3} and \eqref{equation_Nut_charge_vanishing} for the two four dimensional spaces that we discussed earlier in this section and discuss some features.

For the singular fixed point $a_3$ of the action on the algebraic surface in Example \ref{example_quadric_Morse_inequalities}, we can set $\lambda^2=e^{ip(2\theta)}$ and $\mu^2=e^{iq (2\theta)}$ to expand the local equivariant $\chi_1$ invariant given in Example 7.2.8 of \cite{jayasinghe2023l2} to obtain $N(p,q)=\frac{-1}{2} (pq)^{-1}$, and $\tau_3(p,q)=\frac{1}{2}(\frac{1}{3}(pq^{-1}+qp^{-1}))$. We observe that these differ from the corresponding quantities in the smooth setting by a factor of $1/2$, consistent with the fact that the singularity is an orbifold singularity with local covering degree $2$.

For the singular fixed point $a_3$ of the action on the algebraic surface in Example \ref{example_singular_calabi_yau}, we can set $\lambda/\mu=e^{ip(2\theta)}$ and $1/\mu^4=e^{iq (2\theta)}$ to expand the local equivariant $\chi_1$ invariant given in \cite[\S 7.3.6]{jayasinghe2023l2} to obtain $N(p,q)=-(pq)^{-1}$, and $\tau_3(p,q)=(\frac{1}{3}(pq^{-1}+qp^{-1}))$ similar to the smooth setting. The tangent cone at the singularity is given by a product of a smooth disc and a cone over a $T^{3,4}$ knot with angle $6\pi$. The reason the formula is similar to that in the smooth setting is because we pick the equivariant $\chi_1$ invariant corresponding to Dolbeault complex with the minimal domain ($\alpha=1$) to get the correct expansion.

For the singular fixed point $a_1$ which is a non-orbifold singularity, we can set $\lambda^4=e^{ip(2\theta)}$ and $\mu^4=e^{iq(2\theta)}$, to expand the local equivariant $\chi_1$ invariant to obtain $N(p,q)=-3(pq)^{-1}$ and $\tau_3(p,q)=1/2-(1/4)(pq^{-1}+qp^{-1})$.

\begin{remark}[$N(p,q)$ is independent of the choice of domain]
\label{Remark_equivariant_volumes}
The quantity $N(p,q)$ and its generalizations in higher dimensions can be computed using the equivariant signature. We observe that equation \eqref{equation_Nut_charge_vanishing} corresponds to Corollary 3 of \cite{bott1967vector}, which shows one way in which the Bott residue theorem generalizes the classical residue theorem. In the smooth and orbifold cases $N(p,q)$ corresponds to the Pfaffian of matrices corresponding to group actions at fixed points.
Since the local Lefschetz numbers for the Signature operator for various choices of Fredholm domains (and the Baum-Fulton-Quart versions of $\chi_1$ invariants) differ by finite sums of powers of $e^{i\theta}$ (which have power series expansions in $\theta$ which have no negative power terms), the coefficient of the term with lowest power $\theta^{-n}$ of the expansion of these local Lefschetz numbers in powers of $2\theta$ are the same.

\end{remark}

\begin{remark}[Formulas for other cohomology theories]
\label{Remark_formulas_NUTs_other_cohom_theory}
It is easy to observe that if one chooses other cohomology theories, the formulas may change for general singular spaces, and it is not clear how to define an analog of $\tau_3$ without rigidity. If there is a Hilbert complex of the type we study, corresponding to a self-adjoint extension of the Dolbeault complex on $X^{reg}$, then the choice of domain will dictate these changes and will be related to the propogators corresponding to the Dirac-type operator. For theories such as those for which we discussed conjectural holomorphic Morse inequalities in Subsection \ref{subsubsection_conjectural_inequalities}, considerations such as those discussed in Remark \ref{Remark_considerations_integrals_cup_products} apply. 
\end{remark}

\subsection{Spin-Dirac complex and fractional powers of canonical bundles}
\label{subsubsection_spin_Morse_inequalities}

It is well known that a spin structure on a smooth space $X$ with a complex structure corresponds to a square root of the canonical bundle $K$ when it exists. That is we have a line bundle $L$ on $X$ such that $L^{\otimes 2}=K$, and a spin Dirac operator 
\begin{equation}
\label{identified_complex_spin}
    D=\sqrt{2}(\overline{\partial}_L+\overline{\partial}_L^*) : \bigoplus_{q \text{  even}} (\Omega^{0,q}(X;L)) \rightarrow \bigoplus_{q\text{  odd}} (\Omega^{0,q}(X;L)),
\end{equation}
(see for instance Theorem 2.2 of \cite{hitchin1974harmonic}, Section 3.4 of \cite{friedrich2000dirac}).
We denote the associated elliptic complexes as $\mathcal{S}=(L^2\Omega^{0,q}(X;L), \mathcal{D}_{\alpha}(P),P)$ where $P=\overline{\partial}_L$.

\begin{remark}[Self duality of spin complex]
\label{remark_self_dual_spin_serre}
It is easy to see that the complex $\mathcal{S}$ is a \textit{\textbf{self dual complex under Serre duality}} since $L^{-1} \otimes L^{\otimes 2}= L$.
\end{remark}

In the smooth setting, a square root of the canonical bundle $K$ only exists when the first Chern class is even. Hirzebruch studied generalizations of this when $C_1(X)=0 \text{ mod } (N)$, considering the Dolbeault cohomology of the complex twisted by a complex line bundle $L^{\otimes k}$ such that $L^{\otimes N}=K$, for integers $0 < k < N$ (see \cite{hirzebruch1988elliptic}), relating it to elliptic genera of level $N$ (c.f. \cite{DessaiToric16})

The following result is an analog of Proposition 2.1 of \cite{atiyah1970spin}, which allows one to drop the assumption that the action lifts to the bundles that we used in the proof of the Morse inequalities for the bundles that we consider here.

\begin{proposition}
\label{Proposition_action_lifts_to_lifted_bundle}
Let $X^{2n}$ be a resolution of a compact pseudomanifold, equipped with a K\"ahler wedge metric on which a compact connected Lie group $G$ acts effectively, preserving the K\"ahler structure. Let $L$ be a choice of line bundle such that $L^{\otimes N}=K$, where $N>1$ is an integer. Then there exists canonically a Lie group $G_1$ and a homomorphism $H:G_1 \rightarrow G$, which is either the identity or a covering such that $G_1$ acts on $L$ inducing via $H$ the given action of $G$ on $X$.
\end{proposition}

\begin{proof}
We follow the proof of Proposition 2.1 of \cite{atiyah1970spin}, adapting it to the singular K\"ahler setting. The K\"ahler wedge metric on $X$ induces a canonical Hermitian metric on $Q$, by which we denote the principal tangential $U(n)$ bundle of $X$ (i.e., the unitary frame bundle on $X$). 

This is defined by the canonical Riemannian connection (Chern-connection) on $Q$, together with the invariant metric on $U(n)$.

The determinant map is a submersion from $U(n)$ to $U(1)$, and the  invariant metric on $U(n)$ induces an invariant metric on $U(1)$ by the pushforward of the determinant map. We denote by $\widetilde{Q}$ the principal $U(1)$ bundle which is the determinant bundle of $Q$. This then has a natural metric induced by the one on $Q$.

If the holomorphic wedge tangent bundle is the bundle associated to $Q$ with representation $r: U(n) \rightarrow \mathbb{R}^{n} \otimes \mathbb{C}$, then the canonical bundle is the line bundle associated to $\widetilde{Q}$ given by the representation $det(r): U(1) \rightarrow \mathbb{C}$.

The bundle $L$, when it exists, corresponds to a choice of $N$-th root, $s: x \mapsto x^{1/N}$ from $U(1)$ to an $N$ fold covering of $U(1)$, which is diffeomorphic to $U(1)=S^1$. This yields a further principal $U(1)$ bundle $Q_1$ which yields the associated line bundle $L$ via the representation $s \circ det(r): U(1) \rightarrow \mathbb{C}$.

\begin{remark}
    In the case where $N=2$, this choice of squareroot corresponds precisely to the choice of a spin structure (see the proof of Theorem 2.2 of \cite{hitchin1974harmonic}, and that of Proposition 2.1 of \cite{atiyah1970spin}).    
\end{remark}

There is a canonical metric on $Q_1$ as well. The map $s$ has an inverse $s^{-1}: x \rightarrow x^N$ which is a covering map which we can use to pullback the invariant metric on $U(1)$ in $\widetilde{Q}$ to $Q_1$.
Any $g \in G$ induces diffeomorphisms $g_Q: Q \rightarrow Q$ and $g_{\widetilde{Q}}:\widetilde{Q} \rightarrow \widetilde{Q}$, and $g_{Q_1}: Q_1 \rightarrow Q_1$. These are isometries since the metrics on $Q, \widetilde{Q}, Q_1$ are canonically associated to the metric on $X$. 

For any element of the Lie algebra $L(G)$, the corresponding infinitesimal isometry is a vector field on $\widetilde{Q}$ which preserves the strata on the base $X$, which commutes with the action of $L(U(1))=\mathbb{R}$. Lifting this vector field to $Q_1$, we obtain an infinitesimal isometry of $Q_1$, which commutes with the action of $\mathbb{R}$ which is the Lie algebra of the Lie groups $U(1)$ on both $\widetilde{Q}$ and $Q_1$. Thus we have an infinitesimal action of $L(G)$ on $Q_1$ commuting with $U(1)$. Exponentiating gives an action of $\widetilde{G}$, the universal cover of $G$, on $Q_1$ commuting with $U(1)$. Let $G_1$ be the image of $\widetilde{G}$ that acts on $Q_1$ modulo the stabilizer. We have a homomorphism $H: G_1 \rightarrow G$ and the kernel is of order $d$ such that $d$ divides $N$.
\end{proof}

The following result (restated as Theorem \ref{Theorem_harmonic_spinors_vanish} in the introduction) is a strengthening of vanishing of index results, including those in \cite{atiyah1970spin}, \cite[\S 6]{DessaiToric16}, in the case where the group actions are isometric K\"ahler Hamiltonian. In the case when $N=2$ and $k=1$, this corresponds to the vanishing of harmonic spinors.

\begin{theorem}
\label{Theorem_harmonic_spinors_vanish_non_intro}
Let $X$ be a \textbf{smooth} K\"ahler manifold with a Hermitian line bundle $L$ such that $L^{\otimes N}=K_X$, where $N>1$ is an integer. Let $E=L^{k}$, where $0< k <N$ is an integer. We consider the complex $\mathcal{P}=(L^2\Omega^{0,q}(X;L), \mathcal{D}_{\alpha}(P),P)$ where $P=\overline{\partial}_E$.

If there is a K\"ahler action of a compact connected Lie group G with isolated fixed points (with at least one fixed point), then the cohomology of the Dolbeault complex twisted by the bundle $E$ vanishes in all degrees. In particular the twisted Dirac operator is invertible.
\end{theorem}

\begin{proof}
We use the holomorphic Morse inequalities for the action. By Proposition \ref{Proposition_action_lifts_to_lifted_bundle} above, there will always be a lift of the action to $L$, and thus to $E$. By Remark \ref{remark_density_of_circle_actions}, it suffices to prove this for a circle action with isolated fixed points.

We see that for a fixed point of the action $a$ with fundamental neighbourhood $U_a$, the local Morse polynomial is given by equation \eqref{equation_Naki_Mayina}, which reduces to
\begin{equation}
    b^{n_a} \prod_{\gamma_j^a >0} \frac{e^{i\gamma_j^a \theta (k/N)}}{1-e^{i\gamma_j^a \theta}} \prod_{\gamma_j^a <0} \frac{e^{i(N-k)|\gamma_j^a| \theta/N}}{1-e^{i|\gamma_j^a| \theta}}
\end{equation}
and the local Morse polynomial of the dual complex is given by
\begin{equation}
    b^{n-n_a} \prod_{\gamma_j^a >0} \frac{e^{-i(N-k)|\gamma_j^a| \theta(k/N)}}{1-e^{-i|\gamma_j^a| \theta}} \prod_{\gamma_j^a <0} \frac{e^{-i\gamma_j^a \theta/N}}{1-e^{-i\gamma_j^a \theta}}.
\end{equation}
This is because the trace, $E_{a}(\theta)$ is the trace of the action over the appropriate roots of the section of the canonical bundle $dz_1\wedge dz_2 \wedge... \wedge dz_n$, for the complex and the dual complex.

We use $\lambda_j:=e^{i\gamma_j^a \theta}$ for brevity. 
Now the result is a consequence of the fact that the corresponding \textit{classical} Morse polynomial is $0$. 
This follows from the observation that the Laurent expansions of the contributions at fixed points for $\textit{Im}(\theta)>0$ have only terms $e^{ic\theta}$ with $c>0$, while the expansions with $\textit{Im}(\theta)<0$ have only terms with $c<0$.

This implies that the cohomology must vanish identically in all degrees, and the Dirac-type operator is invertible.
\end{proof}

The result of \cite{atiyah1970spin}, that the $\widehat{A}$ genus vanishes on spin manifolds with circle actions, holds even for spaces such as tori which have circle actions which are not Hamiltonian. It is well known that the 2 torus has 4 spin structures, one of which corresponds to the trivial bundle which has non-vanishing cohomology, even though the index of the spin Dirac operator vanishes. This is a space which does not admit positive curvature metrics. On the other hand given Fano manifolds that admit positive curvature metrics, the Lichnerowicz estimates show that harmonic spinors vanish. The distinction between group actions having fixed points vs those that don't is akin to the difference between positive curvature and flat metrics.

\begin{remark}[Generalization for complex manifolds with Bialynicki-Birula decompositions]
It is easy to see that this holds even when the action is not K\"ahler, albeit under the assumption that there exists a Bialynicki-Birula decomposition where the Morse inequalities will hold by \cite{wu2003instanton}.
\end{remark}

Most of the corollaries in \cite{atiyah1970spin} have analogs for K\"ahler Hamiltonian actions that we study here. This is interesting since one can find such bundles for any smooth K\"ahler manifold, even when the first Chern class is not even. We provide a sample of such a corollary for demonstration.

\begin{corollary}
If $X$ is a K\"ahler manifold with harmonic spinors, or with cohomology in the Dolbeault complex twisted by a bundle $K^{k/N}$ for integers $N>k>1$, then it does not admit any K\"ahler actions of compact connected Lie groups, with isolated fixed points and at least one fixed point. In particular, if there are compact K\"ahler Lie group actions on a K\"ahler manifold admitting spinors, the $L^2$ Euler characteristic is zero, or there are non-isolated fixed points.
\end{corollary}

\begin{proof}
Under the assumptions, if there were a compact K\"ahler Lie group action with isolated fixed points, it would contradict the Theorem above. If it does admit a Lie group action and there are no fixed points, then by the $L^2$ de Rham Lefschetz fixed point theorem, the Euler characteristic is $0$.
\end{proof}

The main result of \cite{atiyah1970spin} does not assume that the fixed points are isolated.
More can be said for the case of non-isolated fixed points in the smooth setting using the results in \cite{wu1998equivariant}, which we leave to the interested reader.

In Example 7.29 of \cite{jayasinghe2023l2}, we computed that the Spin Lefschetz number for the singular quadric (the normal variety which is a complete intersection hypersurface in $\mathbb{CP}^3$ that we studied in Example \ref{example_quadric_Morse_inequalities}) is $0$. 
It is easy to use the Morse inequalities of this article to see that there are no harmonic spinors in that example. We will next prove a result which generalizes this observation.

We observed in \cite[\S 7.2.3]{jayasinghe2023l2} that for normal algebraic varieties with vanishing higher local cohomology of fundamental neighbourhoods of fixed points with the choice of domain for the complex corresponding to an attracting fixed point, the local cohomology of the minimal domain of the $\overline{\partial}$ operator is in one to one correspondence with the local ring, using it to match the local and global Lefschetz numbers in the $L^2$ and \cite{baum1979lefschetz} settings in Proposition 7.8 of that article.

We also studied how the local cohomology of the Spin Dirac complexes of complete intersection hypersurfaces with canonical singularities admit a presentation as a module over the holomorphic functions generated by a square root of the Poincar\'e residue, (see \cite[Theorem 1.1]{Weber_Residueforms_2005} c.f. \cite[\S 7.3.3]{jayasinghe2023l2}). Since the Baum-Fulton-MacPherson theory matches the $L^2$ theory for such normal varieties, the dual local cohomology also matches, and the canonical bundle has the same local trivializing harmonic sections given by the Poincar\'e residue. 

Given the vanishing set $X$ of a polynomial $G$ in $\mathbb{C}^n$, the Poincar\'e residue is 
\begin{equation}
\label{equation_poincare_residue}
    \omega=Res_{\mathbb{C}^n|X} \frac{dz_1 \wedge dz_2 \wedge...\wedge dz_n}{G} = \frac{dz_2 \wedge...\wedge dz_n}{\partial G/\partial {z_1}}= \frac{dz_1 \wedge dz_3 \wedge...\wedge dz_n}{\partial G/\partial {z_2}}=...
\end{equation}
and we refer to \cite{reid2012val,YauBong2013completeintersections} for more details (see also \cite[\S 7.3.3]{jayasinghe2023l2}). In this setting we can generalize Theorem \ref{Theorem_harmonic_spinors_vanish_non_intro}.

\begin{theorem}[Vanishing of harmonic spinors]
\label{Theorem_harmonic_spinors_vanish_singular}
Let $X^{2n-2}$ be a complete intersection complex hypersurface in $\mathbb{CP}^n$ with a K\"ahler $\mathbb{T}^{n-1}$ action with isolated fixed points, that descends from a torus action on $\mathbb{CP}^n$. We assume that $X$ is the vanishing set of a single degree $d$ polynomial $f$ in $\mathbb{CP}^n$, a normal algebraic variety with canonical singularities, has local cohomology at fixed points concentrated in a single degree, and $n+1-d \geq 2$.
Then we have that the index of the spin Dirac operator for the domain corresponding to $\mathcal{D}_{1}(\overline{\partial})$ vanishes.
If in addition the torus action is K\"ahler Hamiltonian on $X$ with isolated singularities, then the spin Dirac operator is invertible for the domain corresponding to $\mathcal{D}_{1}(\overline{\partial})$.
\end{theorem}

The assumption that the local cohomology at fixed points is concentrated in a single degree, is equivalent to saying that the local cohomology for the Dolbeault complex corresponding to the trivial bundle and $\alpha=1$ for attracting fixed points corresponds to the holomorphic functions on the fundamental neighbourhood.

In the smooth setting, the condition $(n+1-d)>0$ corresponds to Fano manifolds which admit positive curvature metrics, in which case the result follows from the Lichnerowicz formula.

\begin{proof}
We can write any torus action on $\mathbb{CP}^n$ as
\begin{equation}
    [Z_0:Z_1:...:Z_n] \mapsto [Z_0: e^{i\theta_1}Z_1:...:e^{i\theta_n}Z_n].
\end{equation}
by taking the compactification of $\mathbb{C}^n$ with complex coordinates $z_k=Z_k/Z_0$ where the torus action can be diagonalized to act by $z_k \rightarrow e^{i\theta_k}z_k$.
The assumption that there are only canonical singularities implies that the canonical bundle has a local trivializing harmonic section generated by Poincar\'e residues on each chart (see \cite[Theorem 1.1]{Weber_Residueforms_2005}). Say that the polynomial $f$ has a monomial term $Z_0^{c_0} Z_1^{c_1} ... Z_n^{c_n}$ for $c_j \in \mathbb{N} \cup \{0\}$.

Since the torus action has fixed points only on the origins of the charts $Z_k=1$ which lie on the vanishing set of the polynomial, we need to study the local cohomology groups of the twisted Dolbeault complex at these points. The squareroot of the canonical bundle is locally trivial, and is trivialized by a squareoot of the section $s^2_k$ given by the Poincar\'e residue.
The trace of $T_{\theta}$ over these local cohomology groups then factors into the trace over the local section $s_k$ and the trace over the local holomorphic function for attracting factors, and we can compute over the expanding factors using duality arguments.

With the form of the monomial that we assumed, the Poincar\'e residue over the chart $Z_k=1$ is given by $s_k^2=Res_{\mathbb{C}^n|X} \eta$ where
\begin{equation}
    \eta= \frac{dZ_0 \wedge dZ_1 \wedge... dZ_{k-1}\wedge dZ_{k+1} \wedge... \wedge dZ_n}{Z_0^{c_0} Z_1^{c_1} ...Z_{k-1}^{c_{k-1}} Z_k^{c_{k}} Z_{k+1}^{c_{k+1}}... Z_n^{c_n}}  Z_k^{d-n},
\end{equation}
the trace of the action over which is given by $e^{i\kappa}$ where 
\begin{equation}
    \kappa=\sum_{j=1}^n \theta_j (1-c_j) +\theta_k (-1+d-n)
\end{equation}
since the trace over the Poincar\'e residue is the same as that over the form $\eta$ on $\mathbb{C}^n$ by the action.

If $f$ has only one monomial, then such a hypersurface is smooth and the result follows from Theorem \ref{Theorem_harmonic_spinors_vanish_non_intro}.
Thus we can assume that there are at least two distinct monomials in $f$, at least one of which therefore has some $c_v \geq 2$ for some $v \in \{0,...,n\}$. 
Since $\sum_{j=0}^n c_j=d$, and $n+1-d \geq 2$ the pigeonhole principle shows that there are at least two numbers $l$ such that $c_l=0$. Let us pick one such $l$.
We see that on charts $Z_k=1$ for $k \neq l$, the trace of the geometric endomorphism $T_{g}$ for $g \in G$ corresponding to the action over the local section 
$s_k^2$ contributes a factor $e^{im\theta_l}$ where $m=1-c_l=1>0$. The trace over $s_l^2$ contributes a factor $e^{im\theta_l}$ where $m=1-c_l-1+d-n=d-n<0$ when $n+1-d>1$.

On the chart where $Z_l=1$, the local holomorphic functions are power series in products of $Z_k/Z_l$ for $k \neq l$, which shows that the trace of the action over the local holomorphic functions only contribute factors $e^{im\theta_l}$ where $m \leq 0$.
For the local cohomology of the twisted complex $\mathcal{S}$, from our observations for the Poincar\'e residue, and for the section given by its square root, we see that there are only factors $e^{im\theta_l}$ where $m > 0$ where we use the self-duality under Serre duality property and the fact that the local Morse polynomial is computed in the local cohomology groups corresponding to which the traces converge as power series in the variable $\lambda_l=se^{i m\theta_l}$ for $s<1$.

On all other charts, we see by our discussion and arguments similar to that above, that the traces over sections in the local cohomology groups of the twisted complex contribute factors $e^{im\theta_l}$ where $m>0$.

Since there are at least two numbers $l$ for which $c_l$ vanish, the positivity/negativity of exponents of powers of $e^{i\theta_l}$ discussed above hold for two values of $l$ between $0$ and $n$. 
The space $X$ is a hypersurface there is only a $\mathbb{T}^{n-1}$ action. Since we have two numbers $l$ with the required positivity properties, we see that there the restricted torus action on $X$ still has at least one variable $\lambda_l=e^{i\theta_l}$ which has the required positivity properties.

Since the local Morse polynomials only have positive powers of $\lambda$, self-duality under Serre duality shows us that the dual local Morse polynomials only have negative powers of $\lambda$ appearing in the expansion (and that series converges for $s>1$). Thus the classical Morse polynomial is identically $0$. This proves the result when the action is a K\"ahler Hamiltonian with isolated fixed points.

When one only has a complex action, we only have the Lefschetz fixed point theorem. We saw in \cite[\S 7.2.3]{jayasinghe2023l2} that under our assumptions we can sum the formal series and write the Lefschetz number as a rational function in the torus variables $\lambda_j$ (without renormalization for algebraic torus actions). Thus the local Lefschetz numbers are rational functions in the variable $\lambda_l$, and by the Lefschetz fixed point theorem the global Lefschetz number is a rational function in $\lambda_l$ as well. For Fredholm complexes, the global Lefschetz numbers can only have poles at $0$ or $\infty$ by the finiteness of the index. However all the local Lefschetz numbers vanish at $\lambda_l=0$, and they even vanish at $\infty$ by the self-duality under Serre duality (see Remark \ref{Remark_Self_duality_symmetries}).
This shows that the global Lefschetz number is a rational function that vanishes at $\lambda_l=0$ and $\lambda_l=\infty$, implying it is identically $0$. This proves the theorem.

\end{proof}

\subsection{Rarita-Schwinger operators}
\label{subsection_RS_ops}

Witten used his deformation technique to study the equivariant indices of spin Dirac and Rarita-Schwinger (RS) operators in \cite{witten1983fermion}, pioneering the technique in studying character valued indices.
There are various notions of Rarita-Schwinger operators and Rarita Schwinger fields studied in the literature, some of which correspond to various choices of gauge. 
In this subsection we discuss equivariant character formulas for these operators in the smooth setting, and discuss extensions to the singular spaces that we study. We do this by introducing a $k$-Rarita Schwinger operator which captures the versions of the operator in the literature.

\begin{remark}
\label{remark_Wittens_clarifications}
Different notions of the Rarita-Schwinger operator differ by adding or (in the K-theory sense) subtracting some number of copies of the Dirac operator, and therefore the different formulas differ by an integer multiple of the the spin Dirac index.
There is ``the" Rarita Schwinger operator studied in supergravity, appearing for instance in \cite{witten1983fermion,alvarez1984gravitational} and apparently there are slightly different ways to describe its index in terms of simpler operators.

Outside of supergravity, one can approximate the RS operator as a Dirac operator acting on spinors with values in the (complexified) tangent bundle, and we will call that operator $D_{TX}$.
If one uses that definition, then ``the" RS operator has an index that is the index of $D_{TX}$ minus the index of the spin Dirac operator.

We refer to the Appendix of \cite{freed2002k} which apparently grew out of the question, \textbf{why for the computation of anomalies is the Rarita-Schwinger field on an even dimensional manifold $X$ treated as a spinor field coupled to the
virtual bundle $TX - 1$, whereas in 11 dimensional M-theory it is treated as a spinor field coupled
to $TX - 3$ ?}
\end{remark}

\begin{remark}
Henceforth \textbf{we shall assume that the K\"ahler spaces we study have spin structures}. The techniques for spin$^{\mathbb{C}}$ Dirac operators we introduced in \cite{jayasinghe2023l2} can be used to extend the equivariant index formulas to even dimensional stratified pseudomanifolds with (local) wedge almost complex structures and spin structures.
\end{remark}

In light of Remark \ref{remark_Wittens_clarifications} we define the \textbf{$k$-Rarita-Schwinger operator} (also denoted as $k$-RS operator), as the Dirac operator acting on spinors with values in the tangent bundle plus $k$ copies of the spin Dirac operator (in the $K$-theory sense). Thus the index of the $k$-RS operator is
\begin{equation}
    Ind(D_{k-RS})=Ind(D_{TX})+k Ind_{D_S}
\end{equation}
where $D_S$ denotes the spin Dirac operator in this subsection.

In \cite{witten1983fermion} Witten computes the local equivariant indices of the Rarita-Schwinger operator using the fact that it is the equivariant index of the Dolbeault complex with coefficients in 
\begin{equation}
\label{equation_twist_RS_operator}
    F=(TX\otimes \mathcal{K}^{1/2} \oplus k \mathcal{K}^{1/2}) 
\end{equation}
where $\mathcal{K}$ is the canonical bundle whose squareroot can be identified with the spinor bundle, for $k=-1$. We can use the fact that $TX=TX^{1,0}\oplus TX^{0,1}$ and that the Chern characters satisfy $ch(E_1 \oplus E_2)=ch(E_1)+ch(E_2)$ to study the equivariant formulas, following \cite{witten1983fermion}.
If the weights of the group action at a fixed point $a$ are $\gamma_{a,j}$ on a space of complex dimension $n$, then the equivariant index formula is 
\begin{equation}
\label{equation_RS_equivariant_formula_smooth}
     \sum_{a \in Fix(f)} \Bigg[ \bigg (\prod_{j=1}^{n} \frac{\lambda_{a,j}^{1/2}}{1-\lambda_{a,j}} \bigg) \bigg( \sum_{j=1}^n (\lambda_{a,j}+\lambda_{a,j}^{-1} +k) \bigg) \Bigg]
\end{equation}
equivalently
\begin{equation}
    \sum_{a \in Fix(f)} \bigg(\frac{i}{2} \bigg)^n \Bigg[ \bigg (\prod_{j=1}^{n} \frac{1}{\sin(\frac{1}{2} \gamma_{a,j} \theta)} \bigg) \bigg( \sum_{j=1}^n (2 \cos(\gamma_{a,j} \theta) +k) \bigg) \Bigg]
\end{equation}
where $exp(i \gamma_{a,j} \theta) =\lambda_{a,j}$.
This is essentially equation (41) of \cite{witten1983fermion} (modulo a minor typo there, c.f. equation (65) \cite{alvarez1984gravitational}).
It is easy to check that the index satisfies a product formula
\begin{multline}    \label{equation_product_formula_for_RS}
    Ind (D_{k-RS})(X_1 \times X_2)=Ind (D_{k-RS})(X_1) Ind (D_{S})(X_2)\\
    + Ind (D_{k-RS})(X_2) Ind (D_{S})(X_1) - k Ind (D_{S})(X_1 \times X_2)
\end{multline}
which is simply a consequence of separation of variables in the smooth setting and is equation (45) in \cite{witten1983fermion} for the case where $k=-1$, and and equation (11) of \cite{HommaSemmelmannRaritaSchwinger2019} for $k=+1$.

In particular, since the spin Dirac operator is rigid in the smooth setting, if a space $X$ is a product of two spin manifolds then since the indices of the spin Dirac operator on each factor is zero, the index of the RS operator on $X$ vanishes as well.

The RS operator corresponds to one of the infinitely many representations studied appearing in the Witten genus \cite{bott1989rigidity}. Together with the Atiyah-Hirzebruch rigidity result for the spin Dirac operator \cite{atiyah1970spin}, and the rigidity result for the k-RS operator for $k=0$ in Theorem 2 of \cite{landweber1988circle} we get rigidity for general $k$-RS operators in the smooth setting.


In the singular setting, the $k$-RS operator will still correspond to the Dolbeault complex twisted by the bundle \eqref{equation_twist_RS_operator}. Now there are complications, even if one restricts oneself to the minimal domain of the twisted Dolbeault complex.

Let us consider the toric variety in Example \ref{example_quadric_Morse_inequalities}. The contributions at the smooth fixed points follow from equation \ref{equation_RS_equivariant_formula_smooth}. Since the index of the spin Dirac operator is $0$ (see Example 7.30 of \cite{jayasinghe2023l2}), we see that the index of the $k-RS$ operator for this space is independent of $k$.

We computed the Lefschetz numbers for the complex twisted by $^wT^{1,0}X$ ($p=1$). We also showed how the local harmonic section trivializing the canonical bundle in a neighbourhood of the singularity can be described using the Poincar\'e residue. Using this it is easy to compute the local Lefschetz number corresponding to the Dolbeault complex twisted by $T^{1,0}X \otimes \mathcal{K}^{1/2}$, and the global Lefschetz number for the group action can be computed as 
\begin{equation}
\label{equation_nodal_RS_1}
    \frac{(\lambda^2+\mu^2+\lambda\mu+\lambda\mu) (\lambda \mu)^{1/2}}{(1-\lambda^2)(1-\mu^2)} +\frac{(\mu/\lambda+ 1/\lambda^2) (\mu/\lambda^3)^{1/2}}{(1-\mu/\lambda)(1-1/\lambda^2)}+\frac{(\lambda/\mu+1/\mu^2) (\lambda/\mu^3)^{1/2}}{(1-\lambda/\mu)(1-1/\mu^2)} = \frac{1}{(\lambda \mu)^{1/2}}
\end{equation}
where the summands on the left hand side are the contributions from the singular point and the points $a_1$ and $a_2$ respectively, and we refer to Section 7.3.3 of \cite{jayasinghe2023l2} for some relevant computations from which this can be derived.

There is an identification of the local equivariant index formulae for the Dolbeault complexes twisted by $T^{1,0}X \otimes \mathcal{K}^{-1}$ and $T^{0,1}X$ on $\mathbb{C}^2$ (this also corresponds to complex conjugation of the terms in the numerator), and for a smooth complex surface, the local equivariant index for the Dolbeault complex twisted by $T^{0,1}X \otimes \mathcal{K}^{1/2}$ is the same as the local equivariant index for the Dolbeault complex twisted by $T^{1,0}X \otimes \mathcal{K}^{-1/2}$. Since the space is an orbifold and a normal variety, we can proceed to use known results for orbifolds (see \cite{vergne1994quantification,meinrenken1998symplectic}) to extend the identification to the singularity. We can compute the global Lefschetz number for this complex as
\begin{equation}
\label{equation_nodal_RS_2}
    \frac{(\lambda^2+\mu^2+\lambda\mu+\lambda\mu) (\lambda \mu)^{-1/2}}{(1-\lambda^2)(1-\mu^2)} +\frac{(\mu/\lambda+ 1/\lambda^2) (\mu/\lambda^3)^{-1/2}}{(1-\mu/\lambda)(1-1/\lambda^2)}+\frac{(\lambda/\mu+1/\mu^2) (\lambda/\mu^3)^{-1/2}}{(1-\lambda/\mu)(1-1/\mu^2)} = (\lambda \mu)^{-1/2}
\end{equation}
and again we refer to Section 7.3.3 of \cite{jayasinghe2023l2}.

In \cite{HommaSemmelmannRaritaSchwinger2019}, the authors say that ``it is interesting to see to which extent Weitzenb\"ock formulas for the Rarita-Schwinger operator can be applied and to find examples
of compact spin manifolds of positive scalar curvature admitting non-trivial Rarita-Schwinger fields". Recall that there are isometric K\"ahler Hamiltonian circle actions on smooth manifolds only if they are Fano (a consequence of the fact that given such actions, there are no non-trivial global sections of the canonical bundle which can be concluded by the holomorphic Morse inequalities as observed in \cite{witten1984holomorphic}).
Since the Morse inequalities give an alternative to the B\"ochner technique in the presence of group actions, one could try to conduct a preliminary investigation on answering the question above. However since the Rarita-Schwinger operator is only identified up to lower order terms in the study of supergravity, it is not clear whether counting the elements in the kernel of the twisted Dolbeault-Dirac operators correspond to counting the Rarita-Schwinger fields even though the indices are equal.

The question of counting elements in the cohomology of the twisted Dolbeault complex is nevertheless mathematically fascinating.
Let us consider Morse inequalities for the twisted Dolbeault complex corresponding to the $k$-RS operator on a smooth K\"ahler manifold of complex dimension $n$.
The local Morse polynomial at a critical point is given by equation \eqref{equation_Naki_Mayina}, where 
\begin{equation}
    E_a(\theta)= \bigg(\prod_{j=1}^{n} \lambda_{a,j}^{1/2} \bigg) \bigg( \sum_{j=1}^n (\lambda_{a,j}+\lambda_{a,j}^{-1} +k) \bigg)
\end{equation}
where $\lambda_{a,j}=e^{i \gamma_j^a \theta}$,
and we can rewrite \eqref{equation_Naki_Mayina} for this complex as 
\begin{multline}
\label{equation_smooth_kRS_Hol_Morse}
    b^{n_a} \bigg( \prod_{\gamma_j^a >0} \frac{e^{i |\gamma_j^a| \theta/2}}{1-e^{i |\gamma_j^a| \theta}} \prod_{\gamma_j^a <0} \frac{e^{i |\gamma_j^a| \theta/2}}{1-e^{i |\gamma_j^a| \theta}} \bigg) \bigg( \sum_{m=1}^n (e^{i |\gamma_m^a| \theta}+e^{-i |\gamma_m^a| \theta} +k) \bigg)\\=b^{n_a} \bigg( \prod_{\gamma_j^a} \frac{e^{i |\gamma_j^a| \theta/2}}{1-e^{i |\gamma_j^a| \theta}} \bigg) \bigg( \sum_{m=1}^n (e^{i |\gamma_m^a| \theta}+e^{-i |\gamma_m^a| \theta} +k) \bigg).
\end{multline}
The dual Morse inequalities are
\begin{multline}
     b^{n-n_a} \bigg( \prod_{\gamma_j^a <0} \frac{e^{-i |\gamma_j^a| \theta/2}}{1-e^{-i |\gamma_j^a| \theta}} \prod_{\gamma_j^a >0} \frac{e^{-i |\gamma_j^a| \theta/2}}{1-e^{-i |\gamma_j^a| \theta}} \bigg) \bigg( \sum_{m=1}^n (e^{i |\gamma_m^a| \theta}+e^{-i |\gamma_m^a| \theta} +k) \bigg)\\=b^{n-n_a} \bigg( \prod_{\gamma_j^a} \frac{e^{-i |\gamma_j^a| \theta/2}}{1-e^{-i |\gamma_j^a| \theta}} \bigg) \bigg( \sum_{m=1}^n (e^{i |\gamma_m^a| \theta}+e^{-i |\gamma_m^a| \theta} +k) \bigg).
\end{multline}
It is easy to observe that the $k$-RS operator satisfies \textbf{self duality under Serre duality} in the smooth setting  which has important consequences as highlighted in Remark \ref{Remark_Self_duality_symmetries}.

Let us consider the expository example of spinning the Riemann sphere in Example \ref{Example_spinning_sphere_intro}. The Morse polynomial for the twisted Dolbeault complex corresponding to the $k$-RS operator for the action in that example is
\begin{equation}
    \frac{\lambda^{3/2}+k \lambda^{1/2} + \lambda^{-1/2}}{(1-\lambda)}+b\frac{\lambda^{3/2}+k \lambda^{1/2} + \lambda^{-1/2}}{(1-\lambda)}
\end{equation}
and the dual Morse polynomial is
\begin{equation}
    b \frac{\lambda^{3/2}+k \lambda^{1/2} + \lambda^{-1/2}}{(1-\lambda^{-1})}+\frac{\lambda^{3/2}+k \lambda^{1/2} + \lambda^{-1/2}}{(1-\lambda^{-1})}.
\end{equation}
Then the classical Morse polynomial is 
\begin{equation}
    \lambda^{-1/2}(1+b)
\end{equation}
which is easy to see from just the fact that the Morse polynomial has only one term with a negative power of $\lambda$ in its supertrace expansion, and self duality under Serre duality (see Remark \ref{Remark_Self_duality_symmetries}).
In particular it implies that there can only be two elements in the null space of the corresponding twisted Dirac operator, though it does not guarantee that two such elements exist.

In the general case, the contribution to the local classical Morse polynomial from a given critical point can be figured out using the local Morse polynomial and the self duality under Serre duality. Studying the exponents of the terms in the expansion of \eqref{equation_smooth_kRS_Hol_Morse}, it is easy to obtain an upper bound for the contribution to the global cohomology from a critical point as the sum over $m$ of the quantity $L_m$, where $L_m$ is the number of tuples $(l_1,..., l_n)$ with non-negative integers $l_j$ that solve the inequality
\begin{equation}
    (\sum_{j=1}^n |\gamma^a_j|)-2 |\gamma^a_m|+ \sum_{j=1}^n l_j |\gamma^a_j| \leq 0,
\end{equation}
for each fixed $m$.

\bibliographystyle{alpha}
\bibliography{reference}

\end{document}